\documentclass[11pt]{article}
\usepackage{amsmath,amssymb,amsthm,xcolor,graphicx,bbm}
\usepackage{accents}
\usepackage[unicode,breaklinks=true,colorlinks=true,citecolor = {magenta}]{hyperref}
\usepackage{esint}  %
\usepackage{listings}  %

\usepackage{authblk}

\usepackage[top=1in, bottom=1in, left=1in, right=1in]{geometry}

\usepackage{thmtools} %

\DeclareSymbolFont{yhlargesymbols}{OMX}{yhex}{m}{n}
\DeclareMathAccent{\wideparen}{\mathord}{yhlargesymbols}{"F3}

\numberwithin{equation}{section}

\newtheorem{theorem}{Theorem}[section]
\newtheorem{corollary}[theorem]{Corollary}
\newtheorem{lemma}[theorem]{Lemma}

\newtheorem{definition}[theorem]{Definition} 

\theoremstyle{remark}

\newcommand{\bke}[1]{\left ( #1 \right )}
\newcommand{\bkt}[1]{\left [ #1 \right ]}
\newcommand{\bket}[1]{\left \{ #1 \right \}}
\newcommand{\norm}[1]{\left  \| #1  \right \|}

\newcommand{\abs}[1]{\left | #1 \right |}

\newcommand\al{\alpha}
\newcommand\be{\beta}
\newcommand\ga{\gamma}
\newcommand\de{\delta}
\newcommand\ep{\epsilon}
\newcommand\ve{\varepsilon}
\newcommand\e {\varepsilon}

\newcommand\ka{\kappa}
\newcommand\la{\lambda}
\newcommand\si{\sigma}

\newcommand\De{\Delta}

\newcommand{\R}{\mathbb{R}}

\newcommand{\ZZ}{\mathbb{Z}}
\newcommand{\NN}{\mathbb{N}}

\renewcommand{\div}{\mathop{\rm div}\nolimits}

\newcommand{\supp} {\mathop{\mathrm{supp}}\nolimits}

\newcommand{\esssup} {\mathop{\rm ess\,sup}}

\newcommand{\cM}{\mathcal{M}}

\newcommand{\pd}{\partial}
\newcommand{\nb}{\nabla}

\newcommand{\td}{\tilde}
\newcommand{\wt}[1]{\widetilde {#1}}

\newcommand{\lec}{{\ \lesssim \ }}

\newcommand{\I}{\infty}
\newcommand{\oo}{\infty}

\newcommand{\Eq}[1]{\begin{equation*}#1\end{equation*}}
\newcommand{\EQ}[1]{\begin{equation}#1\end{equation}}
\newcommand{\EQS}[1]{\begin{equation}\begin{split} #1 \end{split}\end{equation}}
\newcommand{\EQN}[1]{\begin{equation*}\begin{split} #1 \end{split}\end{equation*}}
\newcommand{\EN}[1]{\begin{enumerate} #1 \end{enumerate}}

\usepackage{mathtools} %

\newcommand{\loc}{\mathrm{loc}}
\newcommand{\uloc}{\mathrm{uloc}}
\newcommand{\one}{\mathbbm{1}}
 
\begin{document}
\title{Weak and mild solutions to 
the MHD equations and the viscoelastic Navier--Stokes equations with damping
in Wiener amalgam spaces}

\author{
 Chen-Chih Lai \thanks{Email address: cclai.math@gmail.com}
 }
 \affil{\footnotesize Department of Mathematics, Columbia University, New York, NY 10027, USA}
\date{}
\maketitle

\vspace{-0.6cm}

\begin{abstract}
We study the three-dimensional incompressible magnetohydrodynamic (MHD) equations and the incompressible viscoelastic Navier--Stokes equations with damping. 
Building on techniques developed by Bradshaw, Lai, and Tsai (Math. Ann. 2024), we prove the existence of mild solutions in Wiener amalgam spaces that satisfy the corresponding spacetime integral bounds. 
In addition, we construct global-in-time local energy weak solutions in these amalgam spaces using the framework introduced by Bradshaw and Tsai (SIAM J. Math. Anal. 2021).
As part of this construction, we also establish several properties of local energy solutions with $L^2_\uloc$ initial data, including initial and eventual regularity as well as small-large uniqueness, extending analogous results obtained for the Navier--Stokes equations by Bradshaw and Tsai (Comm. Partial Differential Equations 2020).

\end{abstract}

\vspace{-0.4cm}

\renewcommand{\baselinestretch}{0.85}\normalsize
\tableofcontents
\renewcommand{\baselinestretch}{1.0}\normalsize

\section{Introduction}

The incompressible magnetohydrodynamic (MHD) equations describe the interaction of a fluid's velocity field and a magnetic field within a conducting medium, coupling the incompressible Navier--Stokes equations with Maxwell's equations of electromagnetism. These fundamental equations are given by
 \begin{equation}\label{MHD}\tag{MHD}
\setlength\arraycolsep{1.5pt}\def\arraystretch{1.2}
\left.\begin{array}{ll}
\pd_tv-\De v+v\cdot\nb v-b\cdot\nb b+\nb\pi&=0\ \\
\pd_tb-\De b+v\cdot\nb b-b\cdot\nb v&=0\ \\  
~~~~~~~~~~~~~~~~~~~~~~~~~~~~~~\,\nb\cdot v =\nb\cdot b&=0\  
\end{array}\right\} \text{ in }\R^3\times(0,\infty),
\end{equation}
where $v$ is the velocity, $b$ the magnetic field, and $\pi$ the pressure.
The study of MHD equations has attracted considerable attention.
The foundational result of Duvaut and Lions \cite{DL-ARMA1972} established the global existence of weak solutions with finite energy.
Building upon this, Sermange and Temam \cite{ST-CPAM1983} investigated regularity criteria for weak solutions.
Subsequent effort refined these regularity conditions under various assumptions.
For instance, Wu \cite{Wu-DCDS2004} and Zhou \cite{Zhou-DCDS2005} established Serrin-type criteria and scaling-invariant regularity conditions, while He and Xin \cite{HX-JFA2005} and Kang and Lee \cite{Kang-Lee-JDE2009} developed partial regularity results for suitable weak solutions.
Further improvements were made via harmonic analysis methods, as in \cite{CMZ-CMP2008}, and directionally-restricted criteria, such as those by Cao and Wu \cite{CW-JDE2010}.
More recently, local regularity theory for MHD has been advanced in parabolic Morrey spaces \cite{CCHJ-DocumentaMath2021}, and Fern\'andez-Dalgo and Jarr\'in \cite{FJ-JDE2021} provided weak-strong uniqueness results in weighted $L^2$ spaces, alongside constructions of weak suitable solutions in local Morrey spaces.

On the existence side, Miao, Yuan, and Zhang \cite{MYZ-MMAS2007} proved global mild solutions for small data in $BMO^{-1}$, and He and Xin \cite{HX-AMSSB2009} constructed self-similar solutions under small homogeneous initial data. 
Moreover, the existence of forward discretely self-similar and self-similar local Leray solutions is established in the critical space $L^{3,\infty}$ \cite{Lai-JMFM2019} and in the weighted $L^2$ spaces \cite{FJ-JMFM2021}.
 The criticality of the $L^{3,\infty}$ class was also highlighted in \cite{MNS-JMS2007}, where global regularity of weak solutions was established under this condition.
Additional contributions include the construction of global smooth solutions under spectral constraints \cite{LZZ-JDE2016}, the use of Morrey spaces to ensure global well-posedness for small data \cite{LXZZ-JDE2022}, and the construction of forward self-similar solutions via topological and blow-up methods \cite{Yang-arXiv2024}.

 Complementing the MHD system, the incompressible viscoelastic Navier--Stokes equations with damping (vNSEd) model non-Newtonian fluids with both viscous and elastic characteristics. In the simplified setting where both relaxation and retardation times are infinite, the vNSEd system reads
\begin{equation}\label{vNSE}
\setlength\arraycolsep{1.5pt}\def\arraystretch{1.2}
\left.\begin{array}{ll}
\pd_tv-\De v+v\cdot\nb v-\nb\cdot({\bf F}{\bf F}^\top)+\nb\pi&=0\ \\
\pd_t{\bf F}+v\cdot\nb {\bf F}-(\nb v){\bf F}&=0\ \\  
~~~~~~~~~~~~~~~~~~~~~~~~~~~~~~~~~~~~~~~~~~~~~~~\nb\cdot v &=0\  
\end{array}\right\} \text{ in }\R^3\times(0,\infty),
\end{equation}
with initial data \[v|_{t=0}=v_0\ \text{ and }\ {\bf F}|_{t=0}={\bf F}_0\ \text{ in }\R^3,\]
where $v$ is the velocity field, ${\bf F}$ is the local deformation tensor of the fluid, and $\pi$ is the pressure. 
This model arises from Oldroyd-type theories for viscoelastic fluids and captures the interplay between fluid motion and elastic stresses. 
The addition of a damping term in the equation for ${\bf F}$ (following Lin--Liu--Zhang \cite{LLZ-CPAM2005}) is critical for obtaining global solutions, particularly in the absence of intrinsic dissipative mechanisms. 
To be more precise, they introduced the following viscoelastic Navier-Stokes equations with damping as a way to approximate solutions of \eqref{vNSE}:
\begin{equation}\label{vNSEd0}
\setlength\arraycolsep{1.5pt}\def\arraystretch{1.2}
\left.\begin{array}{ll}
\pd_tv-\De v+v\cdot\nb v-\nb\cdot({\bf F}{\bf F}^\top)+\nb\pi&=0\ \\
\pd_t{\bf F}-\mu\De {\bf F}+v\cdot\nb{\bf F}-(\nb v){\bf F}&=0\ \\  
~~~~~~~~~~~~~~~~~~~~~~~~~~~~~~~~~~~~~~~~~~~~\nb\cdot v &=0\  
\end{array}\right\} \text{ in }\R^3\times(0,\infty),
\end{equation}
for a damping parameter $\mu>0$.
Existence results for smooth solutions under smallness conditions or specific symmetries have been established by various authors \cite{CZ-CPDE2006,LLZ-ARMA2008,LLZ-CPAM2005}.
Note that if $\nb\cdot{\bf F}=0$ at some instance of time, then $\nb\cdot{\bf F}=0$ at all later times. In fact, by taking divergence of $\eqref{vNSEd0}_2$ and using $\eqref{vNSEd0}_3$, one have the following equation for $\nb\cdot {\bf F}$:
\[\pd_t(\nb\cdot{\bf F})+v\cdot\nb(\nb\cdot{\bf F})=\mu\De(\nb\cdot{\bf F}).\]
 Hence it is natural to assume 
\[
\nb\cdot{\bf F}=0.
\]
The authors \cite{LLZ-CPAM2005} noted that using standard weak convergence methods to pass the limit of solutions to \eqref{vNSEd0} as $\mu\to0^+$ does not yield weak solutions of \eqref{vNSE}.
Despite this, system \eqref{vNSEd0} remains an interesting subject of study.
For instance, Lai, Lin, and Wang \cite{LLW-SIMA2017} established the existence of forward self-similar classical solution to \eqref{vNSEd0} for locally H\"{o}lder continuous, $(-1)$-homogeneous initial data. 
Additionally, the existence of forward discretely self-similar and self-similar local Leray solutions in the critical space $L^{3,\infty}$ is established in \cite{Lai-JMFM2019}, following the analysis in \cite{BT-AHP2017}.

Regularity issues for weak solutions of the viscoelastic Navier--Stokes equations with damping have been investigated from several perspectives. 
Hynd \cite{Hynd-SIMA2013} proved a version of the Caffarelli--Kohn--Nirenberg partial regularity theorem adapted to the viscoelastic system with damping, while Kim \cite{Kim-AML2017} established Serrin-type regularity criteria in weak-$L^p$ spaces. 
These results have been further extended in \cite{TWZ-CMS2020}, which proved global existence of mild solutions in scaling-invariant spaces for small data and derived various regularity criteria in Lorentz, multiplier, BMO, and Besov spaces. 
Additional contributions include the construction of global classical solutions with symmetry assumptions in periodic domains by Liu and Lin \cite{LL-JIA2021}, and refined local energy bounds leading to improved $\ep$-regularity conditions in the sense of Caffarelli--Kohn--Nirenberg in \cite{YPW-ACV2025}.

Since the damping parameter $\mu$ does not affect our analysis, we set $\mu = 1$ throughout this paper. 
Then, columnwisely, \eqref{vNSEd0} can be rewritten as
\begin{equation}\label{vNSEd}\tag{vNSEd}
\setlength\arraycolsep{1.5pt}\def\arraystretch{1.2}
\left.\begin{array}{ll}
\pd_tv-\De v+v\cdot\nb v-\underset{n=1}{\overset{3}\sum} f_n\cdot\nb f_n+\nb\pi&=0\ \\
\pd_tf_m-\De f_m+v\cdot\nb f_m- f_m\cdot\nb v&=0\ \\  
~~~~~~~~~~~~~~~~~~~~~~~~~~~~~~~~~~~~\nb\cdot f_m=\nb\cdot v &=0\  
\end{array}\right\} \text{ in }\R^3\times(0,\infty),\ m=1,2,3,
\end{equation}
where $f_m$ is the $m$-th column vector of ${\bf F}$.

A central theme in the analysis of both the \eqref{MHD} and \eqref{vNSEd} systems is the interplay between the nonlinear couplings, scaling symmetries, and the functional framework chosen for solutions. While much progress has been made in classical Lebesgue, Sobolev, and Besov spaces, recent advances have highlighted the utility of Wiener amalgam spaces in studying fluid systems. 
In this paper, the Wiener amalgam spaces are denoted $E^p_q$ and defined by the norm  
\[
\|  f \|_{E^p_q} :=\bigg\|   \norm{f}_{L^p(B_1(k))} \bigg\|_{\ell^q(k\in {\ZZ^3})}<\I. 
\]
These spaces, which blend local integrability and global decay properties, provide a flexible setting that accommodates non-decaying or large initial data while retaining control over both local and global behaviors.
We identify $E^p_\I$ with $L^p_\uloc$ with the norm $\norm{f}_{L^p_\uloc} := \sup_{x_0\in\R^3} \norm{f}_{L^p(B_1(x_0))}$.
The closure of $C^\infty_c(\R^3)$ under the $L^p_\uloc$ norm is denoted by $E^p$.
Note that for $p,p_1,p_2,q,q_1,q_2 \in [1,\infty]$ we have the H\"older inequality:
\EQ{\label{Holder}
\norm{fg}_{E^p_q} \le  \norm{f}_{E^{p_1}_{q_1}}  \norm{g}_{E^{p_2}_{q_2}}, \quad \frac 1p= \frac 1{p_1}+\frac1{p_2},\quad
\frac 1q\le \frac 1{q_1}+\frac1{q_2}.
}

We will consider two kinds of spacetime integrals: For $0<T\le \I$, $x\in \R^3$, and $1\le s,p,q\le\infty$, define the norms $L^s_T E^p_q$ and $E^{s,p}_{T,q}$ as follows:
\EQ{\label{LsEpq-def}
\| f\|_{L^s_T E^p_q} := \|  f \|_{L^s(0,T; E^p_q(\R^3))}, 
}
and  
\EQ{\label{Espq-def}
\norm{f}_{E^{s,p}_{T,q}}:=
\norm{\norm{ f}_{L^s_TL^p(B_1(k))}}_{\ell^q(k\in \ZZ^3)}.
}
These norms are different from each other when $s\neq q$.
By Minkowski's integral inequality,
\EQ{\label{ineq:embedding.parabolic.space}
\norm{f}_{L^s _T E^p_q} \le \norm{f}_{E^{s,p}_{T,q}},
\qquad \text{ if }q\le s,
}
and
\EQ{\label{ineq:embedding.parabolic.space2}
\norm{f}_{E^{s,p}_{T,q}} \le \norm{f}_{L^s _T E^p_q},
\qquad \text{ if }q\ge s.
}

Previous works \cite{BT-SIMA2021, BLT-MathAnn2024} developed a detailed theory for the incompressible Navier--Stokes equations in Wiener amalgam spaces, establishing mild and weak solutions, spacetime integral bounds, and eventual regularity results for different ranges of the Lebesgue exponent $q$. In this paper, we extend these techniques to the \eqref{MHD} and \eqref{vNSEd} systems. Specifically, we prove the existence of mild solutions for small data in critical and subcritical Wiener amalgam spaces, as well as global weak solutions under appropriate integrability and decay conditions. Our analysis demonstrates the robustness of the Wiener amalgam framework in addressing the intricate coupling structures and nonlinearities in these models.

In the following subsections, we introduce the definitions of mild and local energy solutions for the systems \eqref{MHD} and \eqref{vNSEd}, and present our main results.

\subsection{Mild solutions of MHD equations}

A pair of vector fields $(v, b)$ is called a \emph{mild solution} to \eqref{MHD} if it satisfies
\EQ{\label{eq-mild-mhd}
(v,b)(x,t) = (e^{t\De}v_0, e^{t\De}b_0) - B((v,b),(v,b))(t),
}
where $B$ is a bilinear operator defined by $B = (B_1,B_2)$,
\EQS{\label{eq-bilinear-operator}
B_1((v,b),(u,a))(t) = \int_0^t  e^{(t-s)\De}\mathbb{P}\nb\cdot(v\otimes u - b\otimes a)\, ds,\\
B_2((v,b),(u,a))(t) = \int_0^t  e^{(t-s)\De}\mathbb{P}\nb\cdot(v\otimes a - b\otimes u)\, ds,
}
in which $\mathbb{P}$ is the Helmholtz projection operator.
More precisely, the vector components of the bilinear operators $B_1$ and $B_2$ can be expressed by
\EQN{
B_1((v,b),(u,a))_i(x,t) = \int_0^t \int_{\R^d} \pd_l S_{ij}(x-y,t-s) (v_lu_j - b_la_j)(y,s)\, dyds,\quad i=1,2,3,\\
B_2((v,b),(u,a))_i(x,t) = \int_0^t \int_{\R^d} \pd_l S_{ij}(x-y,t-s) (v_la_j - b_lu_j)(y,s)\, dyds,\quad i=1,2,3,
}
where $S_{ij}$ is the Ossen tensor derived by Oseen in \cite{Oseen-book1927}.
We refer the readers to \cite[Section 2.2]{KLLT-CMP2023} for a brief introduction of the Oseen tensor.

We first consider the case of data $(v_0,b_0)\in E^r_q\times E^r_q$ with $r>3$, which we refer to as \emph{subcritical}, and state the existence of mild solutions in the amalgam spaces in the following theorem.

\begin{theorem}[Subcritical data (MHD)]\label{thrm.subcritical-mhd}
Let $r\in (3,\I]$  and $q\in[1,\infty]$. If $v_0, b_0\in E^r_q$ are divergence free, then, for any  positive time $T=T(\|(v_0,b_0)\|_{E^r_q\times E^r_q})$ chosen so that
\EQ{ \label{def:T.subcrit-mhd}    T^{1/2 - 3/(2r)} +  T^{1/2 }	  
\lec \|(v_0, b_0)\|_{E^r_q\times E^r_q} ^{-1},
} there exists a unique mild solution $(v, b)\in L^\infty(0,T;E^r_q\times E^r_q) \cap C((0,T);E^r_q\times E^r_q)$ to \eqref{MHD}. Moreover, $(v,b)$ satisfies
\EQ{ \label{th1.1eq2-mhd}
\sup_{0\leq t\leq T} \| (v,b)(t)\|_{E^r_q\times E^r_q} \leq C \|(v_0, b_0)\|_{E^r_q\times E^r_q} . 
}

If $q,r<\I$, then $(v,b)\in C([0,T];E^r_q\times E^r_q)$. If $q=\I$ or $r=\I$, then we still have $\|(e^{t\Delta}v_0- {v(t)}, e^{t\Delta}b_0- {b(t)})\|_{E^r_q\times E^r_q} \to 0$ as $t\to 0^+$. 

Furthermore, if $r<\infty$, then for any $s \in [r,\infty]$ and $p\in[r,3r]$ 
with $\frac2s + \frac3p = \frac3r$,
\[
\norm{(v,b)}_{ E^{s,p}_{T,m}\times E^{s,p}_{T,m}} \le C \|(v_0, b_0)\|_{E^r_q\times E^r_q},\qquad m\ge q,
\]
provided $(1+T^{\frac1s+\epsilon})(T^{\frac12-\frac3{2r}} + T^{1-\frac1s}) \lec \norm{(v_0,b_0)}_{E^r_q\times E^r_q}^{-1}$ for all $\epsilon>0$. 
\end{theorem}

Theorem \ref{thrm.subcritical-mhd} is proved in Section \ref{sec-subcritical}.

We now turn to the \emph{critical} case, i.e., the case when data $(v_0,b_0)\in E^3_q\times E^3_q$.
When the data is sufficiently small, we have the following existence theorem of mild solutions.

\begin{theorem}[Critical data I (MHD)]\label{thrm:critical-mhd} Let $q\in[1,\I]$. Fix $T>0$.
There exists $\ve=\ve(T)>0$ such that for all {divergence-free} $v_0,b_0\in E^3_q$ with $\norm{(v_0, b_0)}_{E^3_q\times E^3_q}\le\ve$, there exists a mild solution $(v,b)$  to \eqref{MHD} with 
\[
(v,b) \in L^\infty(0,T; E^3_q\times E^3_q)\quad \text{and}\quad t^{\frac12}(v, b)\in L^\infty (0,T; E^\oo_q\times E^\oo_q).
\]
The solution is unique in the class 
\EQ{\label{eq-uniquesmall-mhd}
\sup_{0<t<T} t^{\frac14} \norm{(v,b)}_{E^6_q\times E^6_q} \le 2 \sup_{0<t<T} t^{\frac14} \norm{(e^{t\De} v_0, e^{t\De} b_0)}_{E^6_q\times E^6_q}.
}
Furthermore, $\norm{(v,b)}_{L^\infty_T (E^3_q\times E^3_q)} + \norm{t^{1/2}(v,b)}_{L^\infty_T (E^\oo_q\times E^\oo_q)} \lec \norm{(v_0,b_0)}_{E^3_q\times E^3_q} $. We have $(v,b)\in C((0,T); E^3_q\times E^3_q) $ for $q=\I$ and $(v,b)\in C([0,T); E^3_q\times E^3_q)$ for $q<\I$. If $q=\I$, then we have for any ball $B$ and $\delta\in (0,2]$ that 
\EQ{\label{convergencetodata-mhd}
\lim_{t\to 0^+} \|  (v,b) (t) - (v_0,b_0)\|_{L^{3-\delta}(B)\times L^{3-\delta}(B)}=0.
}

For any $s\in [3,\infty]$ and $p\in [3,9]$ given by $\frac2s + \frac3p = 1$, by taking $\ve \le \ve_0(T,s)$ sufficiently small, this solution further satisfies%
\EQ{\label{eq1.7-mhd}
\norm{(v,b)}_{E^{s,p}_{T,m}\times E^{s,p}_{T,m}} + \one_{q \le s}  \norm{(v,b)}_{L^s_T (E^p_m\times E^p_m)} \le C \norm{(v_0,b_0)}_{E^3_q\times E^3_q},\quad \forall m \in [q,\infty].
}
\end{theorem}

The following theorem concerns the critical case with enough decay, $1\le q\le3$.

\begin{theorem}[Critical data II (MHD)]\label{thrm:critical2-mhd}
Let $1\le q\leq 3$. 
For all divergence-free $v_0,b_0\in E^3_q$, there exist $T=T(v_0,b_0)>0$ and 
a unique mild solution $(v,b)$  to \eqref{MHD} satisfying
\[
(v,b) \in BC([0,T); E^3_q\times E^3_q)\quad \text{and}\quad t^{\frac12} (v, b) \in L^\I (0,T; E^\oo_{q_2}\times E^\oo_{q_2}),
\]
with $1/q_2 = 1/q-1/3$, $q_2 \in [\frac 32,\infty]$. For any 
$s\in [3,\infty)$, $\frac2s + \frac3p = 1$, and $m \in [q,\infty]$,
there is $T_1 \in (0,T]$ such that 
\EQ{\label{eq-thmIII-Ebound-mhd}
(v,b) \in E^{s,p}_{T_1,m}\times E^{s,p}_{T_1,m}.
}
Furthermore, there is $\e(q)>0$ such that $T=\oo$ if $\norm{(v_0,b_0)}_{E^3_q\times E^3_q} \le \e(q)$.
If we assume further%
\EQ{\label{eq-thmIII-mbound-mhd}
m > p' = \frac p{p-1},\quad \text{and}\quad
m \ge m_1,\quad
\frac 2 s +\frac 3 {m_1} = \frac 3 q, 
}
with $m>m_1(s,q)$ when $q=1$,
then there exists $\e_1(s,q,m)>0$ such that $T_1=\oo$ if $\norm{(v_0,b_0)}_{E^3_q\times E^3_q} \le \e_1(s,q,m)$.
Instead of \eqref{eq-thmIII-mbound-mhd}, %
if we assume
\EQ{
m \ge \max(p',m_1), \quad \text{and}\quad \left\{
\begin{aligned}
m>m_1&\quad \text{if }\quad q=1,\\
m \ge p &\quad \text{if }\quad 3s<5q,
\end{aligned}\right.
}
then there exists $\e_2(s,q,m)>0$ such that $v,b \in L^s_{T=\infty} E^p_m$ if $\norm{(v_0,b_0)}_{E^3_q\times E^3_q} \le \e_2(s,q,m)$.
\end{theorem}

We prove Theorem \ref{thrm:critical2-mhd} in Section \ref{sec-critical-II}.

\subsection{Mild solutions of viscoelastic Navier--Stokes equations with damping}

A pair $(v, {\bf F})$, ${\bf F} = [f_1, f_2, f_3]\in\R^{3\times 3}$, is called a \emph{mild solution} to \eqref{vNSEd} if it satisfies
\EQS{\label{eq-mild-vnsed}
v(x,t) &= e^{t\De}v_0 - \mathcal{B}_0((v,{\bf F}),(v,{\bf F}))(t),\\
f_m(x,t) &= e^{t\De}(f_m)_0 - \mathcal{B}_m((v,{\bf F}),(v,{\bf F}))(t),\quad m=1,2,3,
}
where
\EQN{
\mathcal{B}_0((v,{\bf F}),(u,{\bf G}))(t) &= \int_0^t  e^{(t-s)\De}\mathbb{P}\nb\cdot\bke{v\otimes u - \sum_{n=1}^3 f_n\otimes g_n}\, ds,\quad {\bf G} = [g_1, g_2, g_3],\\
\mathcal{B}_m((v,{\bf F}),(u,{\bf G}))(t) &= \int_0^t  e^{(t-s)\De}\mathbb{P}\nb\cdot(v\otimes g_m - f_m\otimes u)\, ds.
}

The main results of the mild solutions for the viscoelastic Navier--Stokes equations with damping are stated as follows:

\begin{theorem}[Subcritical data (vNSEd)]\label{thrm.subcritical-vnsed}
Let $r\in (3,\I]$  and $q\in[1,\infty]$. If $(v_0, {\bf F}_0)\in E^r_q\times {\bf E}^r_q$, ${\bf F}_0=[(f_1)_0,(f_2)_0,(f_3)_0]$, where $v_0$ and $(f_m)_0$, $m=1,2,3$, are divergence free, then, for any  positive time $T=T(\|(v_0,{\bf F}_0)\|_{E^r_q\times {\bf E}^r_q})$ chosen so that
\EQ{ \label{def:T.subcrit-mhd}    T^{1/2 - 3/(2r)} +  T^{1/2 }	  
\lec \|(v_0, {\bf F}_0)\|_{E^r_q\times {\bf E}^r_q} ^{-1},
} there exists a unique mild solution $(v, {\bf F})\in L^\infty(0,T;E^r_q\times {\bf E}^r_q) \cap C((0,T);E^r_q\times {\bf E}^r_q)$ to \eqref{vNSEd}. Moreover, $(v,{\bf F})$ satisfies
\EQ{ \label{th1.1eq2-mhd}
\sup_{0\leq t\leq T} \| (v,{\bf F})(t)\|_{E^r_q\times {\bf E}^r_q} \leq C \|(v_0, {\bf F}_0)\|_{E^r_q\times {\bf E}^r_q} . 
}

If $q,r<\I$, then $(v,{\bf F})\in C([0,T];E^r_q\times {\bf E}^r_q)$. If $q=\I$ or $r=\I$, then we still have $\|(e^{t\Delta}v_0- {v(t)}, e^{t\Delta}{\bf F}_0- {{\bf F}(t)})\|_{E^r_q\times {\bf E}^r_q} \to 0$ as $t\to 0^+$. 

Furthermore, if $r<\infty$, then for any $s \in [r,\infty]$ and $p\in[r,3r]$ 
with $\frac2s + \frac3p = \frac3r$,
\[
\norm{(v,{\bf F})}_{ E^{s,p}_{T,m}\times {\bf E}^{s,p}_{T,m}} \le C \|(v_0, {\bf F}_0)\|_{E^r_q\times {\bf E}^r_q},\qquad m\ge q,
\]
provided $(1+T^{\frac1s+\epsilon})(T^{\frac12-\frac3{2r}} + T^{1-\frac1s}) \lec \norm{(v_0,{\bf F}_0)}_{E^r_q\times {\bf E}^r_q}^{-1}$ for all $\epsilon>0$. 
\end{theorem}

\begin{theorem}[Critical data I (vNSEd)]\label{thrm:critical-vnsed} 
Let $q\in[1,\I]$. Fix $T>0$.
There exists $\ve=\ve(T)>0$ such that for all {divergence-free} $v_0,(f_m)_0\in E^3_q$, $m=1,2,3$, with $\norm{(v_0, {\bf F}_0)}_{E^3_q\times {\bf E}^3_q}\le\ve$, ${\bf F}_0 = [(f_1)_0,(f_2)_0,(f_3)_0]$, there exists a mild solution $(v,{\bf F})$  to \eqref{vNSEd} with 
\[
(v,{\bf F}) \in L^\infty(0,T; E^3_q\times {\bf E}^3_q)\quad \text{and}\quad t^{\frac12}(v, {\bf F})\in L^\infty (0,T; E^\oo_q\times {\bf E}^\oo_q).
\]
The solution is unique in the class 
\EQ{\label{eq-uniquesmall-mhd}
\sup_{0<t<T} t^{\frac14} \norm{(v,{\bf F})}_{E^6_q\times {\bf E}^6_q} \le 2 \sup_{0<t<T} t^{\frac14} \norm{(e^{t\De} v_0, e^{t\De} {\bf F}_0)}_{E^6_q\times {\bf E}^6_q}.
}
Furthermore, $\norm{(v, {\bf F})}_{L^\infty_T (E^3_q\times {\bf E}^3_q)} + \norm{t^{1/2}(v,{\bf F})}_{L^\infty_T (E^\oo_q\times {\bf E}^\oo_q)} \lec \norm{(v_0,{\bf F}_0)}_{E^3_q\times {\bf E}^3_q} $. We have $(v,{\bf F})\in C((0,T); E^3_q\times {\bf E}^3_q) $ for $q=\I$ and $(v,{\bf F})\in C([0,T); E^3_q\times {\bf E}^3_q)$ for $q<\I$. If $q=\I$, then we have for any ball $B$ and $\delta\in (0,2]$ that 
\EQ{\label{convergencetodata-mhd}
\lim_{t\to 0^+} \|  (v,{\bf F}) (t) - (v_0,{\bf F}_0)\|_{L^{3-\delta}(B)\times {\bf L}^{3-\delta}(B)}=0.
}

For any $s\in [3,\infty]$ and $p\in [3,9]$ given by $\frac2s + \frac3p = 1$, by taking $\ve \le \ve_0(T,s)$ sufficiently small, this solution further satisfies%
\EQ{\label{eq1.7-mhd}
\norm{(v,{\bf F})}_{E^{s,p}_{T,m}\times {\bf E}^{s,p}_{T,m}} + \one_{q \le s}  \norm{(v,{\bf F})}_{L^s_T (E^p_m\times {\bf E}^p_m)} \le C \norm{(v_0,{\bf F}_0)}_{E^3_q\times {\bf E}^3_q},\quad \forall m \in [q,\infty].
}
\end{theorem}

\begin{theorem}[Critical data II (vNSEd)]\label{thrm:critical2-vnsed}
Let $1\le q\leq 3$. 
For all divergence-free $v_0,(f_1)_0,(f_2)_0,(f_3)_0\in E^3_q$, there exist $T=T(v_0,{\bf F}_0)>0$, ${\bf F}_0 = [(f_1)_0,(f_2)_0,(f_3)_0]$, and 
a unique mild solution $(v,{\bf F})$ to \eqref{vNSEd} satisfying
\[
(v,{\bf F}) \in BC([0,T); E^3_q\times {\bf E}^3_q)\quad \text{and}\quad t^{\frac12} (v, {\bf F}) \in L^\I (0,T; E^\oo_{q_2}\times {\bf E}^\oo_{q_2}),
\]
with $1/q_2 = 1/q-1/3$, $q_2 \in [\frac 32,\infty]$. For any 
$s\in [3,\infty)$, $\frac2s + \frac3p = 1$, and $m \in [q,\infty]$,
there is $T_1 \in (0,T]$ such that 
\EQ{\label{eq-thmIII-Ebound-mhd}
(v,{\bf F}) \in E^{s,p}_{T_1,m}\times {\bf E}^{s,p}_{T_1,m}.
}
Furthermore, there is $\e(q)>0$ such that $T=\oo$ if $\norm{(v_0,{\bf F}_0)}_{E^3_q\times {\bf E}^3_q} \le \e(q)$.
If we assume further%
\EQ{\label{eq-thmIII-mbound-mhd}
m > p' = \frac p{p-1},\quad \text{and}\quad
m \ge m_1,\quad
\frac 2 s +\frac 3 {m_1} = \frac 3 q, 
}
with $m>m_1(s,q)$ when $q=1$,
then there exists $\e_1(s,q,m)>0$ such that $T_1=\oo$ if $\norm{(v_0,{\bf F}_0)}_{E^3_q\times {\bf E}^3_q} \le \e_1(s,q,m)$.
Instead of \eqref{eq-thmIII-mbound-mhd}, %
if we assume
\EQ{
m \ge \max(p',m_1), \quad \text{and}\quad \left\{
\begin{aligned}
m>m_1&\quad \text{if }\quad q=1,\\
m \ge p &\quad \text{if }\quad 3s<5q,
\end{aligned}\right.
}
then there exists $\e_2(s,q,m)>0$ such that $(v,{\bf F}) \in L^s_{T=\infty} (E^p_m\times {\bf E}^p_m)$ if $\norm{(v_0,{\bf F}_0)}_{E^3_q\times {\bf E}^3_q} \le \e_2(s,q,m)$.
\end{theorem}

\subsection{Weak solutions of MHD equations}

We first introduce the notion of local energy solutions for the MHD equations, which is consistent with the concept introduced in \cite{BT-SIMA2021} for the Navier--Stokes equations.

\begin{definition}[local energy solution (MHD)]\label{def-local-energy-sol-mhd}
Let $0<T\le\infty$. A pair of vector fields $(v,b)$, $v,b\in L^2_\loc(\R^3\times[0,T))$, is a local energy solution to \eqref{MHD} with divergence-free initial data $v_0,b_0\in L^2_\uloc(\R^3)$, denoted as $(v,b)\in\mathcal N_{\rm MHD}(v_0,b_0)$, if the following hold:

1. There exists $\pi\in L^{3/2}_\loc(\R^3\times[0,T))$ such that $(v,b,\pi)$ is a distributional solution to \eqref{MHD}.

2. For any $R>0$, $(v,b)$ satisfies
\EQN{
\esssup_{0\le t<R^2\wedge T} &\sup_{x_0\in\R^3} \int_{B_R(x_0)} \frac12 \bke{|v(x,t)|^2 + |b(x,t)|^2} dx\\ 
&+ \sup_{x_0\in\R^3} \int_0^{R^2\wedge T} \int_{B_R(x_0)} \bke{|\nb v(x,t)|^2 + |\nb b(x,t)|^2} dxdt < \infty.
}

3. For any $R>0$, $x_0\in\R^3$, and $0<T'<T$, there exists a function of time $c_{x_0,R}(t)\in L^{3/2}(0,T')$ so that, for every $0<t<T'$ and $x\in B_{2R}(x_0)$,
\EQS{\label{eq-pi-formula-mhd}
\pi(x,t) &= -\De^{-1} \div\div\bkt{(v\otimes v - b\otimes b)\chi_{4R}(x-x_0)}\\
&\quad - \int_{\R^3} \bke{K(x-y) - K(x_0-y)} (v\otimes v - b\otimes b )(y,t) \bke{1-\chi_{4R}(y-x_0)} dy + c_{x_0,R}(t)
}
in $L^{3/2}(B_{2R}(x_0)\times(0,T'))$ where $K(x)$ is the kernel of $\De^{-1}\div\div$, $K_{ij}(x) = \pd_i\pd_j \frac{-1}{4\pi|x|}$, and $\chi_{4R}(x)$ is the characteristic function of $B_{4R}$.

4. For all compact subsets $K$ of $\R^3$ we have $v(t)\to v_0$ and $b(t)\to b_0$ in $L^2(K)$ as $t\to0^+$.

5. For all cylinders $Q$ compactly supported in $\R^3\times(0,T)$ and all nonnegative $\phi\in C^\infty_c(Q)$, we have the local energy inequality
\EQS{\label{lei_mhd}
2\int\int\left(|\nb v|^2+|\nb b|^2\right)\phi\,dxdt\le&\int\int\left(|v|^2+|b|^2\right)\left(\pd_t\phi+\De\phi\right)dxdt\\
&+\int\int\left(|v|^2+|b|^2+2\pi\right)(v\cdot\nb\phi)\,dxdt\\
&-2\int\int(v\cdot b)(b\cdot\nb\phi)\,dxdt.
}

6. The functions
\[
t\mapsto \int_{\R^3} v(x,t)\cdot w(x)\, dx\qquad
t\mapsto \int_{\R^3} b(x,t)\cdot w(x)\, dx
\]
are continuous in $t\in[0,T)$ for any compactly supported $w\in L^2(\R^3)$.

For given divergence-free $v_0, b_0\in L^2_\uloc$, let $\mathcal{N}_{\rm MHD}(v_0,b_0)$ denote the set of all local energy solutions to \eqref{MHD} with initial data $(v_0,b_0)$.
\end{definition}

\begin{theorem}[Eventual regularity in $E^2_q$ (MHD)]\label{bt4-thm-1.3-mhd}%
Assume $v_0,b_0\in E^2_q$, where $1\le q\le3$, are divergence free, and $(v,b)\in\mathcal N_{\rm MHD}(v_0,b_0)$. 
Then $(v,b)$ has eventual regularity, i.e., there is $t_1<\infty$ such that $v$ and $b$ are regular at $(x,t)$ whenever $t\ge t_1$, and 
\[
\norm{(v,b)(\cdot,t)}_{L^\infty\times L^\infty} \lec t^{1/2},
\]
for sufficiently large $t$.
\end{theorem}

The proof of Theorem \ref{bt4-thm-1.3-mhd} is given in Section \ref{sec-eventual-reg}

Define the $\ell^q$ \emph{local energy}
\EQ{\label{ETq.def}
\norm{(v,b)}_{{\bf LE}_q(0,T)}
:= \norm{(v,b)}_{E^{\infty,2}_{T,q}\times E^{\infty,2}_{T,q}} + \norm{(\nb v,\nb b)}_{E^{2,2}_{T,q}\times E^{2,2}_{T,q}}.
}

\begin{theorem}[Explicit growth rate in $E^2_q$ (MHD)]\label{thm-1.4BT4-mhd}
Assume $v_0,b_0\in E^2_q$, where $1\le q<\infty$, are divergence free, and $(v,b)\in\mathcal N_{\rm MHD}(v_0,b_0)$ satisfies, for some $T_2>0$,
\[
\norm{(v,b)}_{{\bf LE}_q(0,T_1)}<\infty,\quad \forall T_1\in(0,T_2).
\]
Then, for any $R\ge 1$, with $T=\min\bke{\la_1(1+\norm{(v_0,b_0)}_{E^2_q\times E^2_q})^{-4}R^2, T_2}$, we have
\EQ{\label{eq-thm1.4BT4}
\bigg\|\esssup_{0\leq t \leq T} \int_{B_R(Rk) } \bke{|v|^2 + |b|^2} dx+ \int_0^T\int_{B_R(Rk) } \bke{|\nabla v|^2 + |\nabla b|^2} dx\,dt \bigg\| _{\ell^{q/2}(k\in \ZZ^3)}
\leq  C \norm{(v_0,b_0)}_{E^2_q\times E^2_q}^2,
}
for positive constants $\la_1$ and $C$ independent of $(v_0,b_0)$ and $R$.
In particular, if $T_2=\infty$ then $T\to\infty$ as $R\to\infty$.
\end{theorem}

The proof of Theorem \ref{thm-1.4BT4-mhd} is given in Section \ref{sec-growth-rate}

\begin{theorem}[Existence in $E^2_q$ (MHD)]\label{th4.8-mhd}
Assume $v_0,b_0\in E^2_q$, where $1\le q<\infty$, are divergence free. Then, there exists a time-global local energy solution $(v,b)$ and associated pressure $\pi$ to \eqref{MHD} in $\R^3$  with initial data $v_0,  b_0$ so that, for any $0<T<\infty$,
\EQ{\label{eq-LEq-global-mhd}
\norm{(v,b)}_{{\bf LE}_q(0,T)}<\infty.
}
In particular, $(v,b)\in L^\infty(0,T; E^2_q\times E^2_q)$.
\end{theorem}

The proof of Theorem \ref{th4.8-mhd} is divided into two cases: $q\ge2$ and $1\le q<2$.
The case $q\ge2$ is addressed in Section \ref{sec-q>=2}, and the case $1\le q<2$ is handled separately in Section \ref{sec-1<=q<2}.

\subsection{Weak solutions of viscoelastic Navier--Stokes equations with damping}

We now define analogous local energy solutions to the viscoelastic Navier--Stokes equations with damping as follows.

\begin{definition}[local energy solution (vNSEd)]
Let $0<T\le\infty$. A pair of a vector field and a tensor $(v,{\bf F})$, ${\bf F} = [f_1,f_2,f_3]$, $v, f_m\in L^2_\loc(\R^3\times[0,T))$, $m=1,2,3$, is a local energy solution to \eqref{vNSEd} with initial data $v_0, {\bf F}_0\in L^2_\uloc(\R^3)$, ${\bf F}_0 = [(f_1)_0, (f_2)_0, (f_3)_0]$, where $v_0, (f_m)_0$ are divergence free, denoted as $(v,{\bf F})\in\mathcal N_{\rm vNSEd}(v_0,{\bf F}_0)$, if the following hold:

1. There exists $\pi\in L^{3/2}_\loc(\R^3\times[0,T))$ such that $(v,{\bf F},\pi)$ is a distributional solution to \eqref{vNSEd}.

2. For any $R>0$, $(v,{\bf F})$ satisfies
\EQN{
\esssup_{0\le t<R^2\wedge T} &\sup_{x_0\in\R^3} \int_{B_R(x_0)} \frac12 \bke{|v(x,t)|^2 + |{\bf F}(x,t)|^2} dx\\ 
&+ \sup_{x_0\in\R^3} \int_0^{R^2\wedge T} \int_{B_R(x_0)} \bke{|\nb v(x,t)|^2 + |\nb {\bf F}(x,t)|^2} dxdt < \infty.
}

3. For any $R>0$, $x_0\in\R^3$, and $0<T'<T$, there exists a function of time $c_{x_0,R}(t)\in L^{3/2}(0,T')$ so that, for every $0<t<T'$ and $x\in B_{2R}(x_0)$,
\EQS{\label{eq-pi-formula-vnsed}
\pi(x,t) &= -\De^{-1} \div\div\bkt{\bke{v\otimes v - \sum_{n=1}^3 f_n\otimes f_n}\chi_{4R}(x-x_0)}\\
&\quad - \int_{\R^3} \bke{K(x-y) - K(x_0-y)} \bke{v\otimes v - \sum_{n=1}^3 f_n\otimes f_n }(y,t) \bke{1-\chi_{4R}(y-x_0)} dy + c_{x_0,R}(t)
}
in $L^{3/2}(B_{2R}(x_0)\times(0,T'))$ where $K(x)$ is the kernel of $\De^{-1}\div\div$, $K_{ij}(x) = \pd_i\pd_j \frac{-1}{4\pi|x|}$, and $\chi_{4R}(x)$ is the characteristic function of $B_{4R}$.

4. For all compact subsets $K$ of $\R^3$ we have $v(t)\to v_0$ and $f_m(t)\to (f_m)_0$, $m=1,2,3$, in $L^2(K)$ as $t\to0^+$.

5. For all cylinders $Q$ compactly supported in $\R^3\times(0,T)$ and all nonnegative $\phi\in C^\infty_c(Q)$, we have the local energy inequality
\EQS{\label{lei_mhd}
2\int\int\left(|\nb v|^2+|\nb {\bf F}|^2\right)\phi\,dxdt\le&\int\int\left(|v|^2+|{\bf F}|^2\right)\left(\pd_t\phi+\De\phi\right)dxdt\\
&+\int\int\left(|v|^2+|{\bf F}|^2+2\pi\right)(v\cdot\nb\phi)\,dxdt\\
&-2\sum_{n=1}^3\int\int(v\cdot f_n)(f_n\cdot\nb\phi)\,dxdt.
}

6. The functions
\[
t\mapsto \int_{\R^3} v(x,t)\cdot w(x)\, dx\qquad
t\mapsto \int_{\R^3} {\bf F}(x,t)\cdot w(x)\, dx
\]
are continuous in $t\in[0,T)$ for any compactly supported $w\in L^2(\R^3)$.

For given divergence-free $v_0, (f_m)_0\in L^2_\uloc$, $m=1,2,3$, let $\mathcal{N}_{\rm MHD}(v_0,{\bf F}_0)$, ${\bf F}_0 = [(f_1)_0, (f_2)_0, (f_3)_0]$, denote the set of all local energy solutions to \eqref{vNSEd} with initial data $(v_0,{\bf F}_0)$.
\end{definition}

The main results concerning weak solutions of the viscoelastic Navier--Stokes equations with damping are stated as follows:

\begin{theorem}[Eventual regularity in $E^2_q$ (vNSEd)]\label{bt4-thm-1.3-vnsed}%
Assume $v_0,(f_1)_0, (f_2)_0, (f_3)_0\in E^2_q$, where $1\le q\le3$, are divergence free, and $(v,{\bf F})\in\mathcal N_{\rm vNSEd}(v_0,{\bf F}_0)$, ${\bf F}_0 = [(f_1)_0, (f_2)_0, (f_3)_0]$.
Then $(v,{\bf F})$ has eventual regularity, i.e., there is $t_1<\infty$ such that $v$ and ${\bf F}$ are regular at $(x,t)$ whenever $t\ge t_1$, and 
\[
\norm{(v,{\bf F})(\cdot,t)}_{L^\infty\times {\bf L}^\infty} \lec t^{1/2},
\]
for sufficiently large $t$.
\end{theorem}

Define the $\ell^q$ \emph{local energy}
\EQN{
\norm{(v,{\bf F})}_{{\bf LE}_q(0,T)}
:= \norm{(v,{\bf F})}_{E^{\infty,2}_{T,q}\times {\bf E}^{\infty,2}_{T,q}} + \norm{(\nb v,\nb {\bf F})}_{E^{2,2}_{T,q}\times {\bf E}^{2,2}_{T,q}}.
}

\begin{theorem}[Explicit growth rate in $E^2_q$ (vNSEd)]\label{thm-1.4BT4-vnsed}
Assume $v_0,(f_1)_0, (f_2)_0, (f_3)_0\in E^2_q$, where $1\le q<\infty$, are divergence free, and $(v,{\bf F})\in\mathcal N_{\rm vNSEd}(v_0,{\bf F}_0)$, ${\bf F}_0 = [(f_1)_0, (f_2)_0, (f_3)_0]$, satisfies, for some $T_2>0$,
\[
\norm{(v,{\bf F})}_{{\bf LE}_q(0,T_1)}<\infty,\quad \forall T_1\in(0,T_2).
\]
Then, for any $R\ge 1$, with $T=\min\bke{\la_1(1+\norm{(v_0,{\bf F}_0)}_{E^2_q\times {\bf E}^2_q})^{-4}R^2, T_2}$, we have
\EQN{
\bigg\|\esssup_{0\leq t \leq T} \int_{B_R(Rk) } \bke{|v|^2 + |{\bf F}|^2} dx+ \int_0^T\int_{B_R(Rk) } \bke{|\nabla v|^2 + |\nabla {\bf F}|^2} dx\,dt \bigg\| _{\ell^{q/2}(k\in \ZZ^3)}
\leq  C \norm{(v_0,{\bf F}_0)}_{E^2_q\times {\bf E}^2_q}^2,
}
for positive constants $\la_1$ and $C$ independent of $(v_0,{\bf F}_0)$ and $R$.
In particular, if $T_2=\infty$ then $T\to\infty$ as $R\to\infty$.
\end{theorem}

\begin{theorem}[Existence in $E^2_q$ (vNSEd)]\label{th4.8-vnsed}
Assume $v_0,(f_1)_0, (f_2)_0, (f_3)_0\in E^2_q$, where $1\le q<\infty$, are divergence free. Then, there exists a time-global local energy solution $(v,{\bf F})$ and associated pressure $\pi$ to \eqref{vNSEd} in $\R^3$  with initial data $v_0,  {\bf F}_0$, ${\bf F}_0 = [(f_1)_0, (f_2)_0, (f_3)_0]$, so that, for any $0<T<\infty$,
\EQN{
\norm{(v,{\bf F})}_{{\bf LE}_q(0,T)}<\infty.
}
In particular, $(v,{\bf F})\in L^\infty(0,T; E^2_q\times {\bf E}^2_q)$.
\end{theorem}

The remainder of the paper is organized as follows: In Section \ref{sec-mild}, we construct mild solutions in critical and subcritical spaces and prove Theorems \ref{thrm.subcritical-mhd}, \ref{thrm:critical-mhd}, \ref{thrm:critical2-mhd}, \ref{thrm.subcritical-vnsed}, \ref{thrm:critical-vnsed}, and \ref{thrm:critical2-vnsed}.
In Section \ref{sec-weak}, we examine properties of local energy solutions, including uniqueness and regularity, establish a priori bounds and explicit growth rate, 
and prove the global existence results: Theorems \ref{bt4-thm-1.3-mhd}, \ref{thm-1.4BT4-mhd}, \ref{th4.8-mhd}, \ref{bt4-thm-1.3-vnsed}, \ref{thm-1.4BT4-vnsed}, and \ref{th4.8-vnsed}.

\section{Construction of mild solutions in Wiener amalgam spaces}\label{sec-mild}

This section is dedicated to proving Theorems \ref{thrm.subcritical-mhd}, \ref{thrm:critical-mhd}, \ref{thrm:critical2-mhd}, \ref{thrm.subcritical-vnsed}, \ref{thrm:critical-vnsed}, and \ref{thrm:critical2-vnsed}.
The proof techniques closely follow those outlined in \cite[Section 3]{BLT-MathAnn2024}.
Specifically, we apply the Picard iteration scheme as in \cite[Section 3]{BLT-MathAnn2024} to the problems \eqref{eq-mild-mhd} and \eqref{eq-mild-vnsed} to construct mild solutions for both the MHD equations \eqref{MHD} and the viscoelastic Navier--Stokes equations with damping \eqref{vNSEd}, respectively.
Since the structure of \eqref{vNSEd} is analogous to that of  \eqref{MHD}, we prove  Theorems \ref{thrm.subcritical-mhd}, \ref{thrm:critical-mhd}, and \ref{thrm:critical2-mhd} for \eqref{MHD}. 
The proofs of Theorems \ref{thrm.subcritical-vnsed}, \ref{thrm:critical-vnsed}, and \ref{thrm:critical2-vnsed} for \eqref{vNSEd} are omitted for brevity.

\subsection{Mild solutions in subcritical spaces: Proof of Theorem \ref{thrm.subcritical-mhd}}\label{sec-subcritical}

The proof of Theorem \ref{thrm.subcritical-mhd} is an adaption of the proof of \cite[Theorem 1.1]{BLT-MathAnn2024} for the Navier--Stokes equations to the MHD equations.

Define \[\norm{(v,b)}_{\mathcal{E}_T} = \sup_{0\le t\le T} \norm{(v,b)(t)}_{E^r_q\times E^r_q}.\]
By \cite[Lemma 2.1]{BLT-MathAnn2024}, we have the linear estimate
\[
\norm{(e^{t\De}v_0, e^{t\De}b_0)}_{E^r_q\times E^r_q} \lec \norm{(v_0,b_0)}_{E^r_q\times E^r_q}.
\]
So,
\EQ{\label{eq-linear_mE_T}
\norm{(e^{t\De}v_0, e^{t\De}b_0)}_{\mathcal{E}_T} \le C_1 \norm{(v_0,b_0)}_{E^r_q\times E^r_q}. 
}
For bilinear estimate, again by \cite[Lemma 2.1]{BLT-MathAnn2024}, we estimate:
\EQN{
\| B_1&((v,b), (u,a))  (t)\|_{E^r_q} \\
&\lec \int_0^t \bigg(\frac 1 {(t-s)^{\frac 1 2+\frac3{2r}}   }  +\frac 1 {(t-s)^{\frac 1 2 }   }\bigg)
 \bke{\| (v\otimes u)(s)\|_{E^{r/2}_q} + \| (b\otimes a)(s)\|_{E^{r/2}_q}} ds\\
&\lec \bke{ t^{\frac12 - \frac3{2r}} +  t^{\frac12} } \bke{\sup_{0<\tau<t} \| v(\tau)\|_{E^r_q}
\sup_{0<\tau<t} \|u(\tau)\|_{E^r_\infty} + \sup_{0<\tau<t} \| b(\tau)\|_{E^r_q}
\sup_{0<\tau<t} \|a(\tau)\|_{E^r_\infty} }\\
&\lec \bke{ t^{\frac12 - \frac3{2r}} +  t^{\frac12} } \bke{\sup_{0<\tau<t} \| v(\tau)\|_{E^r_q}
\sup_{0<\tau<t} \|u(\tau)\|_{E^r_q} + \sup_{0<\tau<t} \| b(\tau)\|_{E^r_q}
\sup_{0<\tau<t} \|a(\tau)\|_{E^r_q} }\\
&\lec \bke{ t^{\frac12 - \frac3{2r}} +  t^{\frac12} } \| (v,b)\|_{\mathcal E_T}\|(u,a)\|_{\mathcal E_T},
}
where we used the embedding $E^r_q\subset E^r_\infty$.
Similarly, 
\EQN{
\| B_2&((v,b), (u,a))  (t)\|_{E^r_q} \\
&\lec \int_0^t \bigg(\frac 1 {(t-s)^{\frac 1 2+\frac3{2r}}   }  +\frac 1 {(t-s)^{\frac 1 2 }   }\bigg)
 \bke{\| (v\otimes a)(s)\|_{E^{r/2}_q} + \| (b\otimes u)(s)\|_{E^{r/2}_q}} ds\\
&\lec \bke{ t^{\frac12 - \frac3{2r}} +  t^{\frac12} } \bke{\sup_{0<\tau<t} \| v(\tau)\|_{E^r_q}
\sup_{0<\tau<t} \|a(\tau)\|_{E^r_\infty} + \sup_{0<\tau<t} \| b(\tau)\|_{E^r_q}
\sup_{0<\tau<t} \|u(\tau)\|_{E^r_\infty} }\\
&\lec \bke{ t^{\frac12 - \frac3{2r}} +  t^{\frac12} } \bke{\sup_{0<\tau<t} \| v(\tau)\|_{E^r_q}
\sup_{0<\tau<t} \|a(\tau)\|_{E^r_q} + \sup_{0<\tau<t} \| b(\tau)\|_{E^r_q}
\sup_{0<\tau<t} \|u(\tau)\|_{E^r_q} }\\
&\lec \bke{ t^{\frac12 - \frac3{2r}} +  t^{\frac12} } \| (v,b)\|_{\mathcal E_T}\|(u,a)\|_{\mathcal E_T}.
}
Thus, the full bilinear estimate becomes 
\EQ{\label{eq-bilinear_mE_T}
\norm{B((v,b), (u,a))}_{\mathcal E_T} \le C_2 \bke{ T^{\frac12 - \frac3{2r}} +  T^{\frac12} } \| (v,b)\|_{\mathcal E_T}\|(u,a)\|_{\mathcal E_T}.
}
We look for a solution of the form 
\[
(v,b) = (e^{t\De}v_0,e^{t\De}b_0) - B((v,b),(v,b)).
\]
Suppose $T$ is small enough so that $  \norm{(v_0,b_0)}_{E^r_q\times E^r_q} < (8C_1C_2(   T^{1/2 - 3/(2r)} +  T^{1/2}	 ))^{-1}$. 
Then by the Picard contraction principle, there exists a unique strong mild solution satisfying 
\EQ{\label{eq-picard-uniqueclass}
\norm{(v,b)}_{\mathcal E_T} \le 2C_1 \norm{(v_0,b_0)}_{E^r_q\times E^r_q}.
}

To prove continuity at time zero, assume $r,q<\I$. 
Then
\EQN{
\| v(t) - v_0\|_{E^r_q} &\leq \| B_1(u,u)(t) \|_{E^r_q} + \| e^{t\Delta}v_0 - v_0\|_{E^r_q}\\
&\lec \bke{ t^{\frac12 - \frac3{2r}} +  t^{\frac12} } \bkt{ \sup_{0<\tau<t} \| v(\tau)\|_{E^r_q} \sup_{0<\tau<t} \| v(\tau) \|_{E^r_\infty} + \sup_{0<\tau<t} \| b(\tau)\|_{E^r_q} \sup_{0<\tau<t} \| b(\tau) \|_{E^r_\infty} }\\
&\qquad +\| e^{t\Delta}v_0 - v_0\|_{E^r_q}\\
&\le \bke{ t^{\frac12 - \frac3{2r}} +  t^{\frac12} } \bkt{ \sup_{0<\tau<t} \| v(\tau)\|_{E^r_q} \sup_{0<\tau<t} \| v(\tau) \|_{E^r_q} + \sup_{0<\tau<t} \| b(\tau)\|_{E^r_q} \sup_{0<\tau<t} \| b(\tau) \|_{E^r_q} }\\
&\qquad +\| e^{t\Delta}v_0 - v_0\|_{E^r_q}.
}
Both terms tend to zero as $t\to0^+$, the latter by \cite[Lemma 2.3]{BLT-MathAnn2024}.
Hence, $\| v(t) - v_0\|_{E^r_q} \to 0$ as $t\to 0^+$.
Similarly, 
\EQN{
\| b(t) - b_0\|_{E^r_q} 
&\lec \bke{ t^{\frac12 - \frac3{2r}} +  t^{\frac12} } \bkt{ \sup_{0<\tau<t} \| v(\tau)\|_{E^r_q} \sup_{0<\tau<t} \| b(\tau) \|_{E^r_q} + \sup_{0<\tau<t} \| v(\tau)\|_{E^r_q} \sup_{0<\tau<t} \| b(\tau) \|_{E^r_q} }\\
&\qquad +\| e^{t\Delta}b_0 - b_0\|_{E^r_q}
\to 0 \text{ as }t\to 0^+.
}
The continuity at $t\in (0,T)$ can be shown as usual, see e.g., \cite[lines 3-8, page 86]{Tsai-book}, including $r=\I$ or $q=\I$. 

If either $r=\I$ or $q=\I$, the semigroup terms no longer vanish, but the bilinear terms still tend to zero: 
\[
\bke{ t^{\frac12 - \frac3{2r}} +  t^{\frac12} }  \sup_{0<\tau<t} \| v(\tau)\|_{E^r_q}\sup_{0<\tau<t}  \| v(\tau) \|_{E^r_q}  \to 0,\ 
\bke{ t^{\frac12 - \frac3{2r}} +  t^{\frac12} }  \sup_{0<\tau<t} \| b(\tau)\|_{E^r_q}\sup_{0<\tau<t}  \| b(\tau) \|_{E^r_q}  \to 0,
\]
and 
\[
\bke{ t^{\frac12 - \frac3{2r}} +  t^{\frac12} }  \sup_{0<\tau<t} \| v(\tau)\|_{E^r_q}\sup_{0<\tau<t}  \| b(\tau) \|_{E^r_q}  \to 0.
\]
Hence,
\[
\| v (t)-e^{t\Delta}v_0 \|_{E^r_q} + \| b (t)-e^{t\Delta}b_0 \|_{E^r_q}\to 0\text{ as }t\to 0^+,
\]
as asserted in the theorem.

For uniqueness, let $(v,b), (v',b')\in L^\infty(0,T;E^r_q\times E^r_q) \cap C((0,T);E^r_q\times E^r_q)$ be two mild solutions with initial data $(v_0, b_0)\in E^r_q\times E^r_q$. 
Then, for $0<t<T'\le T$,
\EQN{
&\norm{(v - v')(t)}_{E^r_q} 
\le \norm{B_1((v-v',b-b'), (v,b))(t)}_{E^r_q} + \norm{B_1((v',b'), (v-v',b-b'))(t)}_{E^r_q} \\
&\quad\lec \bke{ t^{\frac12 - \frac3{2r}} +  t^{\frac12} }   \bke{\norm{(v,b)}_{\mathcal E_T}+\norm{(v',b')}_{\mathcal E_T}}
\bke{ \sup_{0<t<T'} \| (v - v')(t)\|_{E^r_q} + \sup_{0<t<T'} \| (b - b')(t)\|_{E^r_q}},
}
so that we have
\EQN{
\sup_{0<t<T'} \| (v - v')(t)\|_{E^r_q}
&\lec \bke{ T'^{\frac12 - \frac3{2r}} +  T'^{\frac12}} \bke{\norm{(v,b)}_{\mathcal E_T}+\norm{(v',b')}_{\mathcal E_T}}\\
&\qquad\cdot\sup_{0<t<T'} 
\bke{ \sup_{0<t<T'} \| (v - v')(t)\|_{E^r_q} + \sup_{0<t<T'} \| (b - b')(t)\|_{E^r_q}}.
}
Similarly,
\EQN{
\sup_{0<t<T'} \| (b - b')(t)\|_{E^r_q}
&\lec \bke{ T'^{\frac12 - \frac3{2r}} +  T'^{\frac12}} \bke{\norm{(v,b)}_{\mathcal E_T}+\norm{(v',b')}_{\mathcal E_T}}\\
&\qquad\cdot\sup_{0<t<T'} 
\bke{ \sup_{0<t<T'} \| (v - v')(t)\|_{E^r_q} + \sup_{0<t<T'} \| (b - b')(t)\|_{E^r_q}}.
}
Thus, for small enough $T'>0$, this implies $(v,b)=(v',b')$ on $(0,T')$, and repeating the argument yields uniqueness on $(0,T)$.

To obtain the spacetime integral bound, assume $3<r\le s\le \infty$, $r\le p<\infty$, 
$\frac2s + \frac3p = \frac3r$, and $1\le q=m \le\infty$.
Consider the Banach space
\[
X_T = \mathcal E_T \cap (E^{s,p}_{T,q}\times E^{s,p}_{T,q}).
\]
We can assume $m=q$ since $\norm{f}_{E^{s,p}_{T,m}}\le \norm{f}_{E^{s,p}_{T,q}}$ for $m\ge q$. %
From 
$\frac3p = \frac3r-\frac2s  \ge \frac3r-\frac2r$ we get $p \le 3r<\infty$.
For the linear term, by \eqref{eq-linear_mE_T} and \cite[Lemma 2.4]{BLT-MathAnn2024} (which needs $r<\infty$ and $r \le s$), we have for a fixed $\epsilon>0$ that
\EQN{
\norm{(e^{t\De} v_0, e^{t\De}b_0)}_{X_T} 
&= \norm{(e^{t\De} v_0, e^{t\De} b_0)}_{\mathcal E_T} + \norm{(e^{t\De} v_0,e^{t\De} b_0)}_{E^{s,p}_{T,q}\times E^{s,p}_{T,q}}\\
& \le C_3(1+T^{1/s+\epsilon}) \norm{(v_0,b_0)}_{E^r_q\times E^r_q}.
}
For the bilinear term, by \eqref{eq-bilinear_mE_T} and \cite[Lemma 2.7]{BLT-MathAnn2024}
with $\td p = p/2$ and $\td s = s/2$ so that $\si = \frac12 - \frac3{2r}>0$ due to $r>3$, (allowing $s=\infty$),
\EQN{
&\norm{B((v,b),(u,a))}_{X_T} = \norm{B((v,b),(u,a))}_{\mathcal E_T} + \norm{B_1((v,b),(u,a))}_{E^{s,p}_{T,q}} + \norm{B_2((v,b),(u,a))}_{E^{s,p}_{T,q}}\\
&\quad\le C_4 \left[ \bke{T^{\frac12-\frac3{2r}} + T^{\frac12}} \norm{(v,b)}_{\mathcal E_T} \norm{(u,a)}_{\mathcal E_T}\right.\\
&\qquad\left. + \bke{T^{\frac12 - \frac3{2 r}} + T^{1-\frac1s }} \bke{\norm{v\otimes u}_{E^{\frac{s}2,\frac{p}2}_{T,q}} + \norm{b\otimes a}_{E^{\frac{s}2,\frac{p}2}_{T,q}} + \norm{v\otimes a}_{E^{\frac{s}2,\frac{p}2}_{T,q}} + \norm{b\otimes u}_{E^{\frac{s}2,\frac{p}2}_{T,q}} } \right].
}
Since%
\[
\norm{f\otimes g}_{E^{\frac{s}2,\frac{p}2}_{T,q}} 
\le \norm{f}_{E^{s,p}_{T,q}} \norm{g}_{E^{s,p}_{T,\infty}}
\le \norm{f}_{E^{s,p}_{T,q}} \norm{g}_{E^{s,p}_{T,q}},
\]
we derive%
\[
\norm{B((v,b),(u,a))}_{X_T} \le 2C_4 \bke{T^{\frac12-\frac3{2r}} + T^{1 - \frac1s}} \norm{(v,b)}_{X_T} \norm{(u,a)}_{X_T}.
\]
Choose $T>0$ small enough such that $ \norm{(v_0,b_0)}_{E^r_q\times E^r_q} < [8C_3 C_4 (1 + T^{\frac1s+\epsilon})(T^{\frac12-\frac3{2r}} + T^{1 - \frac1s})]^{-1}$,  
the Picard iteration yields a unique strong mild solution $(\td v,\td b)\in X_T$. 
Since $X_T\subset L^\infty_T (E^r_q\times E^r_q)$, uniqueness implies $(\td v,\td b)$, and the solution satisfies the desired spacetime integral bound.
This completes the proof of Theorem \ref{thrm.subcritical-mhd}.
\qed

\subsection{Mild solutions in critical spaces with small data: Proof of Theorem \ref{thrm:critical-mhd}}\label{sec-critical-I}

The proof of Theorem \ref{thrm:critical-mhd} is an adaption of the proof of \cite[Theorem 1.2]{BLT-MathAnn2024} for the Navier--Stokes equations to the MHD equations.

Let
\[
\norm{(v,b)}_{\tilde{\mathcal{E}}_T} 
:= \sup_{0<t<T} \norm{(v,b)(\cdot,t)}_{E^3_q\times E^3_q}    + \sup_{0<t<T} t^{\frac12} \norm{(v,b)(\cdot,t)}_{E^\I_q\times E^\I_q}
\]
and
\[ 
\norm{(v,b)}_{\tilde{\mathcal{F}}_T} 
:= \sup_{0<t<T} t^{\frac14} \norm{(v,b)(\cdot,t)}_{E^6_q\times E^6_q}.
\]
The inclusion $\tilde{\mathcal{E}}_T \subset \tilde{\mathcal{F}}_T$ is obvious.
We define the spaces
\EQ{
\mathcal{E}_T := \tilde{\mathcal{E}}_T \cap E^{s,p}_{T,m} \qquad \text{ and } \qquad 
\mathcal{F}_T := \tilde{\mathcal{F}}_T \cap E^{s,p}_{T,m} ,
}
with corresponding norms
\[
\norm{(v,b)}_{\mathcal{E}_T} := \norm{(v,b)}_{\tilde{\mathcal{E}}_T} + \norm{(v,b)}_{E^{s,p}_{T,m}}\qquad \text{ and } \qquad 
\norm{(v,b)}_{\mathcal{F}_T} := \norm{(v,b)}_{\tilde{\mathcal{F}}_T} + \norm{(v,b)}_{E^{s,p}_{T,m}}.
\]
The inclusion $\mathcal{E}_T\subset\mathcal{F}_T$ follows immediately.

From \cite[(2.16)]{BLT-MathAnn2024}, and choosing $\epsilon >\frac 3{2m}-\frac 3 {2q}$ from the definition of $\beta$ in \cite[Lemma 2.4]{BLT-MathAnn2024}, we have 
\EQ{
\| (e^{t\Delta}v_0,e^{t\Delta}b_0)\|_{E^{s,p}_{T,m}\times E^{s,p}_{T,m}}
\lec (1+T^{\frac1 s +\epsilon}) \|(v_0, b_0)\|_{E^3_q\times E^3_q}.
}
Also,  by \cite[Lemma 2.1]{BLT-MathAnn2024}, 
\EQ{
t^{\frac {1}{4}}  \| (e^{t\Delta}v_0, e^{t\Delta}b_0)\|_{E^6_q\times E^6_q} 
\lec  (1+ T^{ \frac 1 4}) \|(v_0, b_0)\|_{E^3_q\times E^3_q}.
}
Combining, we obtain
\EQ{\label{3.21a}
\| (e^{t\Delta}v_0, e^{t\Delta}b_0)\|_{\mathcal F_T} \le C_1( 1 + T^{\frac1 s +\epsilon} +T^{\frac {1}4}   ) \| (v_0,b_0)\|_{E^3_q\times E^3_q}.
}

Using \cite[Lemma 2.7]{BLT-MathAnn2024} with $\tilde{s}=s/2$, $\td p = p/2$, and $\td m = m$ (so that $\si=0$), and applying H\"{o}lder's inequality, we estimate the bilinear terms:
\EQS{\label{eq-0321-a-B1}
\norm{B_1((v,b),(u,a))}_{E^{s,p}_{T,m}} 
&\le C(1+T^{1 - \frac 1s}) \bke{ \norm{v\otimes u}_{E^{\frac s 2,\frac p 2}_{T,m}} + \norm{b\otimes a}_{E^{\frac s 2,\frac p 2}_{T,m}} }\\
&\le C(1+T^{1 - \frac 1s}) \bke{ \norm{v}_{E^{s,p}_{T,m}} \norm{u}_{E^{s,p}_{T,\infty}} + \norm{b}_{E^{s,p}_{T,m}} \norm{a}_{E^{s,p}_{T,\infty}} }\\
&\le C(1+T^{1 - \frac 1s}) \bke{ \norm{v}_{E^{s,p}_{T,m}} \norm{u}_{E^{s,p}_{T,m}} + \norm{b}_{E^{s,p}_{T,m}} \norm{a}_{E^{s,p}_{T,m}} },
}
where we used the inclusion $E^{s,p}_{T,m} \subset E^{s,p}_{T,\infty}$ in the last inequality. 
Similarly, we have 
\EQS{\label{eq-0321-a-B2}
\norm{B_2((v,b),(u,a))}_{E^{s,p}_{T,m}} 
&\le C(1+T^{1 - \frac 1s}) \bke{ \norm{v\otimes a}_{E^{\frac s 2,\frac p 2}_{T,m}} + \norm{b\otimes u}_{E^{\frac s 2,\frac p 2}_{T,m}} }\\
&\le C(1+T^{1 - \frac 1s}) \bke{ \norm{v}_{E^{s,p}_{T,m}} \norm{a}_{E^{s,p}_{T,\infty}} + \norm{b}_{E^{s,p}_{T,m}} \norm{u}_{E^{s,p}_{T,\infty}} }\\
&\le C(1+T^{1 - \frac 1s}) \bke{ \norm{v}_{E^{s,p}_{T,m}} \norm{a}_{E^{s,p}_{T,m}} + \norm{b}_{E^{s,p}_{T,m}} \norm{u}_{E^{s,p}_{T,m}} }.
}
Hence,
\EQS{\label{eq-0321-a}
\norm{B((v,b),(u,a))}_{E^{s,p}_{T,m}\times E^{s,p}_{T,m}} 
&\le C(1+T^{1 - \frac 1s}) \norm{(v,b)}_{\mathcal{F}_T} \norm{(u,a)}_{\mathcal{F}_T}.
}
Additionally, applying \cite[Lemma 2.1]{BLT-MathAnn2024} and H\"older inequality \eqref{Holder}, we obtain
\EQN{ 
\norm{B_1((v,b),(u,a)}_{E^6_q}(t) 
&\lec \int_0^t \bke{\frac1{(t-\tau)^{\frac12}} + \frac1{(t-\tau)^{\frac34}}} \bke{ \norm{(v\otimes u)(\tau)}_{E^3_{q}} + \norm{(b\otimes a)(\tau)}_{E^3_{q}} } d\tau\\
&\le \int_0^t \bke{\frac1{(t-\tau)^{\frac12}} + \frac1{(t-\tau)^{\frac34}}} \bke{ \norm{v(\tau)}_{E^6_q} \norm{u(\tau)}_{E^6_q} + \norm{b(\tau)}_{E^6_q} \norm{a(\tau)}_{E^6_q} } d\tau\\
&\lec (1+t^{-1/4}) \norm{(v,b)}_{\tilde{\mathcal{F}}_T} \norm{(u,a)}_{\tilde{\mathcal{F}}_T}.
}
Similarly,
\EQN{ 
\norm{B_2((v,b),(u,a)}_{E^6_q}(t) 
&\lec \int_0^t \bke{\frac1{(t-\tau)^{\frac12}} + \frac1{(t-\tau)^{\frac34}}} \bke{ \norm{(v\otimes a)(\tau)}_{E^3_{q}} + \norm{(b\otimes u)(\tau)}_{E^3_{q}} } d\tau\\
&\le \int_0^t \bke{\frac1{(t-\tau)^{\frac12}} + \frac1{(t-\tau)^{\frac34}}} \bke{ \norm{v(\tau)}_{E^6_q} \norm{a(\tau)}_{E^6_q} + \norm{b(\tau)}_{E^6_q} \norm{u(\tau)}_{E^6_q} } d\tau\\
&\lec (1+t^{-1/4}) \norm{(v,b)}_{\tilde{\mathcal{F}}_T} \norm{(u,a)}_{\tilde{\mathcal{F}}_T}.
}
Hence
\[
\| B((v,b),(u,a))\|_{\mathcal F_T} \le C_2 (1+T^{\frac 1 4}+T^{1 - \frac 1s}  )  \|(v,b)\|_{\mathcal F_T} \|(u,a)\|_{\mathcal F_T}.
\]
Taking $\|(v_0,b_0)\|_{E^3_q\times E^3_q}$ small enough, by \eqref{3.21a}, it is possible to make 
\EQ{
\| (e^{t\Delta}v_0, e^{t\Delta}b_0)\|_{\mathcal F_T}  \leq \big[  4C_2 (   1+T^{1/4}+T^{1 - \frac 1s} ) 	  \big]^{-1}.
}
The Picard iteration yields a mild solution $(v,b)\in\mathcal{F}_T$ to \eqref{MHD} 
so that
\[
\| (v,b)\|_{\mathcal F_T} \leq  2\| (e^{t\Delta}v_0,e^{t\Delta}b_0) \|_{\mathcal F_T}.
\]
This solution is unique among all mild solutions $(u,a)$ with data $(v_0,b_0)$ satisfying $\|(u,a)\|_{\mathcal F_T}\le 2\| (e^{t\Delta}v_0, e^{t\Delta}b_0)\|_{\mathcal F_T}$.
In fact, since we can also apply the Picard contraction to $\td {\mathcal F}_T$, the solution is also unique in the class $\|(u,a)\|_{\td {\mathcal F}_T}\leq 2\| (e^{t\Delta}v_0, e^{t\Delta}b_0)\|_{\td {\mathcal F}_T}$.

Next, we show that a solution $(v,b)\in \mathcal{F}_T$ with small enough initial data $(v_0,b _0)\in E^3_q\times E^3_q$ also belongs to $\mathcal{E}_T$. Let $\{(v^{(n)}, b^{(n)})\}_{n\ge1}$ be the Picard iteration sequence in $\mathcal{F}_T$. 
By construction, 
\[
\norm{(v^{(n)}, b^{(n)})}_{ {\mathcal{F}}_T}  \le  2C_1 ( 1 + T^{\frac1 s +\epsilon} +T^{\frac {1}4}   ) \| (v_0, b_0)\|_{E^3_q\times E^3_q}< 1.
\]

Note that
\[
\norm{(v^{(n)}, b^{(n)})}_{\td{\mathcal{E}}_T} 
\le \norm{(e^{t\De} v_0, e^{t\De} b_0)}_{\td{\mathcal{E}}_T} + \norm{B((v^{(n-1)}, b^{(n-1)}), (v^{(n-1)}, b^{(n-1)}))}_{\td{\mathcal{E}}_T}.
\]
By \cite[Lemma 2.1]{BLT-MathAnn2024},
\[
\norm{(e^{t\De} v_0, e^{t\De} b_0)}_{\td{\mathcal{E}}_T} \le C(1+T^{1/2})  \| (v_0, b_0)\|_{E^3_q\times E^3_q}.
\]
As is usual in arguments like this, we now seek estimates for $B((v,b),(u,a))$ in $\td{\mathcal{E}}_T$ in terms of measurements of $(v,b)$ and $(u,a)$ in $\td{\mathcal{F}}_T$ and $\td{\mathcal{E}}_T$. 
We have by \cite[Lemma 2.1]{BLT-MathAnn2024} and H\"older inequality \eqref{Holder},\EQS{\label{ineq:E3qintegral-B1}
\| B_1((v,b),(u,a))\|_{E^3_q}&\lec \int_0^t \bigg(\frac 1 {(t-\tau)^{\frac 1 2  }}+\frac 1 {(t-\tau)^{\frac 3 4  }} \bigg) \bke{ \|(v\otimes u)(\tau)\|_{E^2_q} + \|(b\otimes a)(\tau)\|_{E^2_q} } d\tau\\
&\le \int_0^t \bigg(\frac 1 {(t-\tau)^{\frac 1 2  }}+\frac 1 {(t-\tau)^{\frac 3 4  }} \bigg)  \bke{\|v(\tau)\|_{E^3_q}\|u(\tau)\|_{E^6_q} + \|b(\tau)\|_{E^3_q}\|a(\tau)\|_{E^6_q}  }d\tau\\
&\lec 	(1+T^{1/4})			\|(v,b)\|_{\td{\mathcal{E}}_T}\|(u,a)\|_{\td{\mathcal{F}}_T} ,
}
and
\EQS{\label{ineq:E3qintegral-B2}
\| B_2((v,b),(u,a))\|_{E^3_q}&\lec \int_0^t \bigg(\frac 1 {(t-\tau)^{\frac 1 2  }}+\frac 1 {(t-\tau)^{\frac 3 4  }} \bigg) \bke{ \|(v\otimes a)(\tau)\|_{E^2_q} + \|(b\otimes u)(\tau)\|_{E^2_q} } d\tau\\
&\le \int_0^t \bigg(\frac 1 {(t-\tau)^{\frac 1 2  }}+\frac 1 {(t-\tau)^{\frac 3 4  }} \bigg)  \bke{\|v(\tau)\|_{E^3_q}\|a(\tau)\|_{E^6_q} + \|b(\tau)\|_{E^3_q}\|u(\tau)\|_{E^6_q}  }d\tau\\
&\lec 	(1+T^{1/4})			\|(v,b)\|_{\td{\mathcal{E}}_T}\|(u,a)\|_{\td{\mathcal{F}}_T} ,
}
which imply
\EQS{\label{ineq:E3qintegral}
\| B((v,b),(u,a))\|_{E^3_q\times E^3_q} \lec 	(1+T^{1/4})			\|(v,b)\|_{\td{\mathcal{E}}_T}\|(u,a)\|_{\td{\mathcal{F}}_T}.
}
Moreover, we have
\EQN{
t^{\frac 1 2}\| B_1((v,b),(u,a))\|_{E^\I_q} 
&\lec t^{\frac12}\int_0^t \bigg(\frac 1 {(t-\tau)^{\frac 1 2  }}+\frac 1 {(t-\tau)^{\frac 3 4  }} \bigg) \bke{ \| (v\otimes u)(\tau)\|_{E^6_q} + \| (b\otimes a)(\tau)\|_{E^6_q} } d\tau\\
&\le t^{\frac12} \int_0^t \bigg(\frac 1 {(t-\tau)^{\frac 1 2  }}+\frac 1 {(t-\tau)^{\frac 3 4  }} \bigg)  \bke{\|v(\tau)\|_{E^\infty_q}\|u(\tau)\|_{E^6_q} + \|b(\tau)\|_{E^\infty_q}\|a(\tau)\|_{E^6_q}}  d\tau\\
&\lec (1+T^{1/4}) \|(v,b)\|_{\td{\mathcal{E}}_T}\|(u,a)\|_{\td{\mathcal{F}}_T} ,
}
and
\EQN{
t^{\frac 1 2}\| B_2((v,b),(u,a))\|_{E^\I_q} 
&\lec t^{\frac12}\int_0^t \bigg(\frac 1 {(t-\tau)^{\frac 1 2  }}+\frac 1 {(t-\tau)^{\frac 3 4  }} \bigg) \bke{ \| (v\otimes a)(\tau)\|_{E^6_q} + \| (b\otimes u)(\tau)\|_{E^6_q} } d\tau\\
&\le t^{\frac12} \int_0^t \bigg(\frac 1 {(t-\tau)^{\frac 1 2  }}+\frac 1 {(t-\tau)^{\frac 3 4  }} \bigg)  \bke{\|v(\tau)\|_{E^\infty_q}\|a(\tau)\|_{E^6_q} + \|b(\tau)\|_{E^\infty_q}\|u(\tau)\|_{E^6_q}}  d\tau\\
&\lec (1+T^{1/4}) \|(v,b)\|_{\td{\mathcal{E}}_T}\|(u,a)\|_{\td{\mathcal{F}}_T},
}
which imply 
\EQN{
t^{\frac 1 2}\| B_2((v,b),(u,a))\|_{E^\I_q\times E^\I_q} 
\lec (1+T^{1/4}) \|(v,b)\|_{\td{\mathcal{E}}_T}\|(u,a)\|_{\td{\mathcal{F}}_T}.
}
By switching $(v,b),(u,a)$ in the estimates, %
\[
\| B((v,b),(u,a))\|_{\td{\mathcal E}_T} 
\lec (1 + T^{1/4}) \min \bke{ \| (v,b)\|_{\td{\mathcal E}_T}\|(u,a)\|_{\td{\mathcal F}_T},\, \| (u,a)\|_{\td{\mathcal E}_T}\|(v,b)\|_{\td{\mathcal F}_T}}. 
\]

We now return to our main objective: proving that the Picard iterates $\{v^{(n)},b^{(n)}\}$ are uniformly bounded in $\td{\mathcal{E}}_T$.
From the recursive relation, we have
\EQN{
\norm{(v^{(n)}, b^{(n)})}_{\td{\mathcal E}_T} 
&\le \norm{(e^{t\De} v_0, e^{t\De} b_0)}_{\td{\mathcal E}_T} + \norm{B((v^{(n-1)}, b^{(n-1)}), (v^{(n-1)}, b^{(n-1)}))}_{\td{\mathcal E}_T} \\
&\le C_T  \| (v_0,b_0)\|_{E^3_q\times E^3_q}  +  C_T' \norm{(v^{(n-1)}, b^{(n-1)})}_{\td{\mathcal F}_T} \norm{(v^{(n-1)}, b^{(n-1)})}_{\td{\mathcal E}_T}\\
&\le C_T  \| (v_0,b_0)\|_{E^3_q\times E^3_q}  +  C_T' C_T''  \| (v_0,b_0)\|_{E^3_q\times E^3_q} \norm{(v^{(n-1)}, b^{(n-1)})}_{\td{\mathcal E}_T}.
}
Thus, if $\| (v_0\times b_0)\|_{E^3_q\times E^3_q}  \le (2  C_T' C_T'')^{-1}$, then $\norm{(v^{(n)}, b^{(n)})}_{\td{\mathcal E}_T} 
$ is uniformly bounded by $2C_T  \| (v_0, b_0)\|_{E^3_q\times E^3_q} $.
To show that the limit $(v,b)\in \td{\mathcal E}_T$, consider the difference of successive iterates:
\EQS{\label{ineq:difference.picard.iterates}
&\norm{(v^{(n+1)}, b^{(n+1)}) - (v^{(n)}, b^{(n)})}_{\td{\mathcal E}_T} \\
&= \norm{B((v^{(n)}, b^{(n)}), (v^{(n)}, b^{(n)})) - B((v^{(n-1)}, b^{(n-1)}),(v^{(n-1)}, b^{(n-1)}))}_{\td{\mathcal E}_T} \\
&\le \norm{B((v^{(n)}, b^{(n)})-(v^{(n-1)},b^{(n-1)}),(v^{(n)}, b^{(n)}))}_{\td{\mathcal E}_T} + \norm{B((v^{(n-1)}, b^{(n-1)}),(v^{(n)}, b^{(n)})-(v^{(n-1)}, b^{(n-1)}))}_{\td{\mathcal E}_T}\\
&\lec \norm{(v^{(n)}, b^{(n)}) - (v^{(n-1)}, b^{(n-1)})}_{\td{\mathcal F}_T} \bke{\norm{(v^{(n)}, b^{(n)})}_{\td{\mathcal E}_T} + \norm{(v^{(n-1)}, b^{(n-1)})}_{\td{\mathcal E}_T} }
\\&\lec \norm{(v^{(n)}, b^{(n)}) - (v^{(n-1)}, b^{(n-1)})}_{\td{\mathcal E}_T} \bke{\norm{(v^{(n)}, b^{(n)})}_{\td{\mathcal E}_T} + \norm{(v^{(n-1)}, b^{(n-1)})}_{\td{\mathcal E}_T} }.
}
This shows the sequence $\{(v^{(n+1)}, b^{(n+1)})-(v^{(n)}, b^{(n)})\}$ is Cauchy in $\td{\mathcal E}_T$, and the limit of $\{(v^{(n)}, b^{(n)})\}$ lies in $\mathcal E_T$, since convergence in $\mathcal F_T$ implies convergence in the full $\mathcal E_T$-norm.

If $q \le s$, we use the embeddings 
\[
\norm{f}_{L^s_T E^p_m} \le \norm{f}_{L^s_T E^p_q} \le \norm{f}_{E^{s,p}_{T,q}},
\]
from \eqref{ineq:embedding.parabolic.space} to justify the $L^s_T E^p_m$-estimate in \eqref{eq1.7-mhd}.

Next, we prove convergence to the initial data when $q<\I$.
By \cite[Lemma 2.3]{BLT-MathAnn2024} we have
\EQ{\label{limit:zero.heat}
\lim_{T'\to 0^+} \sup_{0<t<T'}t^{\frac 1 4} \| (e^{t\Delta}v_0, e^{t\Delta}b_0)\|_{E^6_q\times E^6_q} = \lim_{T'\to 0} \| (e^{t\Delta}v_0, e^{t\Delta}b_0)\|_{\td{\mathcal F}_{T'}} = 0,
}
whenever $(v_0, b_0)\in E^3_q\times E^3_q$. 
We extend this to the Picard sequence by induction. 
The base case $(v^{(0)}, b^{(0)}) = (e^{t\Delta}v_0, e^{t\Delta}b_0)$ satisfies \eqref{limit:zero.heat}. 
Suppose the inductive hypothesis holds: 
\EQ{\label{limit:inductiveHyp}
\lim_{T'\to 0} \|(v^{(n-1)}, b^{(n-1)})\|_{\td{\mathcal F}_{T'}} = 0.
}
Then, from the iteration estimate in the class $\td {\mathcal F}_{T'}$ where we are taking $T'\leq T$, we have
\EQN{\label{inductiveStep}
\| (v^{(n)}, b^{(n)})\|_{\td {\mathcal F}_{T'}} &\leq \|(e^{t\Delta} v_0, e^{t\Delta} b_0)\|_{\td {\mathcal F}_{T'}}+ \| B((v^{(n-1)}, b^{(n-1)}),(v^{(n-1)}, b^{(n-1)}))\|_{\td {\mathcal F}_{T'}}\\
&\lec \|(e^{t\Delta} v_0, e^{t\Delta} b_0)\|_{\td {\mathcal F}_{T'}}+\| (v^{(n-1)}, b^{(n-1)})\|_{\td {\mathcal F}_{T'}}^2,
}
which implies
\EQ{\label{limit:Picard}
\lim_{T'\to 0} \|(v^{(n)}, b^{(n)})\|_{\td{\mathcal F}_{T'}} = 0
}
 by \eqref{limit:zero.heat} and \eqref{limit:inductiveHyp}. 
 This completes the induction.
 
 Since the Picard sequence converges in $\td {\mathcal F}_T$, the limit $(v,b)$ also satisfies 
  \EQ{\label{limit:zero.ns}
\lim_{T'\to 0^+} \| (v,b)\|_{\td{\mathcal F}_{T'}}= 0.
}
for $T'>0$ sufficiently small.
Using \eqref{ineq:E3qintegral} and \eqref{limit:zero.ns}, we find
\EQN{
\lim_{T'\to 0^+} \sup_{0<t<T'} \| B((v,b),(v,b))\|_{E^3_q\times E^3_q} (t)=0.
}
Combining this with \cite[Lemma 2.3]{BLT-MathAnn2024}, we conclude 
\[
\lim_{t\to 0} \| (v,b)-(v_0,b_0)\|_{E^3_q\times E^3_q}=0.
\] 

If $q=\I$ then we have a weaker mode of convergence. Fix a ball $B$. Take $R>0$ large so that $B\subset B_R(0)$.  We re-write the  bilinear form as  \[B((v,b),(v,b)) = B((v,b),(v\chi_{B_R(0)},b\chi_{B_R(0)})) + B((v,b),(v(1-\chi_{B_R(0)}), b(1-\chi_{B_R(0)}))).\]
If $1< \omega<3$ then we have  
\EQN{
\|   B_1&((v,b),(v\chi_{B_R(0)}, b\chi_{B_R(0)}))  \|_{L^\omega }\\
&\lec \int_0^t \frac 1 {(t-\tau)^{\frac 1 2 + \frac 3 2 (  \frac 2 {3} -\frac 1 \omega )}} \bke{\| |v|^2\chi_{B_R(0)} \|_{L^{3/2}}  + \| |b|^2\chi_{B_R(0)} \|_{L^{3/2}} }(\tau) \,d\tau \\
&\lesssim_R t^{\frac 1 2 - \frac 3 2(\frac 2 3 - \frac 1 \omega)} \| (v,b)\|_{L^\I (L^{3}_\uloc\times L^{3}_\uloc)}^2,
}
\EQN{
\|   B_2&((v,b),(v\chi_{B_R(0)}, b\chi_{B_R(0)}))  \|_{L^\omega }\\
&\lec \int_0^t \frac 1 {(t-\tau)^{\frac 1 2 + \frac 3 2 (  \frac 2 {3} -\frac 1 \omega )}} \bke{\| v\otimes b\chi_{B_R(0)} \|_{L^{3/2}}  + \| b\otimes v\chi_{B_R(0)} \|_{L^{3/2}} }(\tau) \,d\tau \\
&\lec \int_0^t \frac 1 {(t-\tau)^{\frac 1 2 + \frac 3 2 (  \frac 2 {3} -\frac 1 \omega )}} \bke{\| |v|^2\chi_{B_R(0)} \|_{L^{3/2}}  + \| |b|^2\chi_{B_R(0)} \|_{L^{3/2}} }(\tau) \,d\tau \\
&\lesssim_R t^{\frac 1 2 - \frac 3 2(\frac 2 3 - \frac 1 \omega)} \| (v,b)\|_{L^\I (L^{3}_\uloc\times L^{3}_\uloc)}^2,
}
by Young's inequality, so that 
\EQN{
\|   B((v,b),(v\chi_{B_R(0)}, b\chi_{B_R(0)}))  \|_{L^\omega\times L^\omega }
\lesssim_R t^{\frac 1 2 - \frac 3 2(\frac 2 3 - \frac 1 \omega)} \| (v,b)\|_{L^\I (L^{3}_\uloc\times L^{3}_\uloc)}^2.
}
For any $R>0$, the above vanishes as $t\to 0$ provided $\omega<3$.

By taking $R= 2\max_{x\in B}|x|$, we can ensure that for all $x\in B$ and $|y|\geq R$ we have 
$\frac 1 2 |y|\leq |x-y|\leq \frac 3 2 |y|$. Hence, for $x\in B$,
\EQN{
|B&((v,b),(v\chi_{B_R(0)^c}, b\chi_{B_R(0)^c}))(x,t) | \\
&\lesssim  \int_0^t \int_{ y\in B_R(0)^c} \frac 1 {(|x-y|+\sqrt {t-\tau})^4}  (|v|^2+|b|^2)(y,\tau)\,dy\,d\tau\\
&\lesssim  t\sup_{0<\tau<t} \sum_{k=1}^\I  \frac 1 {2^{4k-4}R^4} \int_{  R2^{k-1}\leq  |y| < R2^{k} } (|v|^2+|b|^2)(y,\tau) \,dy
\\&\lesssim t \sup_{0<\tau<t}\sum_{k=1}^\I  \frac {R 2^k} {R^42^{4k-4}}   \bigg(\int_{  |y| < R2^{k} } (|v|^3+|b|^3)(y,\tau) \,dy\bigg)^\frac 2 3
\\&\lesssim t \sum_{k=1}^\I  \frac {R^3 2^{3k} } {R^42^{4k-4}}  \sup_{0<\tau<t}  \| (v, b)(\tau)\|_{L^3_\uloc\times L^3_\uloc}^2
\\&\lesssim \frac t R  \sup_{0<\tau<t} \| (v,b)(\tau)\|_{L^3_\uloc\times L^3_\uloc}^2.
}

Now, 
\EQN{
\| B&((v,b),(v,b)) (t)\|_{L^\omega(B)\times L^\omega(B)} \\
&\lesssim _R  \| B((v,b),(v\chi_{B_R(0)^c},b\chi_{B_R(0)^c})) (t)\|_{L^\omega}  +   \| B((v,b), (v\chi_{B_R(0)^c},b\chi_{B_R(0)^c})) (t)\|_{L^\I(B)} 
\\&\lesssim_R  t^{\frac 1 2 - \frac 3 2(\frac 2 3 - \frac 1 \omega)} \| (v,b)\|_{L^\I (L^{3}_\uloc\times L^3_\uloc)}^2 + t \sup_{0<\tau<t} \| (v,b)(\tau)\|_{L^3_\uloc\times L^3_\uloc}^2.
}
Hence 
\[
\lim_{t\to 0^+}\| B((v,b),(v,b)) (t)\|_{L^\omega(B)\times L^\omega(B)} = 0.
\]
Referring to  \cite[p. 394]{MaTe},
 we have \[\lim_{t\to 0}\|(e^{t\Delta}v_0, e^{t\Delta}b_0) - (v_0,b_0)\|_{L^\omega(B)\times L^\omega(B)} =0.\]   It follows that 
\[
\lim_{t\to 0}\| (v,b)(\cdot,t)-(v_0,b_0)(\cdot)\|_{L^\omega(B)\times L^\omega(B)} =0.
\]

We now prove continuity at positive times. Let $t_1>0$ be fixed. We will send $t\to t_1$.  Note that by \cite[Lemma 2.3]{BLT-MathAnn2024} we have $(e^{t\Delta}v_0, e^{t\Delta}b_0) -(e^{t_1\Delta}v_0, e^{t_1\Delta}b_0)\to 0$ in $E^3_q\times E^3_q$ as $t\to t_1$. We therefore only need to show $B((v,b),(v,b))(t)\to B((v,b),(v,b))(t_1)$. Following \cite[p.~86]{Tsai-book}, we take $\rho$ slightly less than $1$ so that $\rho t_1<t$  and write
\EQN{
B((v,b),(v,b))(t)- B((v,b),(v,b))(t_1) &= \int_{\rho t_1}^t e^{(t-\tau)\Delta} \mathbb P\nb \cdot F\,d\tau
{-} \int_{\rho t_1}^{t_1} e^{(t_1-\tau)\Delta} \mathbb P\nb \cdot F\,d\tau 
\\&+\int_{0}^{\rho t_1} \big(e^{(t-\rho t_1)\Delta}-e^{(t_1-\rho t_1)\Delta} \big)e^{(\rho t_1-\tau)\Delta} \mathbb P\nb \cdot F\,d\tau 
}
where $F=(v\otimes v)(\tau)$, $(b\otimes b)(\tau)$, $(v\otimes b)(\tau)$, or $(b\otimes v)(\tau)$. For the first and second terms,
by \cite[(2.3)]{BLT-MathAnn2024} with $p=\td p = 3$, $\td q = q$ and using the embedding $E^\infty_q\subset E^\infty_\infty$,
we have 
\EQN{
\int_{\rho t_1}^t  \|  e^{(t-\tau)\Delta} \mathbb P\nb\cdot  F\|_{E^3_q}\,d\tau  &\lesssim  \int_{\rho t_1}^{t} \frac 1 {(t-\tau)^{\frac 1 2}\tau^\frac 1 2} \| v,b\|_{E^3_q} (\tau) \tau^{\frac 1 2}\| v,b\|_{E^\I_q} \,d\tau
 \lesssim \frac {(t-\rho t_1)^\frac 1 2}	{(\rho t_1)^\frac 1 2}	\|(v,b)\|_{\mathcal E_T}^2,
}
and 
\EQN{
\int_{\rho t_1}^{t_1} \|e^{(t_1-\tau)\Delta} \mathbb P\nb \cdot F\|_{E^3_q}\,d\tau  &\lesssim  \int_{\rho t_1}^{t_1} \frac 1 {(t_1-\tau)^{\frac 1 2}\tau^\frac 1 2} \| v,b\|_{E^3_q} (\tau) \tau^{\frac 1 2}\| v,b\|_{E^\I_q} \,d\tau
 \lesssim \frac {(t_1-\rho t_1)^\frac 1 2}	{(\rho t_1)^\frac 1 2}	\|(v,b)\|_{\mathcal E_T}^2,
}
both of which can be made arbitrarily small by taking $\rho t_1$ close to $t_1$ and $t$ close to $t_1$.

For the third term we note that by \cite[Lemma 2.3]{BLT-MathAnn2024}, for each $0<\tau<\rho t_1$, we have
\[
\|  \big(e^{(t-\rho t_1)\Delta}-e^{(t_1-\rho t_1)\Delta} \big)e^{(\rho t_1-\tau)\Delta} \mathbb P\nb \cdot F(\tau) \|_{E^3_q}\to 0 \text{ as }t\to t_1,
\]
which follows {(even if $q=\infty$)}  the fact that $e^{(\rho t_1-\tau)\Delta} \mathbb P\nb \cdot F(\tau)\in  E^3_q$, which is a consequence of \cite[Lemma 2.1]{BLT-MathAnn2024}.
Additionally,  
\EQN{
& \|  \big(e^{(t-\rho t_1)\Delta}-e^{(t_1-\rho t_1)\Delta} \big)e^{(\rho t_1-\tau)\Delta} \mathbb P\nb \cdot F(\tau)	\|_{E^3_q}
\\&\lesssim  
 \bigg( \frac 1 {(t-\tau)^\frac 1 2 \tau^\frac 1 2} + \frac 1 {(t_1-\tau)^\frac 1 2 \tau^\frac 1 2}\bigg) \| (v,b)\|_{\mathcal E_T}^2 \in L^1(0,\rho t_1),
}
where integration in   $L^1(0,\rho t_1)$ is with respect to $\tau$.
So, by Lebesgue's dominated convergence theorem, 
\Eq{
\int_0^{\rho t_1 } \|  \big(e^{(t-\rho t_1)\Delta}-e^{(t_1-\rho t_1)\Delta} \big)e^{(\rho t_1-\tau)\Delta} \mathbb P\nb \cdot F(\tau)	\|_{E^3_q}\,d\tau \to 0 \text{ as }t\to t_1.
}
{The above show the continuity of $(v,b)(t)$ at positive times.}

To prove the spacetime integral bound \eqref{eq1.7-mhd} for $p=3$, $s=\infty$, we work in the spaces with the norms
\[
\norm{(v,b)}_{\mathcal E_T^*} = \norm{(v,b)}_{E^{\infty,3}_{T,q}\times E^{\infty,3}_{T,q}} + \norm{(t^{\frac12}v, t^{\frac12}b)}_{E^{\infty,\infty}_{T,q}\times E^{\infty,\infty}_{T,q}},\quad
\norm{(v,b)}_{\mathcal F_T^*} = \norm{(t^{\frac14} v,t^{\frac14} b)}_{E^{\infty,6}_{T,q}\times E^{\infty,6}_{T,q}}.
\]
Note that $\mathcal E_T^* \subset \mathcal F_T^*$.

For the linear estimate in $\mathcal F^*_T$, taking $p=6$ so that $a=1/4$ in \cite[Lemma 3.1]{BLT-MathAnn2024}, we have
\EQS{
\norm{(e^{t\De} v_0,e^{t\De} b_0)}_{\mathcal F^*_T} 
= \norm{(t^{\frac14} e^{t\De} v_0, t^{\frac14} e^{t\De} b_0)}_{E^{\infty,6}_{T,q}\times E^{\infty,6}_{T,q}} 
\le C_1^* (1 + T^{\frac14}) \norm{(v_0,b_0)}_{E^3_q\times E^3_q}.
}

For bilinear estimate, taking $a=b=1/4$, $p=\td p=6$ in \cite[Lemma 3.2]{BLT-MathAnn2024},
\EQN{
\norm{B((v,b),(u,a))}_{\mathcal F_T^*} &= \norm{t^{\frac14}B((v,b),(u,a))}_{E^{\infty,6}_{T,q}\times E^{\infty,6}_{T,q}}\\
&\le C_2^* (1 + T^{\frac34}) \norm{(t^{\frac14}v, t^{\frac14}b)}_{E^{\infty,6}_{T,q}\times E^{\infty,6}_{T,q}} \norm{(t^{\frac14}u, t^{\frac14}a)}_{E^{\infty,6}_{T,q}\times E^{\infty,6}_{T,q}}\\
&= C_2^* (1 + T^{\frac34}) \norm{(v,b)}_{\mathcal F_T^*} \norm{(u,a)}_{\mathcal F_T^*}.
}
By choosing $\norm{(v_0, b_0)}_{E^3_q\times E^3_q}$ small enough so that $\norm{(v_0, b_0)}_{E^3_q\times E^3_q}<\frac1{4C_1^*C_2^*\bke{1+T^{\frac14}}\bke{1+T^{\frac34}}}$, Picard iteration yields a unique mild solution satisfying 
\[
\norm{(v,b)}_{\mathcal F_T^*} \le 2 C_1^*(1+T^{\frac14}) \norm{(v_0, b_0)}_{E^3_q\times E^3_q}.
\]

Now, we claim that any solution $(v,b)\in\mathcal F^*_T$ with sufficiently small $(v_0,b_0)\in E^3_q\times E^3_q$ also belongs to $\mathcal E^*_T$.
By \cite[(3.21)]{BLT-MathAnn2024} of \cite[Lemma 3.1]{BLT-MathAnn2024}, we have
\[
\norm{(e^{t\De} v_0, e^{t\De} b_0)}_{E^{\infty,3}_{T,q}\times E^{\infty,3}_{T,q}} \lec \ln(2+T) \norm{(v_0,b_0)}_{E^3_q\times E^3_q}.
\]
Taking $p=\infty$ so that $a=1/2$ in \cite[Lemma 3.1]{BLT-MathAnn2024}, we also obtain
\[
\norm{(t^{\frac12} v_0, t^{\frac12} b_0)}_{E^{\infty,\infty}_{T,q}} \lec (1+T^\frac12) \norm{(v_0,b_0)}_{E^3_q\times E^3_q}.
\]
Combining both estimates, we conclude:
\EQ{
\norm{(e^{t\De}v_0, e^{t\De}b_0)}_{\mathcal E^*_T} \lec (1+T^\frac12)  \norm{(v_0,b_0)}_{E^3_q\times E^3_q}.
}

Next, we establish the bilinear estimate:
\EQ{ 
\norm{B((v,b),(u,a))}_{\mathcal E^*_T} \lec_T \min\bke{\norm{(v,b)}_{\mathcal E^*_T}\norm{(u,a)}_{\mathcal F^*_T}, \norm{(u,a)}_{\mathcal E^*_T}\norm{(v,b)}_{\mathcal F^*_T}}.
}
Indeed, by applying \cite[Lemma 3.2]{BLT-MathAnn2024} with $(a,b,p,\td p)=(0,1/4,3,6)$ and $(a,b,p,\td p)=(1/2,1/4,\infty,6)$, we obtain
\EQN{
&\norm{B((v,b),(u,a))}_{E^{\infty,3}_{T,q}\times E^{\infty,3}_{T,q}} \\
&\qquad \lec (1+T^{\frac34}) \min\left( \norm{(v,b)}_{E^{\infty,3}_{T,q}\times E^{\infty,3}_{T,q}} \norm{ (t^{\frac14}u,t^{\frac14}a)}_{E^{\infty,6}_{T,q}\times E^{\infty,6}_{T,q}},\right.\\
&\qquad\qquad\qquad\qquad\qquad\qquad\qquad\qquad\left. \norm{ (u,a)}_{E^{\infty,3}_{T,q}\times E^{\infty,3}_{T,q}} \norm{ (t^{\frac14}v,t^{\frac14}b)}_{E^{\infty,6}_{T,q}\times E^{\infty,6}_{T,q}} \right),
}
and
\EQN{
&\norm{t^{\frac12} B((v,b),(u,a))}_{E^{\infty,\infty}_{T,q}\times E^{\infty,\infty}_{T,q}} \\
&\qquad \lec (1+T^{\frac34}) \min\left( \norm{(t^{\frac12}v,t^{\frac12}b)}_{E^{\infty,\infty}_{T,q}\times E^{\infty,\infty}_{T,q}} \norm{ (t^{\frac14}u,t^{\frac14}a)}_{E^{\infty,6}_{T,q}\times E^{\infty,6}_{T,q}}, \right.\\
&\qquad\qquad\qquad\qquad\qquad\qquad\qquad\qquad\left. \norm{(t^{\frac12}u,t^{\frac12}a)}_{E^{\infty,\infty}_{T,q}\times E^{\infty,\infty}_{T,q}} \norm{ (t^{\frac14}v,t^{\frac14}b)}_{E^{\infty,6}_{T,q}\times E^{\infty,6}_{T,q}} \right).
}
Using the same argument as before, we conclude that $(v,b) \in \mathcal E^*_T$, possibly after taking a smaller $T>0$. 
In particular, $(v,b)\in E^{\infty,3}_{T,q}\times E^{\infty,3}_{T,q}$, with the norm controlled by $\norm{(v_0,b_0)}_{E^3_q\times E^3_q}$.
The case $s=\infty$ in Theorem \ref{thrm:critical-mhd} then follows from the embeddings $\norm{(v,b)}_{L^\infty E^3_q\times L^\infty E^3_q}\le \norm{(v,b)}_{E^{\infty,3}_{T,q}\times E^{\infty,3}_{T,q}}$ and $\norm{(t^{\frac12}v, t^{\frac12}b)}_{L^\infty_T(E^\infty_q\times E^\infty_q)}\le \norm{(t^{\frac12}v, t^{\frac12}b)}_{E^{\infty,\infty}_{T,q}\times E^{\infty,\infty}_{T,q}}$.
\qed

\subsection{Mild solutions in critical spaces with enough decay: Proof of Theorem \ref{thrm:critical2-mhd}}\label{sec-critical-II}

The proof of Theorem \ref{thrm:critical2-mhd} is an adaption of the proof of \cite[Theorem 1.3]{BLT-MathAnn2024} for the Navier--Stokes equations to the MHD equations.

By \cite[Lemma 2.1]{BLT-MathAnn2024}, we have for $1\le q \le 3$,
\EQS{\label{u-lin-est}
\sup_{0<t<\I}&\bke{ \norm{(e^{t\De}v_0, e^{t\De}b_0)}_{E^3_q\times E^3_q}  +  t^{\frac12} \norm{(e^{t\De}v_0, e^{t\De}b_0)}_{E^\I_{q_2}\times} +t^{\frac14} \norm{(e^{t\De}v_0, e^{t\De}b_0)}_{E^6_{q_1}\times E^6_{q_1}}} \\
&\qquad\qquad\qquad\qquad\qquad\qquad\qquad\qquad\qquad\qquad\qquad\qquad\qquad\quad\le C \norm{(v_0,b_0)}_{E^3_q\times E^3_q},
}
where
\EQ{
\frac 1{q_1} = \frac 1q - \frac 16, \quad
\frac 1{q_2} = \frac 1q - \frac 13 , \quad
q < q_1 < q_2 \le \oo.
}
For $0<T\le \oo$ and $1\le q \le 3$, let $\mathcal{E}_T$, $\mathcal{F}_T$ be Banach spaces defined as
\EQ{
\mathcal{E}_T 
:= \bket{ (v,b)\in L^\infty(0,T;E^3_q\times E^3_q):\ t^{\frac12} (v,b)(\cdot,t) \in L^\infty(0,T; E^\I_{q_2}\times E^\I_{q_2})
},
}
and
\EQ{
\mathcal{F}_T := \bket{ (v,b) :\ t^{\frac14}(v,b)(\cdot,t) \in L^\infty(0,T; E^6_{q_1}\times E^6_{q_1})} ,
}
with norms
\[
\norm{(v,b)}_{\mathcal{E}_T} 
:= \sup_{0<t<T} \norm{(v,b)(\cdot,t)}_{E^3_q\times E^3_q}  + \sup_{0<t<T} t^{\frac12} \norm{(v,b)(\cdot,t)}_{E^\I_{q_2}\times E^\I_{q_2}},
\]
and
\[
\norm{(v,b)}_{\mathcal{F}_T} 
:= \sup_{0<t<T} t^{\frac14} \norm{(v,b)(\cdot,t)}_{E^6_{q_1}\times E^6_{q_1}},
\]
respectively.
Note that $\mathcal{E}_T \subset \mathcal{F}_T$.

By \cite[Lemma 2.1]{BLT-MathAnn2024} again and {H\"older inequality \eqref{Holder} using} $2q \ge q_1$ due to $q\le 3$, 
\EQS{\label{0819a-B1}
\norm{B_1((v,b),(u,a))}_{E^6_{q_1}}(t) 
&\lec \int_0^t  \frac1{(t-\tau)^{\frac34}} \bke{\norm{v\otimes u(\tau)}_{E^3_{q}} + \norm{b\otimes a(\tau)}_{E^3_{q}} } d\tau\\
&\le \int_0^t  \frac1{(t-\tau)^{\frac34}} \bke{\norm{v(\tau)}_{E^6_{q_1}} \norm{u(\tau)}_{E^6_{q_1}} + \norm{b(\tau)}_{E^6_{q_1}} \norm{a(\tau)}_{E^6_{q_1}} } d\tau\\
&\le \int_0^t  \frac1{(t-\tau)^{\frac34}} \tau^{-1/4}\norm{(v,b)}_{{\mathcal{F}}_T} \tau^{-1/4}\norm{(u,a)}_{{\mathcal{F}}_T}\, d\tau\\
&\lec t^{-1/4} \norm{(v,b)}_{{\mathcal{F}}_T} \norm{(u,a)}_{{\mathcal{F}}_T},
}
\EQS{\label{0819a-B2}
\norm{B_2((v,b),(u,a))}_{E^6_{q_1}}(t) 
&\lec \int_0^t  \frac1{(t-\tau)^{\frac34}} \bke{\norm{v\otimes a(\tau)}_{E^3_{q}} + \norm{b\otimes u(\tau)}_{E^3_{q}} } d\tau\\
&\le \int_0^t  \frac1{(t-\tau)^{\frac34}} \bke{\norm{v(\tau)}_{E^6_{q_1}} \norm{a(\tau)}_{E^6_{q_1}} + \norm{b(\tau)}_{E^6_{q_1}} \norm{u(\tau)}_{E^6_{q_1}} } d\tau\\
&\le \int_0^t  \frac1{(t-\tau)^{\frac34}} \tau^{-1/4}\norm{(v,b)}_{{\mathcal{F}}_T} \tau^{-1/4}\norm{(u,a)}_{{\mathcal{F}}_T}\, d\tau\\
&\lec t^{-1/4} \norm{(v,b)}_{{\mathcal{F}}_T} \norm{(u,a)}_{{\mathcal{F}}_T},
}
so that
\EQ{\label{0819a}
\norm{B((v,b),(u,a))}_{E^6_{q_1}\times E^6_{q_1}}(t) 
\lec t^{-1/4} \norm{(v,b)}_{{\mathcal{F}}_T} \norm{(u,a)}_{{\mathcal{F}}_T}.
}
Hence, %
\[
\| B((v,b),(u,a))\|_{{\mathcal F}_T} \leq c_*   \|(v,b)\|_{{\mathcal{F}}_T} \|(u,a)\|_{{\mathcal{F}}_T},
\]
where $c_*$ is a universal constant.

Concerning the caloric extension of $(v_0,b_0)$, we have for $\|(v_0,b_0)\|_{E^3_q\times E^3_q}$ of any size that
\[
\lim_{T\to 0} \| (e^{t\Delta}v_0, e^{t\Delta}b_0)\|_{{\mathcal F}_T}  = 0,
\]
by \cite[(2.13)]{BLT-MathAnn2024} of \cite[Lemma 2.3]{BLT-MathAnn2024}.  
 Hence, there exists $T=T(u_0)$ so that 
\EQ{\label{small.caloric}
 \| (e^{t\Delta}v_0, e^{t\Delta}b_0)\|_{{\mathcal F}_T} \lesssim c_*^{-1}.
}
If, on the other hand, $\|(v_0,b_0)\|_{E^3_q\times E^3_q} \lesssim c_*^{-1}$, then by \eqref{u-lin-est}, %
we have  \eqref{small.caloric} for $T=\I$.
The Picard contraction theorem then guarantees the existence of a mild solution $(v,b)$ to \eqref{MHD} 
so that
\[
\| (v,b)\|_{\mathcal F_T} \leq  2\| (e^{t\Delta}v_0, e^{t\Delta}b_0)\|_{\mathcal F_T}.
\]
This solution is unique among all mild solutions $(u,a)$ with data $(v_0,b_0)$ satisfying $\|(u,a)\|_{\mathcal F_T}\leq 2\| (e^{t\Delta}v_0, e^{t\Delta}b_0)\|_{\mathcal F_T}$.%

Next, we show that a solution $(v,b)\in \mathcal{F}_T$ with initial data $(v_0,b_0)\in E^3_q\times E^3_q$ also belongs to $\mathcal{E}_T$. Let $\{(v^{(n)}, b^{(n)})\}_{n\ge1}$ be the Picard iteration sequence in $\mathcal{F}_T$. %
By construction, 
\EQ{\label{unFTbound}
\norm{(v^{(n)}, b^{(n)})}_{ {\mathcal{F}}_T}  \le  2 \|(e^{t\Delta}v_0, e^{t\Delta}b_0)\|_{\mathcal F_T}.
}
Note that
\[
\norm{(v^{(n)}, b^{(n)})}_{\mathcal{E}_T} 
\le \norm{(e^{t\De} v_0, e^{t\De} b_0)}_{\mathcal{E}_T} + \norm{B((v^{(n-1)}, b^{(n-1)}), v^{(n-1)}, b^{(n-1)}))}_{\mathcal{E}_T}.
\]
We now bound $B((v,b),(u,a))$ in $\mathcal{E}_T$ in terms of $(v,b)$ and $(u,a)$ in $\mathcal{F}_T$ and $\mathcal{E}_T$. 
We have by \cite[Lemma 2.1]{BLT-MathAnn2024} {and H\"older inequality \eqref{Holder} using $q_1 \ge 6$,} 
\EQS{\label{ineq:E3qintegral2-B1}
\norm{B_1((v,b),(u,a))}_{E^3_{q}}(t) 
&\lec \int_0^t  \frac1{(t-\tau)^{\frac12}} \bke{\norm{v\otimes u(\tau)}_{E^3_{q}} + \norm{b\otimes a(\tau)}_{E^3_{q}} } d\tau\\
&\le \int_0^t  \frac1{(t-\tau)^{\frac12}} \bke{\norm{v(\tau)}_{E^6_{q_1}} \norm{u(\tau)}_{E^6_{q_1}} + \norm{b(\tau)}_{E^6_{q_1}} \norm{a(\tau)}_{E^6_{q_1}} } d\tau\\
&\le \int_0^t  \frac1{(t-\tau)^{\frac12}} \tau^{-1/4}\norm{(v,b)}_{\mathcal{F}_T} \tau^{-1/4}\norm{(u,a)}_{\mathcal{F}_T}\, d\tau\\
&\lec \norm{(v,b)}_{\mathcal{F}_T} \norm{(u,a)}_{\mathcal{F}_T},
}
\EQS{\label{ineq:E3qintegral2-B2}
\norm{B_2((v,b),(u,a))}_{E^3_{q}}(t) 
&\lec \int_0^t  \frac1{(t-\tau)^{\frac12}} \bke{\norm{v\otimes a(\tau)}_{E^3_{q}} + \norm{b\otimes u(\tau)}_{E^3_{q}} } d\tau\\
&\le \int_0^t  \frac1{(t-\tau)^{\frac12}} \bke{\norm{v(\tau)}_{E^6_{q_1}} \norm{a(\tau)}_{E^6_{q_1}} + \norm{b(\tau)}_{E^6_{q_1}} \norm{u(\tau)}_{E^6_{q_1}} } d\tau\\
&\le \int_0^t  \frac1{(t-\tau)^{\frac12}} \tau^{-1/4}\norm{(v,b)}_{\mathcal{F}_T} \tau^{-1/4}\norm{(u,a)}_{\mathcal{F}_T}\, d\tau\\
&\lec \norm{(v,b)}_{\mathcal{F}_T} \norm{(u,a)}_{\mathcal{F}_T},
}
so that 
\EQ{\label{ineq:E3qintegral2}
\norm{B((v,b),(u,a))}_{E^3_{q}\times E^3_{q}}(t) 
\lec \norm{(v,b)}_{\mathcal{F}_T} \norm{(u,a)}_{\mathcal{F}_T}.
}
{By $\mathcal{E}_T \subset \mathcal{F}_T$,}
we have
\[
\norm{B((v,b),(u,a))}_{E^3_{q}\times E^3_{q}}(t) \lec \norm{(v,b)}_{\mathcal{E}_T} \norm{(u,a)}_{\mathcal{F}_T} \wedge \norm{(u,a)}_{\mathcal{E}_T} \norm{(v,b)}_{\mathcal{F}_T}.
\]

Also by \cite[Lemma 2.1]{BLT-MathAnn2024} {and H\"older inequality \eqref{Holder},}
\begin{align}
\nonumber
\| B_1((v,b),(u,a))\|_{E^\I_{q_2}}(t) &\lec\int_0^t \frac 1 {(t-\tau)^{\frac 34 }} \bke{\| v\otimes u\|_{E^6_{q_1}}(\tau) + \| b\otimes a\|_{E^6_{q_1}}(\tau)} d\tau
\\&\lec \int_0^t\frac 1 {(t-\tau)^{\frac 34 }\tau^{3/4}} 
 ( \tau^{1/2}	\|v\|_{E^\I_{q_2}}\tau^{1/4}\|u\|_{E^6_{q_1}} \wedge \tau^{1/2}\|v\|_{E^\I_{q_2}}\tau^{1/4}\|u\|_{E^6_{q_1}}
 \nonumber
\\&\qquad\qquad\qquad\qquad\quad + \tau^{1/2}	\|b\|_{E^\I_{q_2}}\tau^{1/4}\|a\|_{E^6_{q_1}} \wedge \tau^{1/2}\|b\|_{E^\I_{q_2}}\tau^{1/4}\|a\|_{E^6_{q_1}} )\,d\tau
\nonumber
\\&\lec t^{-\frac 1 2}(\norm{(v,b)}_{\mathcal{E}_T} \norm{(u,a)}_{\mathcal{F}_T} \wedge \norm{(u,a)}_{\mathcal{E}_T} \norm{(v,b)}_{\mathcal{F}_T}),
\label{ineq:E3qintegral3-B1}
\end{align}
\begin{align}
\nonumber
\| B_2((v,b),(u,a))\|_{E^\I_{q_2}}(t) &\lec\int_0^t \frac 1 {(t-\tau)^{\frac 34 }} \bke{\| v\otimes a\|_{E^6_{q_1}}(\tau) + \| b\otimes u\|_{E^6_{q_1}}(\tau)} d\tau
\\&\lec \int_0^t\frac 1 {(t-\tau)^{\frac 34 }\tau^{3/4}} 
 ( \tau^{1/2}	\|v\|_{E^\I_{q_2}}\tau^{1/4}\|a\|_{E^6_{q_1}} \wedge \tau^{1/2}\|v\|_{E^\I_{q_2}}\tau^{1/4}\|a\|_{E^6_{q_1}}
 \nonumber
\\&\qquad\qquad\qquad\qquad\quad + \tau^{1/2}	\|b\|_{E^\I_{q_2}}\tau^{1/4}\|u\|_{E^6_{q_1}} \wedge \tau^{1/2}\|b\|_{E^\I_{q_2}}\tau^{1/4}\|u\|_{E^6_{q_1}} )\,d\tau
\nonumber
\\&\lec t^{-\frac 1 2}(\norm{(v,b)}_{\mathcal{E}_T} \norm{(u,a)}_{\mathcal{F}_T} \wedge \norm{(u,a)}_{\mathcal{E}_T} \norm{(v,b)}_{\mathcal{F}_T}),
\label{ineq:E3qintegral3-B2}
\end{align}
so that 
\begin{align}
\nonumber
\| B((v,b),(u,a))\|_{E^\I_{q_2}\times E^\I_{q_2}}(t) 
\lec t^{-\frac 1 2}(\norm{(v,b)}_{\mathcal{E}_T} \norm{(u,a)}_{\mathcal{F}_T} \wedge \norm{(u,a)}_{\mathcal{E}_T} \norm{(v,b)}_{\mathcal{F}_T}).
\label{ineq:E3qintegral3}
\end{align}

Based on the above estimates we conclude
\EQ{\label{eq3.38}
\| B((v,b),(u,a))\|_{\mathcal E_T} \lec \norm{(v,b)}_{\mathcal{E}_T} \norm{(u,a)}_{\mathcal{F}_T} \wedge \norm{(u,a)}_{\mathcal{E}_T} \norm{(v,b)}_{\mathcal{F}_T}.
}

We can now conclude that $\{(v^{(n)}, b^{(n)})\}$ is Cauchy in $\mathcal{E}_T$ by the calculation preceding and including \eqref{ineq:difference.picard.iterates}. {However, the smallness of the constant is now provided by \eqref{small.caloric}-\eqref{unFTbound}, not by $\norm{(v_0,b_0)}_{E^3_q\times E^3_q}$.}

We now show continuity. For small data, we can try to inherit continuity from Theorem \ref{thrm:critical-mhd}. But we will provide a proof valid for general data.
We first address convergence to the initial data.
By \cite[Lemma 2.3]{BLT-MathAnn2024} we have
\EQ{\label{limit:zero.heat2}
\lim_{T'\to 0^+} \sup_{0<t<T'}t^{\frac 1 4} \| (e^{t\Delta}v_0, e^{t\Delta}b_0)\|_{E^6_{q_1}\times E^6_{q_1}} = \lim_{T'\to 0} \| (e^{t\Delta}v_0, e^{t\Delta}b_0)\|_{{\mathcal F}_{T'}} = 0,
}
whenever $(v_0,b_0)\in E^3_q\times E^3_q$. By our estimates in the class $ {\mathcal F}_{T'}$ where we are taking $T'\leq T$, we have 
\EQN{
\| (v^{(n)}, b^{(n)})\|_{{\mathcal F}_{T'}} &\leq \|(e^{t\Delta} v_0,e^{t\Delta} b_0)\|_{ {\mathcal F}_{T'}}+ \| B((v^{(n-1)},b^{(n-1)}),(v^{(n-1)},b^{(n-1)}))\|_{ {\mathcal F}_{T'}}\\
&\lec \|(e^{t\Delta} v_0, e^{t\Delta} b_0)\|_{ {\mathcal F}_{T'}}+\| (v^{(n-1)},b^{(n-1)})\|_{ {\mathcal F}_{T'}}^2.
}
From this and by induction, for any $n$ we have 
\EQN{
\lim_{T'\to 0^+} \| (v^{(n)}, b^{(n)})\|_{{\mathcal F}_{T'}} = 0.
}
The limit \eqref{limit:zero.heat2}, convergence of the Picard iterates in $ {\mathcal F}_T$ and the above inequality imply that, by taking $T'$ small, we can make $\sup_{0<t<T'}t^{\frac 1 4} \| (v,b)(t)\|_{E^6_{q_1}\times E^6_{q_1}}$ small. To elaborate,  we have 
\EQ{
\| (v,b)\|_{ {\mathcal F}_{T'}} \leq \| (v,b) - (v^{(n)},b^{(n)})\|_{ {\mathcal F}_{T'}}  + \| (v^{(n)},b^{(n)})\|_{ {\mathcal F}_{T'}}.
}
We may choose $n$ large so that the first term is small and then make the second term small by taking $T'$ small.
Hence,
\EQ{\label{limit:zero.ns2}
\lim_{T'\to 0^+} { \| (v,b)\|_{{\mathcal F}_{T'}}}= 0.
}
Using \eqref{ineq:E3qintegral2}, this implies  
\EQN{
\lim_{T'\to 0^+} \sup_{0<t<T'} \| B((v,b),(v,b))\|_{E^3_q\times v} (t)=0.
}
This and  \cite[Lemma 2.3]{BLT-MathAnn2024} imply
\[
\lim_{t\to 0} \| (v,b)-(v_0,b_0)\|_{E^3_q\times E^3_q}=0.
\] 

The proof of continuity for positive times follows from the same argument used in the proof of Theorem \ref{thrm:critical-mhd} and is therefore omitted for brevity.

We now prove the spacetime integral bound \eqref{eq-thmIII-Ebound-mhd} for $p\in (3,9]$ and $\frac2s + \frac3p = 1$.  Note that we exclude $p=3$, i.e., $s=\infty$.
By imbedding $E^p_q\subset E^p_m$ for $m>q$, we may assume $m<\infty$. (We do not take $m=q$ since we need $q<m$ for global existence).\
Denote the Banach space
\[
X_T = \mathcal E_T \cap (E^{s,p}_{T,m}\times E^{s,p}_{T,m}).
\]
For the linear term, by \cite[Lemma 2.1]{BLT-MathAnn2024} and \cite[Lemma 2.4]{BLT-MathAnn2024}, 
\EQS{\label{beta0}
\norm{(e^{t\De} v_0, e^{t\De} b_0)}_{X_T} 
&= \sup_{0<t<T} \norm{(e^{t\De} v_0,e^{t\De} b_0)}_{E^3_q\times E^3_q}\\
&\quad  + \sup_{0<t<T} t^{\frac12}\norm{(e^{t\De} v_0, e^{t\De} b_0)}_{E^\infty_{q_2}\times E^\infty_{q_2}} + \norm{(e^{t\De} v_0, e^{t\De} b_0)}_{E^{s,p}_{T,m}\times E^{s,p}_{T,m}}\\
&\le C_3(1+T^{\beta_0})\norm{(v_0,b_0)}_{E^3_q\times E^3_q},
}
for any $\be_0 \in [0,\infty)$ and $\be_0 > \al_0=\frac 3{2m}-\frac3{2q}+\frac 1s$. 
Note that
\EQ{\label{smallspacetimeintegral}
\lim_{T \to 0_+ } \norm{(e^{t\De} v_0, e^{t\De} b_0)}_{E^{s,p}_{T,m}\times E^{s,p}_{T,m}} = 0. 
}

For the bilinear term, by \cite[Lemma 2.7]{BLT-MathAnn2024} with $\td s=s/2$, $\td p=p/2$, and $\td m = \max(1, m/2)$, so that
\EQ{
\si=0, \quad
\al=
\left\{
\begin{aligned}
\tfrac12 - \tfrac3{2m} - \tfrac1s ,\quad &\text{if}\quad 2 \le m \le \infty,\\[1mm]
-1 + \tfrac3{2m} - \tfrac1s,\quad &\text{if}\quad 1<m<2,
\end{aligned}
\right. \quad \al < 1-\tfrac1s,
}
we have
\EQN{
\norm{B((v,b),(u,a))}_{E^{s,p}_{T,m}\times E^{s,p}_{T,m}}
&\le C_4 (1+ T^{\be}) \bke{\norm{v\otimes u}_{E^{\frac{s}2,\frac p2}_{T,\td m}} + \norm{b\otimes a}_{E^{\frac{s}2,\frac p2}_{T,\td m}} + \norm{v\otimes a}_{E^{\frac{s}2,\frac p2}_{T,\td m}} + \norm{b\otimes u}_{E^{\frac{s}2,\frac p2}_{T,\td m}}}
}
for any $\be \in [0,1-\frac 1s]$ and $\be > \al$. Note
\[
\norm{f\otimes g}_{E^{\frac s2,\frac p2}_{T,\td m}}
\le \norm{f}_{E^{s,p}_{T,2\td m}} \norm{g}_{E^{s,p}_{T,2\td m}}\le \norm{f}_{E^{s,p}_{T, m}} \norm{g}_{E^{s,p}_{T, m}}
\]
no matter $m\ge2$ or $1<m<2$.
We conclude, also using \eqref{eq3.38},
\EQS{\label{beta1}
\norm{B((v,b),(u,a))}_{E^{s,p}_{T, m}\times E^{s,p}_{T, m}} &\leq  C_4 (   1+ T^{\be}   ) \norm{(v,b)}_{E^{s,p}_{T, m}\times E^{s,p}_{T, m}} \norm{(u,a)}_{E^{s,p}_{T, m}\times E^{s,p}_{T, m}},
\\
\norm{B((v,b),(u,a))}_{X_T} &\leq  2C_4 (   1+ T^{\be}   ) \norm{(v,b)}_{X_T} \norm{(u,a)}_{X_T}.
}

By \eqref{smallspacetimeintegral}, we can find $T_1 \in (0,T]$ so that
\[
\norm{(e^{t\De} v_0, e^{t\De} b_0)}_{E^{s,p}_{T_1,m}\times E^{s,p}_{T_1,m}}  \le \de= [4C_4 (   1+ T^{\be}   )]^{-1}.
\]
Then the Picard sequence $(v^{(k)}, b^{(k)})$ satisfies $\norm{(v^{(k)},b^{(k)})}_{E^{s,p}_{T_1,m}\times E^{s,p}_{T_1,m}}  \le 2\de$ for all $k \in \NN$, and we get
$\norm{(v,b)}_{E^{s,p}_{T_1,m}\times E^{s,p}_{T_1,m}}  \le 2\de$.
Thus, $(v,b)$ satisfies the spacetime integral bound \eqref{eq-thmIII-Ebound-mhd}.

We now establish the global $E^{s,p}_{T,m}$-estimates when $(v_0,b_0)$ is sufficiently small in $E^3_q\times E^3_q$.
To this end, we aim to eliminate the dependence of constants on $T$, i.e., we choose $\be_0=0$ in \eqref{beta0} and $\be=0$ in \eqref{beta1}.

Let us analyze the conditions under which $\be=0$ in \eqref{beta1} is permissible. 
When $m\ge2$, we additionally assume that $\frac2s + \frac3m \ge 1$ so that the exponent $\al\le 0$. 
In particular, we can take $\be=0$ when $\al=0$ since the condition $1<\td m < m< \infty$ required in \cite[Lemma 2.7]{BLT-MathAnn2024} is satisfied. 
Note that when $m=2$, we have $\td m=1$, but then $\al<0$.

For $1<m<2$,  we impose the additional condition $\frac3m<\frac2s + 2=3-\frac3p$, which ensures $\al<0$ and hence again allows us to take $\be=0$.

With $\be_0=\be=0$ in \eqref{beta0} and \eqref{beta1}, we obtain a global-in-time estimate for $(v,b)$ in $E^{s,p}_{T=\infty,m}\times E^{s,p}_{T=\infty,m}$, provided that the initial data satisfies $ \norm{(v_0,b_0)}_{E^p_q\times E^p_q} < [32C_3 C_4 ]^{-1}$.

Observe that when $m \ge 2$, all required conditions--including the upper bound $\frac2s + \frac3m \ge 1$-- are satisfied if we take $m=p$.
Once we have established $(v,b) \in E^{s,p}_{T=\infty,m}\times E^{s,p}_{T=\infty,m}$ for some $m$, then the inclusion property implies $(v,b) \in E^{s,p}_{T=\infty,\td m}\times E^{s,p}_{T=\infty,\td m}$ for all $\td m \in [m,\infty]$.
Thus, the condition $\frac2s + \frac3m \ge 1$ can be removed entirely.

 We now consider the $L^s_T E^p_m$-estimates of $(v,b)$, restricting to $T=\infty$ for simplicity.
 Fix $s\in [3,\infty)$ and $q\in[1,3]$, and define $p$ by the relation $\frac3p+\frac 2s=1$. 
Using \cite[Lemma 2.8]{BLT-MathAnn2024} with
$\td p = p/2$, $\td s=s/2$ and $\td m\ge1$ such that $\frac1{\td m} - \frac1m=\frac1{\td p} - \frac1p = \frac 1p$, and by applying H\"older inequality, we obtain
\EQN{
\norm{B((v,b),(u,a))}_{L^s_T(E^p_m\times E^p_m)} 
&\lec \norm{v\otimes u}_{L^{\frac{s}2}_TE^{\frac{p}2}_{\td m}} + \norm{b\otimes a}_{L^{\frac{s}2}_TE^{\frac{p}2}_{\td m}} + \norm{v\otimes a}_{L^{\frac{s}2}_TE^{\frac{p}2}_{\td m}} + \norm{b\otimes u}_{L^{\frac{s}2}_TE^{\frac{p}2}_{\td m}}\\
&\lec \norm{(v,b)}_{L^s_T(E^p_m\times E^p_m)} \norm{(u,a)}_{L^s_T(E^p_{p}\times E^p_{p})}
\\
&\lec \norm{(v,b)}_{L^s_T(E^p_m\times E^p_m)} \norm{(u,a)}_{L^s_T(E^p_{p}\times E^p_{p})},\ \text{ if }p\ge m.
}
The condition $\td m \ge 1$ is equivalent to $m \ge p'{=\frac{p}{p-1}}$. 
Hence,
\EQ{\label{eq3.47}
\norm{B((v,b),(u,a))}_{L^s_T(E^p_m\times E^p_m)} \lec \norm{(v,b)}_{L^s_T(E^p_m\times E^p_m)} \norm{(u,a)}_{L^s_T(E^p_m\times E^p_m)},\ \text{ if }p'\le m\le p.
}

Let $\cM(s,q)$ denote the set of all $m$ for which we can establish $(v,b)\in L^s (E^p_m\times E^p_m)$.
Since $L^s E^p_m \subset L^s E^p_{ m_2}$ for $m <  m_2$, the set $\cM(s,q)$, if nonempty, must be an interval of the form
\EQ{\label{m-region}
\underline m < m \le \infty, \quad \text{or}\quad \underline m \le  m \le \infty,
}
for some $\underline m =\underline m(s,q)\in [1,\infty]$. 

Define $m_1$ by
\[
\frac 1{m_1} = \frac 1q - \frac 2{3s}.
\]
Given $s<\infty$ and $q\le 3$, and recalling that
$\frac 1{p} = \frac 13 - \frac 2{3s}$,
we deduce that 
\EQ{\label{m1.range}
q < m_1 \le p.
}

From \cite[Lemma 2.4]{BLT-MathAnn2024} with $d=r=3$, we have
\EQ{\label{eq3.50}
\norm{e^{t\De} f}_{L^s_{T=\infty} E^p_{m} }  
\lesssim  \|f\|_{E^3_q},
}
with constants independent of $R$, provided one of the following holds:
\EN{
\item [A.] $m_1 \le m\le s$, (and $m_1<m\le s$ if $q=1$), %

\item [E$_1$.] $ s<p=m$,  

\item [E$_2$.] $s<m$ and $\frac 1q \ge\frac{5}{3s}$.
}
If the linear estimate \eqref{eq3.50} holds, and if $p'\le m \le p$ so that the bilinear estimate \eqref{eq3.47} holds, then the Picard iteration $(v^{(n)}, b^{(n)})$ converges in $L^s (E^p_m\times E^p_m)$ for sufficiently small initial data.

\noindent{\bf Case A}: This case applies as soon as $m_1 \le s$, i.e., $\frac 1q \ge\frac{5}{3s}$. Strict inequality $m_1<s$ holds when $q=1$.
By \eqref{m1.range} and since $p'<2<s$, the value
$$m^*(s,q)=\max(p', m_1)$$ 
belongs to both $[p',p]$ and $[m_1,s]$.
Thus, $(v^{(n)}, b^{(n)})$ converges in $L^s (E^p_m\times E^p_m)$ for $m=m^*$, or for $m$ slightly larger than $m^*$ when $q=1$ and $3 \le s\le 4$.

\noindent{\bf Case E$_1$}: This case applies when $3\le s<5$, allowing us to take $m=p$. It thus covers the parameter range $1\le q \le 3 \le s <5$, and $\frac 1q <\frac{5}{3s}$.

\noindent{\bf Case E$_2$}: While this case also requires $\frac 1q \ge\frac{5}{3s}$, it does not yield smaller admissible $m$ than Case A.

This completes the proof of $L^s E^p_m$-estimates, and concludes the proof of
Theorem \ref{thrm:critical2-mhd}.
\qed

\section{Local energy solutions in Wiener amalgam spaces}\label{sec-weak}

In this section, we address weak solutions and establish Theorems \ref{bt4-thm-1.3-mhd}, \ref{thm-1.4BT4-mhd}, \ref{th4.8-mhd} for the MHD equations \eqref{MHD}, along with Theorems \ref{bt4-thm-1.3-vnsed}, \ref{thm-1.4BT4-vnsed}, \ref{th4.8-vnsed} for the viscoelastic Navier--Stokes equations with damping \eqref{vNSEd}.
Given the similarity between the structures of \eqref{vNSEd} and \eqref{MHD}, we focus on presenting the proofs of Theorems \ref{bt4-thm-1.3-mhd}, \ref{thm-1.4BT4-mhd}, \ref{th4.8-mhd}  for \eqref{MHD}.
The details of verification of Theorems \ref{bt4-thm-1.3-vnsed}, \ref{thm-1.4BT4-vnsed}, \ref{th4.8-vnsed} for \eqref{vNSEd} are left to the readers.

Define
\[
N_R^0(v_0,b_0) = \sup_{x_0\in\R^3} \frac1R \int_{B_R(x_0)} \bke{|v_0|^2 + |b_0|^2} dx,
\]
and
\EQN{
N_{q,R}^0(v_0,b_0) &=  \frac1R\bkt{\sum_{k\in\ZZ^3} \bke{\int_{B_R(kR)} \bke{|v_0|^2 + |b_0|^2} dx}^{q/2} }^{2/q},\\
N_{\infty,R}^0(v_0,b_0) &= N_R^0(v_0,b_0):= \sup_{x_0\in\R^3} \frac1R \int_{B_R(x_0)} \bke{|v_0|^2 + |b_0|^2} dx.
}

\begin{lemma}\label{lem2.1-BT-CPDE2020}
Let $v_0,b_0\in L^2_\uloc$ be divergence free, and assume $(v,b)\in\mathcal{N}_{\rm MHD}(v_0,b_0)$. 
For all $r>0$ we have
\EQ{\label{eq2.1-BT-CPDE2020}
\esssup_{0\le t\le\si r^2} \sup_{x_0\in\R^3} \int_{B_r(x_0)} \frac{|v|^2 + |b|^2}2\, dxdt + \sup_{x_0\in\R^3} \int_0^{\si r^2} \int_{B_r(x_0)} \bke{|\nb v|^2 + |\nb b|^2} dxdt < CA_0(r),
}
\EQ{\label{eq2.2-BT-CPDE2020}
\sup_{x_0\in\R^3} \int_0^{\si r^2} \int_{B_r(x_0)} \bke{|v|^3 + |b|^3 + |p - c_{x_0,r}(t)|^{3/2}} dxdt < Cr^{\frac12}(A_0(r))^{\frac32},
}
where
\[
A_0(r) = rN_r^0(v_0,b_0) = \sup_{x_0\in\R^3} \int_{B_r(x_0)} \bke{|v_0|^2 + |b_0|^2} dx,
\]
and
\EQ{\label{eq2.3-BT-CPDE2020}
\si = \si(r) = c_0\min\bket{(N_r^0(v_0,b_0))^{-2},1},
}
for a small universal constant $c_0>0$.
\end{lemma}

\begin{proof}
The proof of the lemma is an adaption of the proof of \cite[Lemma 2.1]{BT-CPDE2020} for the Navier--Stokes equations to the MHD equations.

By H\"older and Young inequalities, for any $\de>0$, we have
\[
\norm{v}_{L^3(0,T;L^3)}^3 \lec \norm{v}_{L^6(0,T;L^2)}^{3/2} \norm{v}^{3/2}_{L^2(0,T;L^6)} \lec (\de R)^{-3} \norm{v}_{L^6(0,T;L^2)}^6 + \de R\norm{v}_{L^2(0,T;L^6)}^2,
\]
and
\[
\norm{b}_{L^3(0,T;L^3)}^3 \lec \norm{b}_{L^6(0,T;L^2)}^{3/2} \norm{b}^{3/2}_{L^2(0,T;L^6)} \lec (\de R)^{-3} \norm{b}_{L^6(0,T;L^2)}^6 + \de R\norm{b}_{L^2(0,T;L^6)}^2.
\]
In addition, applying Sobolev inequality yields
\EQN{
\frac1R \int_0^{\si R^2} \int_{B_{2R}(x_0)} |v|^3\, dxdt
&\le \frac{C}{\de^3R^4} \int_0^{\si R^2} \bke{\int_{B_{2R}(x_0)} |v|^2\, dx}^3 dt + \frac{C\de}{R^2} \int_0^{\si R^2} \int_{B_{2R}(x_0)} |v|^2\, dxdt\\
&\quad + C\de\sup_{x_0\in\R^3} \int_0^{\si R^2} \int_{B_{2R}(x_0)} |\nb v|^2\, dxdt,
}
and
\EQN{
\frac1R \int_0^{\si R^2} \int_{B_{2R}(x_0)} |b|^3\, dxdt
&\le \frac{C}{\de^3R^4} \int_0^{\si R^2} \bke{\int_{B_{2R}(x_0)} |b|^2\, dx}^3 dt + \frac{C\de}{R^2} \int_0^{\si R^2} \int_{B_{2R}(x_0)} |b|^2\, dxdt\\
&\quad + C\de\sup_{x_0\in\R^3} \int_0^{\si R^2} \int_{B_{2R}(x_0)} |\nb b|^2\, dxdt,
}
where $C$ is a positive constant independent of $\si$.
For the pressure term, using the local pressure expansion \eqref{eq-pi-formula-mhd}, we obtain
\EQ{
\frac1R \int_0^{\si R^2} \int_{B_{2R}(x_0)} |\pi - c_{x_0,R}(t)|^{3/2}
\le \frac{C}{R} \int_0^{\si R^2} \int_{B_{4R}(x_0)} \bke{|v|^3 + |b|^3} dxdt + \frac{C}{R^4} \int_0^{\si R^2} \bkt{\overline{A}\bke{\frac{t}{R^2}}}^{3/2} dt,
}
where
\[
\overline{A}(\si) = \esssup_{0\le t\le \si R^2} \sup_{x_0\in\R^3} \int_{\R^3} \frac{|v|^2 + |b|^2}2\, \phi(x-x_0)\, dx.
\]
Applying the local energy inequality, we deduce
\EQS{
\int_{\R^3} \frac{|v|^2+|b|^2}2 \phi(x-x_0)\, dx& + \int_0^t \int_{\R^3} \bke{|\nb v|^2 + |\nb b|^2} \phi(x-x_0)\, dxds\\
&\le \frac{A_0(R)}2 + \frac{C}{R^2} \int_0^{\si R^2} \overline{A}\bke{\frac{s}{R^2}} ds + \frac{C}{R^4} \int_0^{\si R^2} \bke{\overline{A}\bke{\frac{s}{R^2}}}^3 ds,
}
for sufficiently small $\de$.
Therefore, we conclude that
\EQ{
\frac{\overline{A}(\si)}R \le \frac{A_0(R)}{2R} + \frac{C}{R^2} \int_0^{\si R^2} \frac{\overline{A}\bke{\frac{s}{R^2}}}R\, ds + \frac{C}{R^2} \int_0^{\si R^2} \bke{\frac{\overline{A}\bke{\frac{s}{R^2}}}R}^2 ds.
}
Making the change of variables $\tau = s/R^2$ and applying Gr\"onwall's inequality \cite[Lemma 2.2]{BT-CPDE2020}, we derive
\[
\overline{A}(\si) \le A_0(R),
\]
for $t\in[0,T_R]$ where $T_R = \si R^2$, and 
\[
\si = c_0\min\bket{(N_R^0(v_0,b_0))^{-2},1}
\]
for some small constant $c_0$ independent of $R$ and $(v_0,b_0)$.
Note that $\overline{A}(\si)$ is nondecreasing in $\si$, and its continuity in $\si$ follows directly from the local energy inequality. With the above estimate for $\overline{A}(\si)$, the lemma follows by the standard \emph{continuation in $\si$} argument, provided that $c_0$ is chosen sufficiently small.
\end{proof}

We recall a local regularity criterion for \emph{suitable weak solutions} that is a replacement for the case of the Navier--Stokes equations proposed in \cite{CKN-CPAM1982} (see also \cite{Lin-CPAM1998, NRS-Acta1996}).
See \cite{HX-JFA2005} for the definition of suitable weak solutions.

\begin{lemma}[$\ep$-regularity criterion {\cite[Theorem 3.1]{MNS-JMS2007}}]\label{lem2.3-BT-CPDE2020}
There exists a universal constant $\ep_*>0$ such that, if $(v,b,\pi)$ is a suitable weak solution of \eqref{MHD} in $Q_r = Q_r(x_0,t_0) = B_r(x_0)\times(t_0-r^2, t_0)$, $B_r(x_0)\subset\R^3$, and
\[
\ep^3 = \frac1{r^2} \int_{Q_r} \bke{|v|^3 + |v|^3 + |\pi|^{3/2}} dxdt < \ep_*,
\]
then $v$ and $b$ are H\"older continuous on $\overline{Q_{r/2}}$.
\end{lemma}

The corresponding $\ep$-regularity criterion for  weak solutions of the viscoelastic Navier--Stokes equations with damping, \eqref{vNSEd}, is established in \cite[Proposition 3.2]{Hynd-SIMA2013}.

\begin{theorem}[Initial and eventual regularity]\label{thm1.2-BT-CPDE2020}
There is a small positive constant $\ep_1$ such that the following holds.
Assume that $v_0, b_0\in L^2_\uloc(\R^3)$ are divergence free and that $(v,b)\in\mathcal{N}_{\rm MHD}(v_0,b_0)$.
Let
\[
N_R^0 := \sup_{x_0\in\R^3} \frac1R \int_{B_R(x_0)} \bke{|v_0|^2 + |b_0|^2} dx.
\]

1. If there exists $R_0>0$ so that
\[
\sup_{R\ge R_0} N_R^0 < \ep_1,
\]
then $(v,b)$ has eventual regularity.
Moreover, if $R_0^2\lec t$, then
\[
t^{1/2} \norm{(v,b)}_{L^\infty\times L^\infty} \lec \bke{\sup_{R\ge R_0} N_R^0}^{1/2} < \infty.
\]

2. If there exists $R_0>0$ so that 
\[
\sup_{R\le R_0} N_R^0 < \ep_1,
\]
then $(v,b)$ has initial regularity. 
Moreover, if $t\le c_0R_0^2$, then
\[
t^{1/2}\norm{(v,b)(\cdot,t)}_{L^\infty\times L^\infty} \lec \bke{\sup_{R\le R_0}N_R^0}^{1/2} <\infty.
\]

3. If $(v_0,b_0)$ satisfies
\[
\sup_{R>0}N_R^0<\ep_1,
\]
then the set of singular times of $(v,b)$ in $\R^3\times(0,\infty)$ is empty.
Moreover, for all $t>0$,
\[
t^{1/2}\norm{(v,b)(\cdot,t)}_{L^\infty\times L^\infty} \lec \bke{\sup_{R>0}N_R^0}^{1/2} < \infty.
\]
\end{theorem}

\begin{proof}
The proof of the theorem is an adaption of the proof of \cite[Theorem 1.2]{BT-CPDE2020} for the Navier--Stokes equations to the MHD equations.

Assume there exists $R_0>0$ such that for all $R\ge R_0$, we have $N_R^0(v_0,b_0)<\ep_1$, where $\ep_1\in(0,1)$ is a small constant to be determined.

Fix $x_0\in\R^3$ and $R>R_0$. Define $\td p(x,t) = p(x,t) - c_{x_0,R}(t)$ where $c_{x_0,R}(t)$ is the function of $t$ from the local pressure expansion \eqref{eq-pi-formula-mhd}.
Then $(v,b)$ is a suitable weak solution to \eqref{MHD} with associated pressure $\td p$.
By the estimate \eqref{eq2.2-BT-CPDE2020}, we have
\[
\int_0^{\si(R)R^2} \int_{B_R(x_0)} \bke{|v|^3 + |b|^3 + |\td p|^{3/2}} dxdt \le C(N_R^0)^{\frac32} R^2.
\]
Thus, if $R\ge R_0$ and $\ep_1\le\bke{c_0C^{-1}\ep_*}^{2/3}$, then the right side is bounded by $\ep_*$, and we may apply Lemma \ref{lem2.3-BT-CPDE2020}.
It follows that
\[
v,b \in L^\infty(Q),\quad Q = B_{\frac{c_0^{1/2}R}2}(x_0)\times\bkt{\frac{3c_0R^2}4, c_0R^2},
\]
and for $(x,t)\in Q$,
\EQ{\label{eq2.10-BT-CPDE2020}
|v(x,t)| + |b(x,t)| \le C_0\bke{\frac{C}{c_0}(N_R^0)^{3/2}}^{\frac13} \bke{\frac{c_0^{1/2}R}2}^{-1} 
\le C(N_R^0)^{\frac12}t^{-\frac12}.
}

Hence, $(v,b)$ is regular in $\R^3\times\left(\frac{3c_0R^2}4, c_0R^2\right]$.
Since $R\ge R_0$ is arbitrary, we conclude that $(v,b)$ is regular at every point $(x,t)\in\R^3\times\left(\frac{3c_0R^2}4, \infty\right)$, with the bound given by \eqref{eq2.10-BT-CPDE2020}. 
Note that this threshold $\frac{3c_0R_0^2}4$ depends only on $(v_0,b_0)$ and is uniform for all $(v,b)\in\mathcal{N}_{\rm MHD}(v_0,b_0)$.

The argument proceeds analogously in the case where $\sup_{R\le R_0} N_R^0<\ep_1$; we omit the details for brevity.

Finally, if $N_R^0<\ep_1$ for all $R>0$, then $\si(R) = c_0$ for all $R>0$. In this case, $(v,b)$ is regular with the same bound \eqref{eq2.10-BT-CPDE2020} throughout the entire space-time domain $\cup_{0<R<\infty} \R^3\times\left(\frac{3c_0R^2}4, c_0R^2\right] = \R^3\times (0,\infty)$.
\end{proof}

As a consequence of Theorem \ref{thm1.2-BT-CPDE2020}, the uniqueness result below can be established by adapting the argument used in the proof of \cite[Theorem 1.7]{BT-CPDE2020}.

\begin{theorem}[Uniqueness for data that is small at high frequencies]\label{thm-1.7-BT-CPDE2020}
Assume that $v_0,b_0\in E^2$ are divergence free.
Let $(v,b), (u,a)\in\mathcal{N}_{\rm MHD}(v_0,b_0)$. 
Then there exist universal constants $0<\ep_3,\tau_0\le1$ so that if
\[
\sup_{0<r\le R} N_r^0\le \ep_3
\]
for some $R>0$, then $(u,a) = (v,b)$ as distributions on $\R^3\times(0,T)$, $T=\tau_0R^2$.
\end{theorem}

\begin{proof}
Assume $\lim_{R\to0} N^0_R<\ep$ for some $\ep>0$, and either $(v_0,b_0)$ satisfies 
\EQ{\label{eq-1.12-BT-CPDE2020}
\lim_{R\to\infty} \sup_{x_0\in\R^3} \frac1{R^2} \int_{B_R(x_0)} (|v_0(x)|^2 + |b_0(x)|^2)\, dx = 0,
}
or $(v_0,b_0)\in E^2\times E^2$.
Let $R_0$ satisfy $\sup_{R<R_0} N_R^0<\ep$.
By Theorem \ref{thm1.2-BT-CPDE2020} we have for $T=c_0R_0^2$,
\[
t^{1/2}\norm{(v,b)(\cdot,t)}_{L^\infty\times L^\infty}\le C(\ep),\quad 0<t\le T.
\]
Using \eqref{eq2.1-BT-CPDE2020} we have the estimate
\EQ{\label{eq4.4-BT-CPDE2020}
\sup_{r>\sqrt{t/c_0},\, x\in\R^3} \frac1r\int_{B_r(x)} (|v(x,t)|^2 + |b(x,t)|^2)\, dx < C\ep^2,
}
where $c_0$ is defined in \eqref{eq2.3-BT-CPDE2020}.
Moreover, by item 3 of Theorem \ref{thm1.2-BT-CPDE2020}, we have
\EQ{\label{eq4.3-BT-CPDE2020}
\norm{(v,b)(t)}_{L^\infty\times L^\infty} < C\ep^2t^{-1/2},\quad 0<t<\infty.
}
Combining \eqref{eq4.4-BT-CPDE2020} and \eqref{eq4.3-BT-CPDE2020}, we deduce that, for $r<\sqrt{t/c_0}$,
\EQS{\label{eq4.5-BT-CPDE2020}
\frac1r\int_{B_r(x)} &(|v(x,t)|^2 + |b(x,t)|^2)\, dx \\
&= \frac1r \bke{\int_{B_r(x)} (|v(x,t)|^2 + |b(x,t)|^2)\, dx}^{1/3} \bke{\int_{B_r(x)} (|v(x,t)|^2 + |b(x,t)|^2)\, dx}^{2/3}\\
&\lec \norm{(v,b)(t)}_{L^\infty\times L^\infty}^{2/3} \bke{\int_{B_{\sqrt{t/c_0}}(x)} (|v(x,t)|^2 + |b(x,t)|^2)\, dx}^{2/3}\\
&\lec C(\ep) t^{-1/3} (t/c_0)^{1/3} = C(\ep).
}
Using the estimates \eqref{eq4.4-BT-CPDE2020} for $r=\sqrt{T}$ and \eqref{eq4.5-BT-CPDE2020} for $r<\sqrt{T}$, we have for all $t\in(0,T)$ and $r\in(0,T^{1/2})$ that
\[
\frac1r\int_{B_r(x)} (|v(x,t)|^2 + |b(x,t)|^2)\, dx \le C(\ep),
\]
which implies that
\EQ{\label{eq4.6-BT-CPDE2020}
\sup_{0<t<\infty} \norm{(v,b)(t)}_{L^2_{\uloc,r}\times L^2_{\uloc,r}} < C(\ep),
}
where $\norm{f}_{L^2_{\uloc,r}}:= \sup_{x\in\R^3} \norm{f}_{L^2(x)}$.

We now check that $(v,b)$ satisfies the integral formula \eqref{eq-mild-mhd}, i.e., it is a mild solution, on $\R^3\times(0,T)$.
If $(v_0,b_0)\in E^2\times E^2$,
then this follows from a direct adaption of \cite[\S8]{KMT-IMRN2018} for the Navier--Stokes equations to the MHD equations.
On the other hand, assume $(v_0,b_0)$ satisfies \eqref{eq-1.12-BT-CPDE2020}.
By \eqref{eq2.1-BT-CPDE2020}, we have
\[
\esssup_{0\le t\le \si(r)r^2} \norm{(v,b)}_{L^2_{\uloc,r}\times L^2_{\uloc,r}}^2 < CA_0(r),\qquad A_0(r) = \norm{(v_0,b_0)}_{L^2_{\uloc,r}\times L^2_{\uloc,r}}^2.
\]
Since $(v_0,b_0)$ satisfies \eqref{eq-1.12-BT-CPDE2020}, we have $\si(r)r^2\to\infty$ as $r\to\infty$.
So, there exists $\bar{R}$ so that, for all $R>\bar{R}$, $\si(R)R^2>T$. 
We conclude that for any $r>0$
\EQ{\label{eq4.15-BT-CPDE2020}
\esssup_{0\le t\le T}  \norm{(v,b)}_{L^2_{\uloc,r}\times L^2_{\uloc,r}}^2 \le f(r),
}
where $f(r) = CA_0(r) + CA_0(\bar{R})$.

Let $(\td v, \td b)$ be defined as
\EQ{
(\td v,\td b)(x,t) = (e^{t\De}v_0, e^{t\De}b_0) - B((v,b),(v,b))(t),
}
where $B$ is a bilinear operator defined by $B = (B_1,B_2)$, where $B_1$ and $B_2$ are given in \eqref{eq-bilinear-operator}.
Using \cite[(1.8), (1.10)]{MaTe}, we have
\[
\norm{(e^{t\De}v_0,e^{t\De}b_0)}_{L^2_{\uloc,r}\times L^2_{\uloc,r}} \le \norm{(v_0,b_0)}_{L^2_{\uloc,r}\times L^2_{\uloc,r}} = (A_0(r))^{1/2},
\]
and
\EQN{
&\norm{\int_0^t e^{(t-s)\De}\mathbb{P}\nb\cdot(v\otimes v - b\otimes b)(s)\, ds}_{L^2_{\uloc,r}} + \norm{\int_0^t e^{(t-s)\De}\mathbb{P}\nb\cdot(v\otimes b - b\otimes v)(s)\, ds}_{L^2_{\uloc,r}}\\
&\le \int_0^t \frac{C}{(t-s)^{1/2}} \norm{(v,b)(s)}_{L^\infty\times L^\infty} \norm{(v,b)(s)}_{L^2_{\uloc,r}\times L^2_{\uloc,r}}\, ds\\
&\le \int_0^t \frac{C(\ep)}{(t-s)^{1/2}s^{1/2}}\, ds
= C(\ep).
}
Thus
\EQ{\label{eq4.11-BT-CPDE2020}
\sup_{0\le t\le T} \norm{(\td v,\td b)}_{L^2_{\uloc,r}\times L^2_{\uloc,r}}^2 
\le CA_0(r) + C(\ep) \sup_{0<s<T} \norm{(v,b)}_{L^2_{\uloc,r}\times L^2_{\uloc,r}}^2
\le Cf(r).
}
Then $(V,B) = (v,b) - (\td v, \td b)$ satisfies 
\[
\esssup_{0\le t\le T} \norm{(V,B)}_{L^2_{\uloc,r}\times L^2_{\uloc,r}}^2 \le Cf(r),\quad \forall r>0.
\]
Following the same logic in \cite[\S8]{KMT-IMRN2018} and adapting it to MHD equations, we have that the mollified $V_\ep$ and $B_\ep$ are harmonic in $x$ and for fixed $t\in(0,T)$
\[
|V_\ep(x,t)| + |B_\ep(x,t)| \le C\bke{\frac1{r^3}\int_{B_r(x)} (|V_\ep(y,t)|^2 + |B_\ep(y,t)|^2)\, dy}^{1/2}
\le C\bke{\frac{f(r)}{r^3}}^{1/2} \to 0\ \text{ as } \ r\to\infty,
\]
This shows $(V_\ep, B_\ep) = (0,0)$ for $t<T$, for all $\ep>0$.
Hence $(V, B) = (0,0)$ and $(v,b) = (\td v,\td b)$.
This shows that any local energy solution with data satisfying the assumptions of Theorem \ref{thm-1.7-BT-CPDE2020} is a mild solution.

Assume $(u,a)\in\mathcal{N}_{\rm MHD}(v_0,b_0)$ also and satisfy the assumptions of Theorem \ref{thm-1.7-BT-CPDE2020}. 
Then, $(u,a)$ is also a mild solution.
Let $(w,d)=(v,b)-(u,a)$.
Then
\[
w(\cdot,t) = -\int_0^t e^{(t-s)\De}\mathbb{P}\nb\cdot(w\otimes w + u\otimes w + w\otimes u - d\otimes d - a\otimes d - d\otimes a)(\cdot,s)\, ds,
\]
and
\[
d(\cdot,t) = -\int_0^t e^{(t-s)\De}\mathbb{P}\nb\cdot(w\otimes d + u\otimes d + w\otimes a - d\otimes w - a\otimes w - d\otimes u)(\cdot,s)\, ds.
\]
By \cite[Corollary 3.1]{MaTe} we have for $t>0$ that
\EQN{
&\norm{(w,d)(t)}_{L^2_\uloc\times L^2_\uloc} \\
&\qquad \le C\sup_{0<s<t}\sqrt{s} (\norm{(v,b)(s)}_{L^\infty\times L^\infty} + \norm{(u,a)(s)}_{L^\infty\times L^\infty}) \norm{(w,d)}_{L^\infty(L^2_\uloc\times L^2_\uloc)}.
}
Since $\sqrt{s} (\norm{(v,b)(s)}_{L^\infty\times L^\infty} + \norm{(u,a)(s)}_{L^\infty\times L^\infty})$ is small we have 
\[
\norm{(w,d)(t)}_{L^2_\uloc\times L^2_\uloc} 
\le \frac12 \norm{(w,d)}_{L^\infty(L^2_\uloc\times L^2_\uloc)}.
\]
Taking the essential supremum on the left hand side leads to uniqueness.
\end{proof}

As a corollary of Theorem \ref{thm-1.7-BT-CPDE2020}, the following local uniqueness result in $E^3$ can be derived by adapting the proof of \cite[Corollary 1.8]{BT-CPDE2020}.

\begin{corollary}[Local uniqueness in $E^3$]\label{cor1.8-BT-CPDE2020}
Assume $v_0,b_0\in E^3$ are divergence free.
Let $(v,b)$ and $(u,a)$ be elements of $\mathcal{N}_{\rm MHD}(v_0,b_0)$. 
Then, there exists $T=T(v_0,b_0)>0$ so that $(v,b) = (u,a)$ as distributions on $\R^3\times(0,T)$.
\end{corollary}

\begin{proof}
Assume $(v_0,b_0)\in E^3\times E^3$.
Then, in particular, $(v_0,b_0)\in E^2\times E^2$.
Let $\ep>0$ be given.
Since $(v_0,b_0)\in E^3\times E^3$, there exists $R_0$ such that 
\[
\sup_{|x_0|\ge R_0}\int_{B_1(x_0)} (|v_0(x)|^3 + |b_0(x)|^3)\, dx < \ep.
\]
On the other hand, since $v_0, b_0\in E^3$, their local $L^3$-norms are uniformly small on sufficiently small balls.
That is, there exists $\ga\in(0,1]$ such that
\[
\sup_{|x_0|\le R_0,\, 0<r\le\ga} \int_{B_r(x_0)} (|v_0(x)|^3 + |b_0(x)|^3)\, dx < \ep.
\]
Applying H\"older's inequality,  we obtain:
\[
\sup_{|x_0|\le R_0,\, 0<r\le\ga} \frac1r \int_{B_r(x_0)} (|v_0(x)|^2 + |b_0(x)|^2)\, dx < |B_1|^{1/3} \ep^{2/3}.
\]
Therefore, by Theorem \ref{thm-1.7-BT-CPDE2020}, any local energy solution with initial data $(v_0,b_0)$ is unique in the local energy class, at least for a short time.
\end{proof}

\subsection{Eventual regularity for local energy solutions}\label{sec-eventual-reg}

\begin{lemma}\label{lem2.2-BT-SIMA2021}
Assume $v_0,b_0\in E^2_q$, $2\le q<\infty$. Then
\[
\lim_{R\to\infty} R^{\frac6q - 2} N_{q,R}^0(v_0,b_0) = 0.
\]
Consequently, if $v_0, b_0\in E^2_q$, then
\[
\lim_{R\to\infty} N^0_R(v_0,b_0) = 0\ \text{ if }\ 2\le q\le 3
\quad\text{ and }\quad
\lim_{R\to\infty} R^{-1} N^0_{q,R}(v_0,b_0) = 0\ \text{ if }\ 2\le q\le 6.
\]
\end{lemma}

\begin{proof}
The proof of the lemma is an adaption of the proof of \cite[Lemma 2.2]{BT-SIMA2021} for the Navier--Stokes equations to the MHD equations.

Let $\ep>0$ be given.
Suppose that $v_0,b_0\in E^2_q$ for some $2\le q<\infty$. 
Then the norms can be expressed as $\norm{v_0}_{E^2_q} = \norm{u}_{\ell^q(\ZZ^3)}$ and $\norm{v_0}_{E^2_q} = \norm{a}_{\ell^q(\ZZ^3)}$, where
\[
u = (u_k)_{k\in\ZZ^3} \in \ell^q(\ZZ^3),\quad u_k = \norm{v_0}_{L^2(B_1(k))},
\]
\[
a = (a_k)_{k\in\ZZ^3} \in \ell^q(\ZZ^3),\quad a_k = \norm{b_0}_{L^2(B_1(k))}.
\]
For any $R\ge1$,  we have the estimate
\EQS{\label{eq-2.2-BT-SIMA2021}
N_{q,R}^0(v_0,b_0) \le \frac{C}{R} \bkt{\sum_{k\in\ZZ^3}\bke{\sum_{|i-kR|<R} (u_i^2 + a_i^2)}^{\frac{q}2}}^{\frac2q}
}
Now fixe any $\de>0$. Since $u,a\in\ell^q(\ZZ^3)$, we can choose $M>1$ large enough such that $\norm{u^{>M}}_{\ell^q}\le\de$ and $\norm{a^{>M}}_{\ell^q}\le\de$, where
\[
u^{>M}_k = 
\begin{cases}
0&\ \text{ if }|k|\le M,\\
u_k&\ \text{ if }|k|>M,
\end{cases}
\quad\text{ and }\quad
a^{>M}_k = 
\begin{cases}
0&\ \text{ if }|k|\le M,\\
a_k&\ \text{ if }|k|>M.
\end{cases}
\]
Set $u^{\le M} = u - u^{>M}$ and $a^{\le M} = a - a^{>M}$. 
Following the argument leading to \cite[(2.4)]{BT-SIMA2021} by H\"older's inequality we obtain, for all $R>M$,
\[
\bkt{R^{\frac6q-2} N_{q,R}^0(v_0,b_0)}^{\frac{q}2} 
\le C \bke{\norm{u^{>M}}_{\ell^q}^q + \norm{a^{>M}}_{\ell^q}^q} + CR^{3-\frac{3q}2} M^{\frac{3(q-2)}2} \bke{\norm{u^{\le M}}_{\ell^q}^q + \norm{a^{\le M}}_{\ell^q}^q}.
\]
To conclude, we first choose $\de>0$ small enough so that $C\bke{\norm{u^{>M}}_{\ell^q}^q + \norm{a^{>M}}_{\ell^q}^q} < \ep/2$.
Then, for this fixed $M=M(\de)$, we choose $R$ sufficiently large to ensure that $CR^{3-\frac{3q}2}M(\de)^{\frac{3(q-2)}2}\bke{\norm{u^{\le M}}_{\ell^q}^q + \norm{a^{\le M}}_{\ell^q}^q} < \ep/2$, which is possible provided $q>2$.

To prove the final assertions, we begin by noting that $N_R^0(v_0,b_0) \le N_{q,R}^0(v_0,b_0)$. 
Moreover, if $v_0,b_0\in E^2_q$ for $q\le 3$, then in particular $v_0,b_0\in E^2_3$.
Therefore,
\[
\lim_{R\to\infty} N_R^0(v_0,b_0) \le \lim_{R\to\infty} N_{3,R}^0(v_0,b_0) = 0.
\]
For the final part, observe that when $q\le 6$ and $R\ge1$, we have $R^{-1} \le R^{\frac6q-2}$.
Thus,
\[
\lim_{R\to\infty} R^{-1} N^0_{q,R}(v_0,b_0) = 0 .
\]
\end{proof}

\begin{proof}[Proof of Theorem \ref{bt4-thm-1.3-mhd}]
For $2\le q\le 3$, Lemma \ref{lem2.2-BT-SIMA2021} implies
\[
\lim_{R\to\infty} N^0_R(v_0,b_0) = 0.
\]
Applying Theorem \ref{thm1.2-BT-CPDE2020} then yields the desired result.

For $1\le q<2$, the same conclusion follows from the fact that $v_0, b_0 \in E^2_q\subset L^2$ since $E^2_q$ embeds into $L^2$ when $q<2$.
This completes the proof of Theorem \ref{bt4-thm-1.3-mhd}.
\end{proof}

\subsection{A priori bounds and explicit growth rate}\label{sec-growth-rate}

In this section we prove new a priori bounds for data $(v_0,b_0)\in E^2_q\times E^2_q$ and use it to prove Theorem \ref{thm-1.4BT4-mhd}.

\begin{lemma}\label{lem3.1-BT-SIMA2021}
Assume $v_0,b_0\in E^2_q$ for some $q\ge1$ are divergence free and that $(v,b)\in\mathcal{N}_{\rm MHD}(v_0,b_0)$ satisfies, for some $T_2>0$,
\EQS{\label{eq-3.4-BT-SIMA2021}
\norm{\esssup_{0\le T\le T_1} \int_{B_1(x_0)} \bke{|v|^2 + |b|^2}dx + \int_0^{T_1}\int_{B_1(x_0)}\bke{|\nb v|^2 + |\nb b|^2} dxdt}_{\ell^{\frac{q}2}(x_0\in\ZZ^3)} < \infty
\ \text{ for all }T_1\in(0,T_2).
}
Then there are positive constants $C_1$ and $\la_0<1$, both independent of $q$ and $R$ such that, for all $R>0$ with $\la_R R^2\le T_2$,
\EQS{\label{eq-3.5-BT-SIMA2021}
\norm{\esssup_{0\le t\le \la_R R^2} \int_{B_R(x_0R)} \frac{|v|^2 + |b|^2}2\, dx + \int_0^{\la_R R^2} \int_{B_R(x_0R)} \bke{|\nb v|^2 + |\nb b|^2} dxdt}_{\ell^{\frac{q}2}(x_0\in\ZZ^3)} \le C_1A_{0,q}(R),
}
where
\[
A_{0,q}(R) = RN^0_{q,R} = \norm{\int_{B_R(x_0R)} \bke{|v_0|^2 + |b_0|^2} dx}_{\ell^{\frac{q}2}(x_0\in\ZZ^3)},\quad
\la_R = \min\bke{\la_0,\,\frac{\la_0R^2}{\bke{A_{0,q}(R)}^2}}.
\]
Furthermore, for all $R>0$,
\EQS{\label{eq-3.6-BT-SIMA2021}
\norm{\int_0^{\la_R R^2} \int_{B_R(x_0R)} |v|^{\frac{10}3} + |b|^{\frac{10}3} + \abs{\pi - c_{R x_0,R}(t)}^{\frac53} dxdt}_{\ell^{\frac{3q}{10}}(x_0\in\ZZ^3)} \le C \bke{A_{0,q}(R)}^{\frac53},\quad q\ge2,
}
and
\EQS{\label{eq-3.6-BT-SIMA2021-q<2}
\norm{\int_0^{\la_R R^2} \int_{B_R(x_0R)} |v|^{\frac{10}3} + |b|^{\frac{10}3} + \abs{\pi - c_{R x_0,R}(t)}^{\frac53} dxdt}_{\ell^\si(x_0\in\ZZ^3)} \le C \bke{A_{0,q}(R)}^{\frac53},\quad 1\le q<2.
}
for all $\si\ge\frac35$.
\end{lemma}

\begin{proof}
The proof of the lemma is an adaption of the proof of \cite[Lemma 3.1]{BT-SIMA2021} for the Navier--Stokes equations to the MHD equations.

Let $\phi_0\in C^\infty_c(\R^3)$ be radial, non-increasing cutoff function such that $\phi_0\equiv1$ on $B_1(0)$, $\supp\phi_0\subset B_2(0)$, and $|\nb\phi_0(x)|\lec1$, $|\nb\phi_0^{1/2}(x)|\lec1$.
Let $R>0$ be as in the statement of the lemma, and define the scaled cutoff $\phi(x) := \phi_0(x/R)$.
Fix $0<\la\le1$.

For each $\ka\in R\ZZ^3$, define the localized energy quantity
\[
e_{R,\la}(\ka) := \esssup_{0\le t\le\la R^2} \int \bke{|v(t)|^2 + |b(t)|^2} \phi(x-\ka)\, dx + \int_0^{\la R^2} \int \bke{|\nb v|^2 + |\nb b|^2} \phi(x-\ka)\, dxdt.
\]
We begin by deriving bounds on $e_{R,\la}(\ka)$, which will then be used to control the quantity
\[
E_{R,q,\la} := \norm{\esssup_{0\le t\le\la R^2} \int \bke{|v(t)|^2 + |b(t)|^2}\phi(x-Rk)\, dx + \int_0^{\la R^2} \int \bke{|\nb v|^2 + |\nb b|^2} \phi(x-Rk)\, dxdt}_{\ell^{\frac{q}2}(k\in\ZZ^3)}^{\frac{q}2}
\]
in terms of $A_{0,q}(R)$ for sufficiently small $\la$.
By assumption, $E_{R,q,\la}<\infty$.
To estimate $e_{R,\la}(\ka)$, we apply the local energy inequality \eqref{lei_mhd}:
\EQN{
\int &\bke{|v(t)|^2 + |b(t)|^2} \phi(x-\ka)\, dx + 2\int_0^t \int \bke{|\nb v|^2 + |\nb b|^2} \phi(x-\ka)\, dxds\\
&\le \int \bke{|v_0|^2 + |b_0|^2} \phi(x-\ka)\, dx + \int_0^t \int \bke{|v|^2 + |b|^2} \De\phi(x-\ka)\, dxds\\
&\quad + \int_0^t\int \bke{|v|^2 + |b|^2} \bke{v\cdot\nb\phi(x-\ka)} dxds + \int_0^t\int 2\pi \bke{v\cdot\nb\phi(x-\ka)} dxds\\
&\quad - 2 \int_0^t \int (b\cdot v) \bke{b\cdot\nb\phi(x-\ka)} dxds.
}
We now estimate each term on the right-hand side, beginning with the second term.
Using the properties of $\phi$, we have:
\EQN{
\int_0^{\la R^2} \int \bke{|v|^2 + |b|^2} \abs{\De\phi(x-\ka)} dxds 
&\le \frac{C}{R^2} \int_0^{\la R^2} \int_{B_{2R}(\ka)} \bke{|v|^2 + |b|^2} dxds\\
&\le C\la \sum_{\ka'\in R\ZZ^3;\,|\ka'-\ka|\le 2R} \esssup_{0\le t\le\la R^2} \int \bke{|v|^2 + |b|^2}\phi(x-\ka')\, dx\\
&\le C\la \sum_{\ka'\in R\ZZ^3;\,|\ka'-\ka|\le 2R} e_{R,\la}(\ka').
}

For the cubic terms, we apply the Gagliardo--Nirenberg inequality:
\[
\int_{B_{2R}} |v|^3\, dx \lec \bke{\int_{B_{2R}} |\nb v|^2 }^{\frac34} \bke{\int_{B_{2R}} |v|^2 }^{\frac34} + R^{-\frac32} \bke{\int_{B_{2R}} |v|^2 }^{\frac32},
\]
\[
\int_{B_{2R}} |b|^3\, dx \lec \bke{\int_{B_{2R}} |\nb b|^2 }^{\frac34} \bke{\int_{B_{2R}} |b|^2 }^{\frac34} + R^{-\frac32} \bke{\int_{B_{2R}} |b|^2 }^{\frac32}.
\]
Let $N = \sup_{0\le t\le\la R^2} \int_{B_{2R}} \bke{|v(t)|^2 + |b(t)|^2} dx + 2\int_0^{\la R^2} \int_{B_{2R}} \bke{|\nb v|^2 + |\nb b|^2} dxdt$.
Then, integrating the above estimates over time yields
\EQS{\label{eq-3.9-BT-SIMA2021}
\int_0^{\la R^2} \int_{B_{2R}} \bke{|v|^3 + |b|^3}\, dxds 
&\lec N^{\frac34} \int_0^{\la R^2}  \bkt{ \bke{\int_{B_{2R}} |\nb v|^2}^{\frac34} + \bke{\int_{B_{2R}} |\nb b|^2}^{\frac34} } ds  + R^{-\frac32} N^{\frac32} \la R^2\\
&\lec N^{\frac32} (\la R^2)^{\frac14} + N^{\frac32}\la R^{\frac12}\\
&\lec N^{\frac32} \la^{\frac14} R^{\frac12},
}
where we've used $\la\le1$ in the final step.
As a consequence, we have
\EQN{
\int_0^{\la R^2} \int \bke{|v|^2 + |b|^2} \bke{v\cdot\nb\phi(x-\ka)} dxds
&\le \frac{C}{R} \int_0^{\la R^2} \int_{B_{2R}(\ka)} \bke{|v|^3 + |b|^3} dxds\\
&\le CR^{-\frac12}\la^{\frac14} \sum_{\ka'\in R\ZZ^3;\, |\ka'-\ka|\le 4R} (e_{R,\la}(\ka'))^{\frac32},
}
and
\EQN{
-2 \int_0^{\la R^2} \int (b\cdot v) \bke{b\cdot\nb\phi(x-\ka)} dxds 
&\le \frac{C}R \int_0^{\la R^2} \int_{B_{2R}(\ka)} |v||b|^2\, dxds\\
&\le \frac{C}{R} \int_0^{\la R^2} \int_{B_{2R}(\ka)} \bke{|v|^3 + |b|^3} dxds\\
&\le CR^{-\frac12}\la^{\frac14} \sum_{\ka'\in R\ZZ^3;\, |\ka'-\ka|\le 4R} (e_{R,\la}(\ka'))^{\frac32}.
}

Now, the only remaining term is the pressure term.
To estimate it, we use the local pressure expansion \eqref{eq-pi-formula-mhd} to write $\pi(x,t)$ for $x\in B_{2R}(\ka)$ as
\EQN{
\pi(x,t) &= -\De^{-1} \div\div\bkt{(v\otimes v - b\otimes b)\chi_{4R}(x-\ka)}\\
&\quad - \int_{\R^3} \bke{K(x-y) - K(\ka-y)} (v\otimes v - b\otimes b )(y,t) \bke{1-\chi_{4R}(y-\ka)} dy + c_{\ka,R}(t)\\
&=: \pi_1(x,t) + \pi_2(x,t) + c_{\ka,R}(t).
}
Note that
\EQS{\label{eq-3.11-BT-SIMA2021}
|K(x-y) - K(\ka-y)|\le \frac{CR}{|\ka-y|^4}
} 
for $|x-\ka|\le 2R$ and $|\ka-y|\ge 4R$.
This ensures that $\pi_2$ is well-defined even if $v$ and $b$ lack decay at infinity.

For the localized term $\pi_1$, standard Calderon--Zygmund theory gives
\EQN{
\norm{\pi_1}_{L^{\frac32}(B_{2R}(\ka))} 
&\le \norm{v\chi_{4R}^{1/2}(\cdot-\ka)}_{L^3}^2 + \norm{b\chi_{4R}^{1/2}(\cdot-\ka)}_{L^3}^2\\
&\le C \sum_{\ka'\in R\ZZ^3;\, |\ka'-\ka|\le9R} \bke{ \norm{v\phi^{1/2}(\cdot-\ka)}_{L^3}^2 + \norm{b\phi^{1/2}(\cdot-\ka)}_{L^3}^2},
}
where we used the support and scaling properties of the cutoff functions.
Then, applying H\"older's inequality and the inequality $\bke{\sum_{j=1}^n a_j^2}\bke{\sum_{j=1}^n |a_j|} \le n\sum_{j=1}^n |a_j|^3$ along with the bound from \eqref{eq-3.9-BT-SIMA2021}, we obtain
\EQN{
\int_0^{\la R^2} \int &2\pi_1 \bke{v\cdot\nb\phi(x-\ka)} dxds\\
&\le \frac{C}{R} \int_0^{\la R^2} \sum_{\ka'\in R\ZZ^3;\,|\ka'-\ka|\le9R} \bke{ \norm{v\phi^{1/2}(\cdot-\ka')}_{L^3}^3 + \norm{b\phi^{1/2}(\cdot-\ka')}_{L^3}^3 } ds\\
&\le CR^{-\frac12}\la^{\frac14} \sum_{\ka'\in R\ZZ^3;\,|\ka'-\ka|\le10R} (e_{R,\la}(\ka'))^{\frac32}.
}

To estimate $\pi_2$, we begin with the following pointwise bound for $x\in B_{2R}(\ka)$:
\EQN{
|\pi_2(x,t)| 
&\le C\int \frac{R}{|\ka-y|^4} \bke{|v|^2 + |b|^2}(y,t) \bke{1-\chi_{4R}(y-\ka)} dy\\
&\le C \sum_{\ka'\in R\ZZ^3;\,|\ka'-\ka|>4R} \int_{B_{2R}(\ka')} \frac{R}{|\ka-y|^4}\bke{|v(y,t)|^2+ |b(y,t)|^2} \phi(y-\ka')\, dy\\
&\le \frac{C}{R^3} \sum_{\ka'\in R\ZZ^3;\,|\ka'-\ka|>4R} \frac1{|\ka/R - \ka'/R|^4} \int_{B_{2R}(\ka')} \bke{|v(y,t)|^2+ |b(y,t)|^2} \phi(y-\ka')\, dy\\
&\le \frac{C}{R^3} \bke{\overline{K}*e_{R,\la}}(\ka),
}
where we used the estimate \eqref{eq-3.11-BT-SIMA2021}, and the convolution $\overline{K}*e_{R,\la}$ is taken over the lattice $R\ZZ^3$, with
\[
\overline{K}(x) = 
\begin{cases}
\dfrac1{|x/R|^4}\ \text{ if } |x|>4R,\\
\overline{K}(x) = 0 \ \text{ otherwise,}
\end{cases}
\ 
\text{ for $x\in R\ZZ^3$.}
\]

Assume now that $q\ge2$.
Using $\la\le1$, we estimate
\EQS{\label{eq-q>2-key}
\int_0^{\la R^2}& \int 2\pi_2(x,s) \bke{v(x,s)\cdot\nb\phi(x-\ka)} dxds\\
&\le \frac{C}{R} \int_0^{\la R^2} \int_{B_{2R}(\ka)} |\pi_2|^{\frac32}\, dxds + \frac{C}{R} \int_0^{\la R^2} \int_{B_{2R}(\ka)} |v|^3\, dxds\\
&\le C\la^{\frac14}R^{-\frac12} \bke{(\overline{K}*e_{R,\la})(\ka)}^{\frac32} + C\la^{\frac14}R^{-\frac12} \sum_{|\ka'-\ka|\le4R} \bke{e_{R,\la}(\ka')}^{\frac32}.
}

Note that $\int_0^{\la R^2} \int 2c_{x_0,R}(s) v\cdot\nb\phi(x-\ka)\, dxds = 0$.

Combining all estimates, we obtain the key bound:
\EQS{\label{eq-3.12-BT-SIMA2021}
e_{R,\la}(\ka) 
&\le \int \bke{|v_0|^2 + |b_0|^2}\phi(x-\ka)\, dx + C\la \sum_{\ka'\in R\ZZ^3;\,|\ka'-\ka|\le2R} e_{R,\la}(\ka')\\
&\quad + C \frac{\la^{\frac14}}{R^{\frac12}} \sum_{\ka'\in R\ZZ^3;\,|\ka'-\ka|\le2R} \bke{e_{R,\la}(\ka')}^{\frac32} + C \frac{\la^{\frac14}}{R^{\frac12}}  \bke{(\overline{K} *e_{R,\la})(\ka)}^{\frac32},
}
provided $\la\le1$.
Note that all constants here are independent of  $q$.
We now raise both sides of \eqref{eq-3.12-BT-SIMA2021} to the power $q/2$ and sum over $\ka\in R\ZZ^3$.
The left-hand side yields $E_{R,q,\la}$.
For the right-hand side:
\begin{itemize}
\item The initial data term is controlled by 
\[
\sum_{\ka\in R\ZZ^3} \bkt{\int \bke{|v_0|^2 + |b_0|^2} \phi(x-\ka)\, dx}^{\frac{q}2} \le C^q \bke{A_{0,q}(R)}^{\frac{q}2},
\]

\item The linear term gives 
\[
\sum_{\ka\in R\ZZ^3} \bkt{ C\la \sum_{\ka'\in R\ZZ^3;\,|\ka'-\ka|\le2R} e_{R,\la}(\ka') }^{\frac{q}2} \le C^q\la^{\frac{q}2} E_{R,q,\la},
\]

\item For the nonlinear cubic term, using $\bke{\sum_{j=1}^n a_j}^p \le n^p \sum_{j=1}^n a_j^p$ for $p\ge1$ and $a_j\ge0$, we have 
\[
\sum_{\ka\in R\ZZ^3} \bkt{ C \frac{\la^{\frac14}}{R^{\frac12}} \sum_{\ka'\in R\ZZ^3;\,|\ka'-\ka|\le2R} \bke{e_{R,\la}(\ka')}^{\frac32}  }^{\frac{q}2} \le C^q \bke{\frac{\la^{\frac14}}{R^{\frac12}}}^{\frac{q}2} \sum_{\ka\in R\ZZ^3} \bke{e_{R,\la}(\ka)}^{\frac{3q}4}.
\]

\item For the convolution term we use Young's convolution inequality to find
\EQN{
\sum_{\ka\in R\ZZ^3} \bke{C\frac{\la^{\frac14}}{R^\frac12} \bke{(\overline{K}*e_{R,\la})(\ka)}^{\frac32} }^{\frac{q}2} 
&\le C^q \bke{\frac{\la^{\frac14}}{R^{\frac12}}}^{\frac{q}2} \sum_{\ka\in R\ZZ^3} \bke{(\overline{K}*e_{R,\la})(\ka)}^{\frac{3q}4}\\
&\le C^q \bke{\frac{\la^{\frac14}}{R^{\frac12}}}^{\frac{q}2} \norm{\overline{K}}_{\ell^1(R\ZZ^3)}^{\frac{3q}4} \norm{e_{R,\la}}_{\ell^{\frac{3q}4}}^{\frac{3q}4},
}
where $\norm{\overline{K}}_{\ell^1(R\ZZ^3)}$ is bounded independently of $R$.
\end{itemize} 

Now, since $\norm{e_{R,\la}}_{\ell^{\frac{3q}4}} \le \norm{e_{R,\la}}_{\ell^{\frac{q}2}}$, we conclude for $E=E_{R,q,\la}$ and some constant $C_2\ge1$ independent of $q$, $R$ that
\EQS{\label{eq-3.13-BT-SIMA2021}
E \le C_2^q \bke{A_{0,q}(R)}^{\frac{q}2} + C_2^q\la^{\frac{q}2}E + C_2^q\bke{\frac{\la^{\frac14}}{R^{\frac12}}}^{\frac{q}2} E^{\frac32}.
}
The right side is finite for $\la<R^{-2}T_2$ by assumption \eqref{eq-3.4-BT-SIMA2021}.
It follows from the same argument as in \cite[p.~2005]{BT-SIMA2021}, $E_{R,q,\la}$ is continuous in $\la$.
So, from \eqref{eq-3.13-BT-SIMA2021} we conclude that 
\[
E\le 2E_0,\qquad E_0 = C_2^q\bke{A_{0,q}(R)}^{\frac{q}2},
\]
provided $C_2^q\la^{\frac{q}2}\le1/4$ and $C_2^q\bke{\frac{\la^{\frac14}}{R^{\frac12}}}^{\frac{q}2} (2E_0)^{\frac12}\le 1/4$, which is achieved if (using $q\ge2$)
\EQN{
\la\le\la_R := \min\bke{\la_0,\,\frac{\la_0R^2}{\bke{A_{0,q}(R)}^2}},
}
where $\la_0 = \min\bke{(2C_2)^{-2}, (2C_2)^{-12}}$.
This shows the first estimate \eqref{eq-3.5-BT-SIMA2021} of Lemma \ref{lem3.1-BT-SIMA2021}, for $q\ge2$, with $C_1 = CC_2^2$.
Note that the constants $C_2$, $\la_0$, and $C_1$ do not depend on $q$ and $R$.

For $1\le q<2$, we replace \eqref{eq-q>2-key} by 
\EQN{
\int_0^{\la R^2} \int 2 \pi_2 v\cdot \nb \phi( x-\ka) \,dx\,ds &\lesssim 
\frac 1 {R^4} {\la R^{\frac 7 2}} \| v\|_{L^\I(0,\la R^2; L^2(B_{2R}(\ka)))}  |\overline K * e_{R,\la}(\ka)  |
\\&\lesssim 
{\frac {\la} {R^{1/2}}} \bke{ \| v\|_{L^\I(0,\la R^2; L^2(B_{2R}(\ka)))}^2+  |\overline K * e_{R,\la}(\ka)  |^2 }
\\&\lesssim {\frac {\la} {R^{1/2}}} \bke{ \sum_{\ka'\in R\ZZ^3; |\ka'-\ka|\leq 2R} e_{R,\la}(\ka') +  |\overline K * e_{R,\la}(\ka)  |^2 }.
}
Raising both sides of the above inequality to the power $q/2$ and sum over $\ka\in R\ZZ^3$, we get
\[ \sum_{\ka\in R\ZZ^3}\bigg(  C\big(\la +{\frac {\la} {R^{1/2}}} \big)  \sum_{\ka'\in R\ZZ^3; |\ka'-\ka|\leq 2R} e_{R,\la}(\ka')   \bigg)^{q/2} 
\leq C^q \bigg(\la+   {\frac {\la} {R^{1/2}}}\bigg)^{q/2}    E_{R,q,\la} ,
\]
Above we have used $\bke{\sum_{j=1}^n a_j}^p \le \sum_{j=1}^n a_j^p$ for $0<p<1$ and $a_j\ge0$.
For the  convolution term we use Young's convolution inequality to obtain 
\EQN{
&\sum_{\ka\in R\ZZ^3} \bigg({\frac {C\la} {R^{1/2}}}   |\overline K * e_{R,\la}(\ka)  |^2 \bigg)^{q/2}
= \bke{\frac {C\la} {R^{1/2}}}^{\frac q2} \sum_{\ka\in R\ZZ^3}  |\overline K * e_{R,\la}(\ka)  |^q
\\&\quad \le \bke{\frac {C\la} {R^{1/2}}}^{\frac q2}  	\|\overline K\|_{\ell^1}	 ^{q}	 	  	\|e_{R,\la}\|_{\ell^q}	 ^{q}
\le \bke{\frac {C\la} {R^{1/2}}}^{\frac q2} \|e_{R,\la}\|_{\ell^{q/2}} ^{q}				
\le  \bigg( \frac {C\la} {R^{1/2}} \bigg) ^{\frac q2}	{E_{R,q,\la}^2},
} 
where we used the fact that $ \| \overline K \|_{\ell^1(R\ZZ^3)}$ is bounded independently of $R$.
We conclude that, for some constant $C_2\ge1$ independent of $q,R$, we have
\EQS{\label{Eq.est}
E_{R,q,\la} &\leq C_2^q A_{0,q}(R)^{q/2} + C_2^{q} \bke{\la+ \frac {\la} {R^{1/2}} }^{q/2}  E_{R,q,\la}+  C_2^q  \bigg( \frac {\la} {R^{2}}  \bigg)^{\frac q8}  E_{R,q,\la}^{3/2}
\\
&\quad +C_2^q \bigg( \frac {\la} {R^{1/2}} \bigg) ^{\frac q2}{E_{R,q,\la}^2}
}
The same argument as in \cite[p.~2005]{BT-SIMA2021} shows that $E_{R,q,\la}$ is continuous in $\la$.
Therefore, we conclude from \eqref{Eq.est} and a continuity argument that the estimate \eqref{eq-3.5-BT-SIMA2021} also holds for $1<q<2$.

We now show \eqref{eq-3.6-BT-SIMA2021} and \eqref{eq-3.6-BT-SIMA2021-q<2}. 
By the Gagliardo--Nirenberg inequality,
\[
\int_{B_R} |v|^{\frac{10}3}\, dx \lec \bke{\int_{B_R} |\nb v|^2 } \bke{\int_{B_R} |v|^2 }^{\frac23} + R^{-2} \bke{\int_{B_R} |v|^2 }^{\frac53},
\]
\[
\int_{B_R} |b|^{\frac{10}3}\, dx \lec \bke{\int_{B_R} |\nb b|^2 } \bke{\int_{B_R} |b|^2 }^{\frac23} + R^{-2} \bke{\int_{B_R} |b|^2 }^{\frac53}.
\]
Denoting $N = \sup_{0\le t\le\la R^2} \int_{B_R} \bke{|v(t)|^2 + |b(t)|^2} dx + 2\int_0^{\la R^2} \int_{B_R} \bke{|\nb v|^2 + |\nb b|^2} dxds$ with $\la = \la_R$, we have
\EQS{\label{eq-3.15-BT-SIMA2021}
\int_0^{\la R^2} \int_{B_R} \bke{|v|^{\frac{10}3} + |b|^{\frac{10}3} } dxds 
&\lec N^{\frac23} \int_0^{\la R^2} \bke{\int_{B_R} |\nb v|^2 + |\nb b|^2 } ds + R^{-2} N^{\frac53} \la R^2\\
&\lec N^{\frac53} + \la N^{\frac53} 
\lec N^{\frac53},
}
using $\la\le1$.
For $k\in R\ZZ^3$ and $Q(k) = B_R(k)\times(0,\la_R R^2)$, by \eqref{eq-3.15-BT-SIMA2021} with $B_R$ replaced by $B_R(k)$, we have $N\le e_{R,\la}(k)$ and hence
\EQN{
\sum_{k\in R\ZZ^3} \bkt{\int_{Q(k)} \bke{|v|^{\frac{10}3} + |b|^{\frac{10}3}} dxdt}^{\frac{3q}{10}}
\le C\sum_{k\in R\ZZ^3} \bke{(e_{R,\la}(k))^{\frac53}}^{\frac{3q}{10}}
\le CE_0,\quad q\ge2,
}
and, for $\si>0$ satisfying $\frac{5\si}3\ge\frac{q}2$ (which is implied if $\si\ge\frac35$), we have
\EQN{
\sum_{k\in R\ZZ^3}& \bkt{\int_{Q(k)} \bke{|v|^{\frac{10}3} + |b|^{\frac{10}3}} dxdt}^{\si}\\
&\le C\sum_{k\in R\ZZ^3} \bke{e_{R,\la}(k)}^{\frac{5\si}3}
\le C\bkt{\sum_{k\in R\ZZ^3} \bke{e_{R,\la}(k)}^{\frac{q}2} }^{\frac2q\cdot\frac{5\si}3}
\le CE_0^{\frac2q\cdot\frac{5\si}3},\quad 1\le q<2.
}
By Calderon--Zygmund estimates, 
\EQN{
\sum_{k\in R\ZZ^3} &\bke{\int_{Q(k)} |\pi_1|^{\frac53} dxdt}^{\frac{3q}{10}}\\
&\le C \sum_{k\in R\ZZ^3} \sum_{k'\in R\ZZ^3;\, |k-k'|<10R} \bkt{ \bke{\int_{Q(k')} |v|^{\frac{10}3}\, dxdt }^{\frac{3q}{10}} + \bke{\int_{Q(k')} |b|^{\frac{10}3}\, dxdt }^{\frac{3q}{10}}}\\
&\le CE_0,\quad q\ge2,
}
and
\EQN{
\sum_{k\in R\ZZ^3} &\bke{\int_{Q(k)} |\pi_1|^{\frac53} dxdt}^{\si}\\
&\le C \sum_{k\in R\ZZ^3} \sum_{k'\in R\ZZ^3;\, |k-k'|<10R} \bkt{ \bke{\int_{Q(k')} |v|^{\frac{10}3}\, dxdt }^{\si} + \bke{\int_{Q(k')} |b|^{\frac{10}3}\, dxdt }^{\si}}\\
&\le CE_0^{\frac2q\cdot\frac{5\si}3},\quad 1\le q<2.
}

For $\pi_2$, recall $\pi_2$ in $B_R(k)$ is bounded by $R^{-3}\overline{K}*e_{R,\la}(k)$ and hence
\[
\int_{Q(k)} |\pi_2|^{\frac53} dxdt \le C\la\bke{(\overline{K}*e_{R,\la})(k)}^{\frac53}.
\]
Thus,
\EQN{
\sum_{k\in R\ZZ^3} \bke{\int_{Q(k)} |\pi_2|^{\frac53} dxdt}^{\frac{3q}{10}}
&\le C \la^{\frac{3q}{10}} \sum_{k\in R\ZZ^3} \bke{(\overline{K}*e_{R,\la})(k)}^{\frac{q}2}\\
&\le C \la^{\frac{3q}{10}} \norm{\overline{K}}_{\ell^1}^{\frac{q}2} \sum_{k\in R\ZZ^3} \bke{e_{R,\la}(k)}^{\frac{q}2}
\le CE_0,\quad q\ge2,
}
and, if $\frac{5\si}3\ge1$, which is implied by our assumptions, we have
\EQN{
\sum_{k\in R\ZZ^3} \bke{\int_{Q(k)} |\pi_2|^{\frac53} dxdt}^{\si}
&\le C \la^{\si} \sum_{k\in R\ZZ^3} \bke{(\overline{K}*e_{R,\la})(k)}^{\frac{5\si}3}\\
&\le C \la^{\si} \norm{\overline{K}}_{\ell^1}^{\frac{5\si}3} \sum_{k\in R\ZZ^3} \bke{e_{R,\la}(k)}^{\frac{5\si}3}
\le CE_0^{\frac2q\cdot\frac{5\si}3},\quad 1\le q<2.
}
We conclude that
\[
\norm{\int_{Q(k)} |v|^{\frac{10}3} + |b|^{\frac{10}3} + |\pi_1 + \pi_2|^{\frac53}dxdt}_{\ell^{\frac{3q}{10}}(k\in R\ZZ^3)} 
\le CE_0^{\frac{10}{3q}} 
= C\bke{A_{0,q}(R)}^{\frac53},\quad q\ge2,
\]
and, for $\frac{5\si}3\ge1$,
\[
\norm{\int_{Q(k)} |v|^{\frac{10}3} + |b|^{\frac{10}3} + |\pi_1 + \pi_2|^{\frac53}dxdt}_{\ell^{\si}(k\in R\ZZ^3)} 
\le CE_0^{\frac{10}{3q}} 
= C\bke{A_{0,q}(R)}^{\frac53},\quad 1\le q<2.
\]
This shows \eqref{eq-3.6-BT-SIMA2021} and \eqref{eq-3.6-BT-SIMA2021-q<2} and completes the proof.
\end{proof}

We now prove Theorem \ref{thm-1.4BT4-mhd}, which follows directly as a simple consequence of Lemma \ref{lem3.1-BT-SIMA2021}.

\begin{proof}[Proof of Theorem \ref{thm-1.4BT4-mhd}]
The proof of Theorem \ref{thm-1.4BT4-mhd} is an adaption of the proof of \cite[Theorem 1.4]{BT-SIMA2021} for the Navier--Stokes equations to the MHD equations.

Observe that, by the assumption $R\ge1$, the estimate \eqref{eq-2.2-BT-SIMA2021}, and H\"older's inequality, we have
\EQS{\label{eq-3.17-BT-SIMA2021}
N_{q,R}^0(v_0,b_0) 
&\le \frac{C}{R}\bkt{\sum_{k\in\ZZ^3} \bke{\sum_{|i-kR|<R}u_i^2 + a_i^2}^{\frac{q}2}}^{\frac2q}\\
&\le \frac{C}{R}\bkt{ \bke{\sum_{k\in\ZZ^3} \sum_{|i-kR|<R}u_i^q R^{\bke{3-\frac6q}\frac{q}2}}^{\frac2q} + \bke{\sum_{k\in\ZZ^3} \sum_{|i-kR|<R}a_i^q R^{\bke{3-\frac6q}\frac{q}2}}^{\frac2q} }\\
&\le CR^{2-\frac6q}\bke{\norm{v_0}_{E^2_q}^2 + \norm{b_0}_{E^2_q}^2},
}
where $u_i = \norm{v_0}_{L^2(B_1(i))}$ and $a_i = \norm{b_0}_{L^2(B_1(i))}$ for $i\in\ZZ^3$.
Now, applying Lemma \ref{lem3.1-BT-SIMA2021}, we obtain the bound
\[
\norm{\esssup_{0\le t\le \la_R R^2}\int_{B_R(x_0R)} \frac{|v|^2 + |b|^2}2\, dx + \int_0^{\la_R R^2} \int_{B_R(x_0R)} \bke{|\nb v|^2 + |\nb b|^2} dxdt}_{\ell^{\frac{q}2}(x_0\in\ZZ^3)} 
\le C_1 A_{0,q}(R).
\]
Next, using the definition of $\la_R$ from Lemma \ref{lem3.1-BT-SIMA2021} and the estimate \eqref{eq-3.17-BT-SIMA2021}, we find
\EQN{
\la_R R^2 = \min\bke{\la_0R^2,\, \frac{\la_0R^2}{\bke{N_{q,R}^0(v_0,b_0)}^2}}
&\ge \min\bke{\la_0R^2,\, \frac{\la_0R^{\frac{12}q-2}}{C^2\bke{\norm{v_0}_{E_q^2}^4 + \norm{b_0}_{E_q^2}^4}}}\\
&\ge \frac{\la_1R^{\min\bke{2,\,\frac{12}q-2}}}{\bke{1+\norm{(v_0,b_0)}_{E_q^2\times E_q^2} }^4},
}
where $\la_1 := \la_0(1+C)^{-2}$.
Furthermore, from \eqref{eq-3.17-BT-SIMA2021}, we also have $A_{0,q}(R) = RN_{q,R}^0(v_0,b_0)\le CR^{3-\frac{6}q}\bke{\norm{(v_0,b_0)}_{E^2_q\times E^2_q}^2}$.
This yields the desired upper bound in the statement of Theorem \ref{thm-1.4BT4-mhd}.
\end{proof}

\begin{lemma}[Far-field regularity of local energy solutions with data in $E^2_q$]\label{lem-solution-in-E4q}
Assume $v_0,b_0\in E^2_q$ for some $q\ge1$ are divergence free.
If $(v,b)$ is a local energy solution on $\R^3\times(0,T_0)$ evolving from $(v_0,b_0)$ with $(v,b)\in{\bf LE}_q(0,T_0)$.
Then 
\[
(v,b)(t)\in E^4_q\times E^4_q\ \text{ for a.e. }t\in(0, T_0].
\]
\end{lemma}
\begin{proof}
By Lemma \ref{lem3.1-BT-SIMA2021} with $R=1$, we have
\EQ{\label{eq-5.2-BT-SIMA2021}
\norm{e_k}_{\ell^{q/2}(k\in\ZZ^3)}\le CA_{0,q},\quad
\norm{f_k}_{\ell^{q/3}(k\in\ZZ^3)}\le CA_{0,q}^{3/2},
}
where
\[
e_k = \esssup_{0\le t\le T_0} \int_{B_1(k)} \frac{|v|^2+|b|^2}2\, dx + \int_0^{T_0}\int_{B_1(k)}|\nb v|^2 + |\nb b|^2\, dxdt,
\]
\[
f_k = \int_0^{T_0}\int_{B_1(k)} |v|^3 + |b|^3 + |\pi-c_{k,1}(t)|^{3/2} dxdt,
\]
and
\[
A_{0,q} = \norm{\int_{B_1(k)} |v_0(x)|^2 + |b_0(x)|^2\, dx}_{\ell^{q/2}(k\in\ZZ^3)} = \norm{(v_0,b_0)}_{E^2_q\times E^2_q}^2.
\]

Since
\[
\lim_{R\to\infty} \norm{f_k}_{\ell^{q/3}(k\in\ZZ^3; |k|>R)} = 0,
\]
by Lemma \ref{lem2.3-BT-CPDE2020}, there exists $R_0>0$ such that
\[
v,b\in L^\infty\cap C_\loc(B_{R_0}^c\times[T_0/2,T_0])
\]
and
\[
\lim_{R\to\infty} \norm{v}_{L^\infty(B_{R_0}^c\times[T_0/2,T_0])} + \norm{b}_{L^\infty(B_{R_0}^c\times[T_0/2,T_0])} = 0.
\]
In fact, for $|k|>R_0$,
\[
\norm{v}_{L^\infty(B_1(k)\times[T_0/2,T_0])}^3 + \norm{b}_{L^\infty(B_1(k)\times[T_0/2,T_0])}^3 \le C\sum_{|k'-k|\le2} f_{k'}.
\]
Thus
\[
\norm{\norm{v}_{L^\infty(B_1(k)\times[T_0/2,T_0])} + \norm{b}_{L^\infty(B_1(k)\times[T_0/2,T_0])} }_{\ell^q(k\in\ZZ^3;|k|>R_0)}
\le C\norm{f_k}_{\ell^{q/3}(k\in\ZZ^3;|k|>R_0-2)}^{1/3}\ll1.
\]
By \eqref{eq-5.2-BT-SIMA2021} and Sobolev imbedding,
\[
v,b\in L^{8/3}(0,T_0; L^4(B_{R_0})).
\]
Thus
\[
(v,b)(t)\in E^4_q\times E^4_q\ \text{ for a.e. }t\in(0, T_0],
\]
completing the proof of the lemma.
\end{proof}

\subsection{Global existence}

In this section, we prove Theorem \ref{th4.8-mhd}.

For $q\ge2$, we achieve this by considering the perturbed MHD equations 
\EQS{\label{eq-4.31-BT-SIMA2021}
\pd_tv - \De v + v\cdot\nb v + u\cdot\nb v + v\cdot\nb u - b\cdot\nb b - a\cdot\nb b - b\cdot\nb a + \nb\pi &= 0,\\
\pd_tb - \De b + v\cdot\nb b + u\cdot\nb b + v\cdot\nb a - b\cdot\nb v - a\cdot\nb v - b\cdot\nb u \qquad\ \ &=0,\\
\nb\cdot v = \nb\cdot b &= 0.
}
where $u$ and $a$ are given divergence-free vector fields. 
A local energy solution to the perturbed MHD equations, \eqref{eq-4.31-BT-SIMA2021}, is a weak solution $(v,b)$ satisfying Definition \ref{def-local-energy-sol-mhd} with the obvious modifications, namely, $(v,b)$ and $\pi$ satisfy the perturbed system as distributions and also satisfy the perturbed local energy inequality.

For $1\le q<2$, we achieve this via the localized and regularized MHD equations:
\EQS{\label{eq-4.10-BLT-MathAnn2024}
\pd_tv^\ep - \De v^\ep + \bke{\mathcal{J}_\ep(v^\ep)\cdot\nb}(v^\ep\Phi_\ep) - \bke{\mathcal{J}_\ep(b^\ep)\cdot\nb}(b^\ep\Phi_\ep) + \nb\pi^\ep &= 0,\\
\pd_tb^\ep - \De b^\ep + \bke{\mathcal{J}_\ep(v^\ep)\cdot\nb}(b^\ep\Phi_\ep) - \bke{\mathcal{J}_\ep(b^\ep)\cdot\nb}(v^\ep\Phi_\ep) \qquad\quad &= 0,\\
\nb\cdot v^\ep = \nb\cdot b^\ep &= 0,
}
where $\mathcal{J}_\ep(f) = \eta_\ep*f$ for a spatial mollifier $\eta_\ep(x) = \ep^{-3}\eta(x/\ep)$ and $\Phi_\ep(x) = \Phi(\ep x)$ for a fixed radially decreasing cutoff function $\Phi$ satisfying $\Phi=1$ on $B_1(0)$ and $\supp(\Phi)\subset B_{3/2}(0)$.

\subsubsection{The case $q\ge2$}\label{sec-q>=2}

\begin{lemma}\label{lem-4.1-BT-SIMA2021}
Let $\ep\in(0,1]$ and $\de>0$ be given, let $T_0>0$, and let $\eta$ be a spatial mollifier in $\R^3$.
Assume that $v_0,b_0\in L^2$ are divergence free and that $u,a:\R^3\times[0,T_0]\to\R^3$ satisfies $\nb\cdot u = \nb\cdot a = 0$ and
\[
\esssup_{0<t\le T_0} \norm{(u,a)(t)}_{L^3_\uloc\times L^3_\uloc} < \de.
\]
Then there exist $T_{\ep,\de} = \min\bke{T_0,\, C(\ep)(\norm{(v_0,b_0)}_{E^2_q\times E^2_q}+\de)^{-2}}$ and a mild solution $(v_\ep, b_\ep)$ of the integral equation
\EQS{\label{eq-4.2-BT-SIMA2021}
v_\ep(t) &= e^{t\De}v_0 + \int_0^t e^{(t-s)\De}\mathbb{P}\nb\cdot\bke{(\eta_\ep * v_\ep)\otimes v_\ep - (\eta_\ep * b_\ep)\otimes b_\ep} ds + L_t^{(1)}(v_\ep,b_\ep),\\
b_\ep(t) &= e^{t\De}b_0 + \int_0^t e^{(t-s)\De}\mathbb{P}\nb\cdot\bke{(\eta_\ep * v_\ep)\otimes b_\ep - (\eta_\ep * b_\ep)\otimes v_\ep} ds + L_t^{(2)}(v_\ep,b_\ep),
}
for $0<t<T_{\ep,\de}$, where
\[
L_t^{(1)}(v_\ep,b_\ep) = \int_0^t e^{(t-s)\De}\mathbb{P}\nb\cdot\bke{(\eta_\ep * u)\otimes v_\ep + v_\ep\otimes(\eta_\ep * u) - (\eta_\ep * a)\otimes b_\ep - b_\ep\otimes(\eta_\ep * a)} ds,
\]
\[
L_t^{(2)}(v_\ep,b_\ep) = \int_0^t e^{(t-s)\De}\mathbb{P}\nb\cdot\bke{(\eta_\ep * u)\otimes b_\ep + b_\ep\otimes(\eta_\ep * u) - (\eta_\ep * a)\otimes v_\ep - b_\ep\otimes(\eta_\ep * u)} ds,
\]
with $(v_\ep, b_\ep)\in{\bf LE}_q(0,T_{\ep,\de})\cap C([0,T_{\ep,\de}]; L^2\times L^2)$, and $(v_\ep, b_\ep)$ satisfies
\EQS{\label{eq-4.3-BT-SIMA2021}
\norm{\bke{\esssup_{0<t<T_{\ep,\de}} \int_{B_1(k)} \bke{|v_\ep(x,t)|^2 + |b_\ep(x,t)|^2} dx}^{\frac12} }_{\ell^q}
&\le 2C\norm{(v_0,b_0)}_{E^2_q\times E^2_q}\quad \text{ and }\\
\sup_{0<t<T_{\ep,\de}} \norm{(v_\ep,b_\ep)(t)}_{L^2\times L^2} 
&\le 2C\norm{(v_0,b_0)}_{L^2\times L^2}
}
for a universal constant $C>0$.
This is the unique mild solution of \eqref{eq-4.2-BT-SIMA2021} in the class \eqref{eq-4.3-BT-SIMA2021}.
There exists a pressure $\pi_\ep$ so that $(v_\ep, b_\ep)$ and $\pi_\ep$ solve
\EQS{\label{eq-4.29-BT-SIMA2021}
\pd_tv_\ep - \De v_\ep &+ (\eta_\ep * v_\ep)\cdot\nb v_\ep + (\eta_\ep * u)\cdot\nb v_\ep + v_\ep\cdot\nb(\eta_\ep* u)\\
& - (\eta_\ep * b_\ep)\cdot\nb b_\ep - (\eta_\ep * a)\cdot\nb b_\ep - b_\ep\cdot\nb(\eta_\ep* a) + \nb\pi_\ep = 0,\\
\pd_tb_\ep - \De b_\ep &+ (\eta_\ep * v_\ep)\cdot\nb b_\ep + (\eta_\ep * u)\cdot\nb b_\ep + v_\ep\cdot\nb(\eta_\ep* a)\\
& - (\eta_\ep * b_\ep)\cdot\nb v_\ep - (\eta_\ep * a)\cdot\nb v_\ep - b_\ep\cdot\nb(\eta_\ep * u) \qquad\ \ =0,\\
&\qquad\qquad\qquad\qquad\qquad\qquad\qquad\quad\,
\nb\cdot v_\ep = \nb\cdot b_\ep = 0,
}
in the weak sense on $\R^3\times(0, T_{\ep,\de})$.
Finally, $(v_\ep, b_\ep)$ and $\pi_\ep$ are smooth by the interior regularity of the Stokes equations with smooth coefficients.
\end{lemma}

\begin{proof}
The proof of the lemma is an adaption of the proof of \cite[Lemma 4.1]{BT-SIMA2021} for the Navier--Stokes equations to the MHD equations.

Note that $v_0,b_0\in E^2_q$ since $L^2\subset E^2_q$.
We begin by establishing estimates for the iterates of the Picard scheme.
Define the initial iterates as
\[
(v_\ep^1, b_\ep^1) = (e^{t\De} v_0, e^{t\De} b_0),
\]
and for $n>1$, set
\EQN{
v_\ep^n(t) &= e^{t\De}v_0 + \int_0^t e^{(t-s)\De}\mathbb{P}\nb\cdot\bke{(\eta_\ep * v_\ep^{n-1})\otimes v_\ep^{n-1} - (\eta_\ep * b_\ep^{n-1})\otimes b_\ep^{n-1}} ds + L_t^{(1)}(v_\ep^{n-1},b_\ep^{n-1}),\\
b_\ep^n(t) &= e^{t\De}b_0 + \int_0^t e^{(t-s)\De}\mathbb{P}\nb\cdot\bke{(\eta_\ep * v_\ep^{n-1})\otimes b_\ep^{n-1} - (\eta_\ep * b_\ep^{n-1})\otimes v_\ep^{n-1}} ds + L_t^{(2)}(v_\ep^{n-1},b_\ep^{n-1}).
}

For the initial iterates, it follows from the same approach obtaining \cite[(4.7)]{BT-SIMA2021} that
\EQS{\label{eq-4.5,4.6-BT-SIMA2021}
\norm{\bke{\sup_{0<t<1} \int_{B_1(k)} \abs{v_\ep^1(x,t)}^2 dx}^{\frac12}}_{\ell^q}
\le C\norm{v_0}_{E^2_q}
\ \text{ and }\ 
\norm{\bke{\sup_{0<t<1} \int_{B_1(k)} \abs{b_\ep^1(x,t)}^2 dx}^{\frac12}}_{\ell^q}
\le C\norm{b_0}_{E^2_q}.
}

For the $n$th iterates, we use the assumption that
\EQN{
\norm{\bke{\sup_{0<t<T_\ep} \int_{B_1(k)} \bke{|v_\ep^{n-1}(x,t)|^2 + |b_\ep^{n-1}(x,t)|^2} dx}^{\frac12}}_{\ell^q}
< 2C\norm{(v_0, b_0)}_{E^2_q\times E^2_q}.
}
We have
\EQS{\label{eq-4.10-BT-SIMA2021}
\int_{B_1(k)} \abs{v_\ep^n(x,t)}^2 dx 
&\le \int_{B_1(k)} \abs{e^{t\De} v_0(x)}^2 dx + I^{(1)}(k) + J^{(1)}(k),\\
\int_{B_1(k)} \abs{b_\ep^n(x,t)}^2 dx 
&\le \int_{B_1(k)} \abs{e^{t\De} b_0(x)}^2 dx + I^{(2)}(k) + J^{(2)}(k),
}
where 
\EQN{
I^{(1)}(k) &= \int_{B_1(k)} \abs{\int_0^t e^{(t-s)\De}\mathbb{P}\nb\cdot\bke{(\eta_\ep * v_\ep^{n-1})\otimes v_\ep^{n-1} - (\eta_\ep * b_\ep^{n-1})\otimes b_\ep^{n-1}} ds}^2 dx,\\
J^{(1)}(k) &= \int_{B_1(k)} \abs{L_t^{(1)}(v_\ep^{n-1}, b_\ep^{n-1})}^2 dx,\\
I^{(2)}(k) &= \int_{B_1(k)} \abs{\int_0^t e^{(t-s)\De}\mathbb{P}\nb\cdot\bke{(\eta_\ep * v_\ep^{n-1})\otimes b_\ep^{n-1} - (\eta_\ep * b_\ep^{n-1})\otimes v_\ep^{n-1}} ds}^2 dx,\\
J^{(2)}(k) &= \int_{B_1(k)} \abs{L_t^{(2)}(v_\ep^{n-1}, b_\ep^{n-1})}^2 dx.
}
The first two terms on the right-hand side of \eqref{eq-4.10-BT-SIMA2021} have already been estimated in \eqref{eq-4.5,4.6-BT-SIMA2021}.
For $I^{(1)}(k)$ and $I^{(2)}(k)$, using the same technique deriving \cite[(4.11), (4.12)]{BT-SIMA2021}, we get
\EQS{\label{eq-4.11,12-BT-SIMA2021}
I^{(1)}(k)& + I^{(2)}(k) \\
&\le C(\ep) t^2 \bkt{ \bke{\sup_{0<s<t}\norm{v_\ep^{n-1}}_{E^2_q}^2} \bke{(\wt{K}*u^{n-1})(k)}^2 + \bke{\sup_{0<s<t}\norm{b_\ep^{n-1}}_{E^2_q}^2} \bke{(\wt{K}*a^{n-1})(k)}^2 }\\
&\quad + C(\ep) t \left[ \sup_{0<s<t} \norm{v_\ep^{n-1}}_{E^2_q}^2 \sup_{0<s<t} \sum_{|k-k'|<8} \norm{v_\ep^{n-1}}_{L^2(B_1(k'))}^2\right.\\
&\qquad\qquad\qquad \left.+ \sup_{0<s<t} \norm{b_\ep^{n-1}}_{E^2_q}^2 \sup_{0<s<t} \sum_{|k-k'|<8} \norm{b_\ep^{n-1}}_{L^2(B_1(k'))}^2 \right],
}
where
\[
u_k^{n-1} = \bke{\sup_{0<s<t} \int_{B_1(k)} \abs{v_\ep^{n-1}(y)}^2 dy}^{1/2},\qquad
a_k^{n-1} = \bke{\sup_{0<s<t} \int_{B_1(k)} \abs{b_\ep^{n-1}(y)}^2 dy}^{1/2},
\]
and
\EQ{\label{eq-def-wt-K}
\wt{K}(k) = |k|^{-4}\ \text{ if } |k|\ge4,\qquad
\wt{K}(k) = 0 \ \text{ otherwise.}
}

We next estimate the terms $J^{(1)}(k)$ and $J^{(2)}(k)$.
Using the same argument for \cite[(4.13), (4.14)]{BT-SIMA2021} yields
\EQS{\label{eq-4.13,14-BT-SIMA2021}
J^{(1)}(k) + J^{(2)}(k) 
&\le C\de^2 t^2\bke{(\wt{K}*u^{n-1})(k)+ (\wt{K}*a^{n-1})(k)}^2\\
&\quad + C(\ep) t\de^2\bke{\norm{v_\ep^{n-1}}_{L^2(B_4(k))}^2 + \norm{b_\ep^{n-1}}_{L^2(B_4(k))}^2 }.
}

Combining \eqref{eq-4.5,4.6-BT-SIMA2021}, \eqref{eq-4.11,12-BT-SIMA2021}, and \eqref{eq-4.13,14-BT-SIMA2021}, we have

\EQS{\label{eq-4.15-BT-SIMA2021}
&\int_{B_1(k)} \bke{\abs{v_\ep^n(x,t)}^2 + \abs{b_\ep^n(x,t)}^2 }dx \\
&\ \le C\sum_{|k-k'|\le4} \int_{B_1(k')} \bke{|v_0(x)|^2 + |b_0(x)|^2} dx + C\bkt{\bke{\wt{K} * (u+a)}(k)}^2\\
&\ \quad + C(\ep) t^2 \bkt{ \bke{\sup_{0<s<t}\norm{v_\ep^{n-1}}_{E^2_q}^2} \bke{(\wt{K}*u^{n-1})(k)}^2 + \bke{\sup_{0<s<t}\norm{b_\ep^{n-1}}_{E^2_q}^2} \bke{(\wt{K}*a^{n-1})(k)}^2 }\\
&\ \quad + C(\ep) t \left[ \sup_{0<s<t} \norm{v_\ep^{n-1}}_{E^2_q}^2 \sup_{0<s<t} \sum_{|k-k'|<8} \norm{v_\ep^{n-1}}_{L^2(B_1(k'))}^2\right.\\
&\ \qquad\qquad\qquad \left.+ \sup_{0<s<t} \norm{b_\ep^{n-1}}_{E^2_q}^2 \sup_{0<s<t} \sum_{|k-k'|<8} \norm{b_\ep^{n-1}}_{L^2(B_1(k'))}^2 \right]\\
&\ \quad + C\de^2 t^2\bke{(\wt{K}*u^{n-1})(k)+ (\wt{K}*a^{n-1})(k)}^2 + C(\ep) t\de^2\bke{\norm{v_\ep^{n-1}}_{L^2(B_4(k))}^2 + \norm{b_\ep^{n-1}}_{L^2(B_4(k))}^2 }.
}
Taking the supremum in time of the left-hand side of \eqref{eq-4.15-BT-SIMA2021}, applying the $\ell^{\frac{q}2}$ norm, using Young's convolution inequality, and raising everything to the $1/2$ power yields
\EQS{\label{eq-4.16-BT-SIMA2021}
\norm{(v_\ep^n, b_\ep^n)}_{{\bf LE}_q^\flat(0,t)}
&\le C\norm{(v_0,b_0)}_{E^2_q\times E^2_q} + C(\ep) t^{\frac12} \norm{(v_\ep^{n-1}, b_\ep^{n-1})}_{{\bf LE}_q^\flat(0,t)}^2\\
&\quad + C(\ep)\de t^{\frac12} \norm{(v_\ep^{n-1}, b_\ep^{n-1})}_{{\bf LE}_q^\flat(0,t)},
}
where $\norm{\,\cdot\,}_{{\bf LE}_q^\flat(I)}$ is the first part of the norm $\norm{\,\cdot\,}_{{\bf LE}_q(I)}$ defined by $\norm{(v,b)}_{{\bf LE}_q^\flat(I)} = \norm{(v,b)}_{E^{\infty,2}_{T,q}\times E^{\infty,2}_{T,q}}$.
So, if $t$ is small as determined by $C(\ep)$, $\ep$, and $\norm{(v_0,b_0)}_{E^2_q\times E^2_q}$ (but independently of $n$), $t\le C(\ep)/\bke{\norm{(v_0, b_0)}_{E^2_q\times E^2_q}^2 + \de^2}$,
then the right-hand side of \eqref{eq-4.16-BT-SIMA2021} is controlled by $2C\norm{(v_0,b_0)}_{E^2_q\times E^2_q}$.
This establishes a uniform-in-$n$ bound for $(v_\ep^n, b_\ep^n)$.

These uniform bounds and the estimation methods above allow us to show the difference estimate
\[
\norm{(v_\ep^{n+1}, b_\ep^{n+1}) - (v_\ep^n, b_\ep^n)}_{{\bf LE}_q^\flat(0,t)} 
\le C(\ep)\sqrt{t} \bke{\norm{(v_0,b_0)}_{E^2_q\times E^2_q} + \de} \norm{ (v_\ep^n, b_\ep^n) - (v_\ep^{n-1}, b_\ep^{n-1}) }_{{\bf LE}_q^\flat(0,t)}.
\]
Thus, if $t$ is sufficiently small, then $(v_\ep^n, b_\ep^n)$ is a Cauchy sequence in ${\bf LE}_q^\flat(0,t)$ norm and converges to a limit $(v_\ep^n, b_\ep^n)$ in the sense that 
\[
\norm{ (v_\ep^n, b_\ep^n) - (v_\ep, b_\ep) }_{{\bf LE}_q^\flat(0,t)} \to 0,\quad\text{ as }n\to\infty.
\]
This convergence implies $(v_\ep, b_\ep)$ satisfies \eqref{eq-4.2-BT-SIMA2021} and \eqref{eq-4.3-BT-SIMA2021}.

Uniqueness in the class \eqref{eq-4.3-BT-SIMA2021} follows from the same difference estimates as before.
Suppose  $(v_1,b_1)$ and $(v_2,b_2)$ are two mild solutions of \eqref{eq-4.2-BT-SIMA2021} satisfying \eqref{eq-4.3-BT-SIMA2021}. 
Then
\[
\norm{ (v_1, b_1) - (v_2, b_2) }_{{\bf LE}_q^\flat(0,t)} \le C(\ep)\sqrt{t}(\norm{(v_0,b_0)}_{E^2_q\times E^2_q} + \de) \norm{ (v_1, b_1) - (v_2, b_2) }_{{\bf LE}_q^\flat(0,t)}.
\]
Hence, if $t>0$ is sufficiently small, we conclude that $(v_1,b_1) = (v_2,b_2)$.

We now recover a pressure associated to $(v_\ep,b_\ep)$. 
It is known that $(v_\ep,b_\ep)\in L^\infty(0,T_{\ep,\de};L^2\times L^2)$ for some $T_{\ep,\de}>0$, with $\norm{(v_\ep,b_\ep)}_{L^\infty(0,T_{\ep,\de};L^2\times L^2)}\le2C\norm{(v_0,b_0)}_{L^2\times L^2}$.
Therefore, the following nonlinear terms belongs to $L^\infty(0,T_{\ep,\de};L^2)$:
\[
\eta_\ep * v_\ep\otimes v_\ep + \eta_\ep * u\otimes v_\ep + v_\ep\otimes \eta_\ep * u 
-\eta_\ep * b_\ep\otimes b - \eta_\ep * a\otimes b_\ep - b_\ep\otimes \eta_\ep * a.
\]
Consequently $\pi_\ep = (-\De)^{-1} \pd_i\pd_j (\eta_\ep * v_\ep\otimes v_\ep + \eta_\ep * u\otimes v_\ep + v_\ep\otimes \eta_\ep * u -\eta_\ep * b_\ep\otimes b - \eta_\ep * a\otimes b_\ep - b_\ep\otimes \eta_\ep * a)$ is well-defined.
It follows that $v_\ep-e^{t\De}v_0$ solves the Stokes system with pressure $\pi_\ep$ and forcing term equal to $\nb\cdot(\eta_\ep * v_\ep\otimes v_\ep + \eta_\ep * u\otimes v_\ep + v_\ep\otimes \eta_\ep * u -\eta_\ep * b_\ep\otimes b - \eta_\ep * a\otimes b_\ep - b_\ep\otimes \eta_\ep * a)$.
Adding back the linear term $e^{t\De}v_0$, we see that $(v_\ep,\pi_\ep)$ solve the perturbed, regularized Navier--Stokes equations.
The local pressure expansion \eqref{eq-pi-formula-mhd} follows from the definition of $\pi_\ep$.

We now establish the estimate
\EQ{\label{eq-4.17-BT-SIMA2021}
\norm{\int_0^t\int_{B_1(k)} \bke{|\nb v_\ep|^2 + |\nb b_\ep|^2} dxds}_{\ell^{q/2}(k)} < \infty.
}
This follow from the local energy equality satisfied by $(v_\ep,b_\ep)$ and the associated pressure $\pi_\ep$, valid for $t\le T_{\ep,\de}$ due to the regularity of the solutions due to smoothness and convergence to the data in $L^2_\loc$:
\EQS{
&\int_{B_1(k)} \bke{|v_\ep|^2 + |b_\ep|^2}(x,t)\phi(x-k)\, dx + 2\int_0^t\int \bke{|\nb v_\ep|^2 + |\nb b_\ep|^2}\phi(x-k)\, dxds\\
&\quad = \int \bke{|v_0|^2 + |b_0|^2}\phi(x-k)\, dx + \int_0^t\int \bke{|v_\ep|^2 + |b_\ep|^2} \De\phi(x-k)\, dxds\\
&\quad\quad + \int_0^t \int \bke{|v_\ep|^2 + |b_\ep|^2} (\eta_\ep * v_\ep + \eta_\ep * u)\cdot\nb\phi(x-k)\, dxds\\
&\quad\quad -2\int_0^t\int (v_\ep\cdot\nb(\eta_\ep * u))\cdot v_\ep\phi(x-k)\, dxds
 +2\int_0^t\int (b_\ep\cdot\nb(\eta_\ep * a))\cdot v_\ep\phi(x-k)\, dxds\\
&\quad\quad -2\int_0^t\int (v_\ep\cdot\nb(\eta_\ep * a))\cdot b_\ep\phi(x-k)\, dxds
 +2\int_0^t\int (b_\ep\cdot\nb(\eta_\ep * u))\cdot b_\ep\phi(x-k)\, dxds\\
&\quad\quad + 2\int_0^t\int \pi_\ep (v_\ep\cdot\nb\phi(x-k))\, dxds
 -2\int_0^t\int (v_\ep\cdot b_\ep)(\eta_\ep * b_\ep + \eta_\ep * a)\cdot \nb\phi(x-k)\, dxds.
}
To estimate the nonlinear terms, we use the bound $\norm{\eta_\ep * f}_{L^\infty(B_2(k))} \le C(\ep)\norm{f}_{L^2_\uloc}$ for $\ep\le1$.
Hence, we have
\EQS{\label{eq-4.19-BT-SIMA2021}
\int_0^t\int &\bke{|v_\ep|^2 + |b_\ep|^2} (\eta_\ep * v_\ep)\cdot\nb\phi(x-k)\, dxds\\
&\le C(\ep) \norm{v_\ep}_{L^2\infty L^2_\uloc} \esssup_{0<s<t} \sum_{k'\sim k} \int_{B_1(k')} \bke{|v_\ep|^2 + |b_\ep|^2} dx,
}
and 
\EQS{\label{eq-4.19-BT-SIMA2021-b}
\int_0^t\int &(v_\ep\cdot b_\ep) (\eta_\ep * b_\ep)\cdot\nb\phi(x-k)\, dxds\\
&\lec \int_0^t\int \bke{|v_\ep|^2 + |b_\ep|^2}  (\eta_\ep * b_\ep)\cdot\nb\phi(x-k)\, dxds\\
&\le C(\ep) \norm{b_\ep}_{L^2\infty L^2_\uloc} \esssup_{0<s<t} \sum_{k'\sim k} \int_{B_1(k')} \bke{|v_\ep|^2 + |b_\ep|^2} dx.
}
From the assumptions on $u$ and $a$, we also have
\[
\norm{\eta_\ep * u}_{L^\infty} + \norm{\nb(\eta_\ep * u)}_{L^\infty}
\le C(\ep) \norm{u}_{L^\infty(0,T_0;L^3_\uloc)}
\le C(\ep) \de,
\]
and
\[
\norm{\eta_\ep * a}_{L^\infty} + \norm{\nb(\eta_\ep * a)}_{L^\infty}
\le C(\ep) \norm{a}_{L^\infty(0,T_0;L^3_\uloc)}
\le C(\ep) \de.
\]
These imply the following estimates:
\EQS{\label{eq-4.20-BT-SIMA2021}
\int_0^t\int& \bke{|v_\ep|^2 + |b_\ep|^2 } (\eta_\ep * u)\cdot\nb\phi(x-k)\, dxds\\
&\le C(\ep)\de\esssup_{0<s<t}\sum_{k'\sim k} \int_{B_1(k')} \bke{|v_\ep(x,s)|^2 + |b_\ep(x,s)|^2} dx,
}
\EQS{\label{eq-4.20-BT-SIMA2021-b}
\int_0^t\int& (v_\ep\cdot b_\ep) (\eta_\ep * a)\cdot\nb\phi(x-k)\, dxds\\
&\lec \int_0^t\int \bke{v_\ep|^2 + |b_\ep|^2 } (\eta_\ep * a)\cdot\nb\phi(x-k)\, dxds\\
&\le C(\ep)\de\esssup_{0<s<t}\sum_{k'\sim k} \int_{B_1(k')} \bke{|v_\ep(x,s)|^2 + |b_\ep(x,s)|^2} dx,
}
\EQS{\label{eq-4.21-BT-SIMA2021-vuv}
\int_0^t\int (v_\ep\cdot\nb(\eta_\ep * u))\cdot v_\ep\phi(x-k)\, dxds
\le C(\ep)\de\esssup_{0<s<t}\sum_{k'\sim k}\int_{B_1(k')} |v_\ep|^2\, dx,
}
\EQS{\label{eq-4.21-BT-SIMA2021-bav}
\int_0^t\int (b_\ep\cdot\nb(\eta_\ep * a))\cdot v_\ep\phi(x-k)\, dxds
&\le C(\ep)\de\esssup_{0<s<t}\sum_{k'\sim k}\int_{B_1(k')} |b_\ep||v_\ep|\, dx\\
&\le C(\ep)\de\esssup_{0<s<t}\sum_{k'\sim k}\int_{B_1(k')} \bke{|b_\ep|^2 + |v_\ep|^2} dx,
}
\EQS{\label{eq-4.21-BT-SIMA2021-vab}
\int_0^t\int (v_\ep\cdot\nb(\eta_\ep * a))\cdot b_\ep\phi(x-k)\, dxds
&\le C(\ep)\de\esssup_{0<s<t}\sum_{k'\sim k}\int_{B_1(k')} |v_\ep||b_\ep|\, dx\\
&\le C(\ep)\de\esssup_{0<s<t}\sum_{k'\sim k}\int_{B_1(k')} \bke{|v_\ep|^2 + |b_\ep|^2} dx,
}
and
\EQS{\label{eq-4.21-BT-SIMA2021-bub}
\int_0^t\int (b_\ep\cdot\nb(\eta_\ep * u))\cdot b_\ep\phi(x-k)\, dxds
\le C(\ep)\de\esssup_{0<s<t}\sum_{k'\sim k}\int_{B_1(k')} |b_\ep|^2\, dx.
}

The pressure satisfies the local pressure expansion \eqref{eq-pi-formula-mhd}, which allows us--after incorporating an additive constant--to express it as a sum of two components: $\pi_\ep(x,t) + c = \pi_{\ep,{\rm near}} + \pi_{\ep,{\rm far}}$, where $\pi_{\ep,{\rm near}}$ is a Calderon--Zygmund operator applied to a localized term, and $\pi_{\ep,{\rm far}}$ is a nonsingular integral operator acting on data supported away from the ball $B_2(k)$.
Due to the structure of the pressure term in the local energy inequality, the additive constant $c$ plays no role and may be disregarded.
By applying the Calderon--Zygmund inequality, the contribution from $\pi_{\ep,{\rm near}}$ to the local energy inequality can be estimated in the same way as the nonlinear and perturbative terms discussed earlier. 
Specifically, it is controlled by the right-hand sides of estimates \eqref{eq-4.19-BT-SIMA2021} through \eqref{eq-4.21-BT-SIMA2021-bub}.
We are thus left to estimate only the far-filed component $\pi_{\ep,{\rm far}}$.
In $B_2(k)\times(0,T_0)$, 
\EQN{
|\pi_{\ep,{\rm far}}|
&\le C \sum_{k'\in\ZZ^3; |k'-k|>4} \frac1{|k-k'|^4} \int_{B_2(k')} \bke{|v_\ep||\eta_\ep * v_\ep| + |\eta_\ep * u| |v_\ep| + |b_\ep||\eta_\ep * b_\ep| + |\eta_\ep * a| |b_\ep|} dy\\
&\le C(\ep) \sum_{k'\in\ZZ^3; |k'-k|>4} \frac1{|k-k'|^4} \left(\norm{v_\ep}_{L^2(B_3(k'))}^2 + \norm{v_\ep}_{L^2(B_3(k'))} \norm{u}_{L^\infty(0,T_0;L^3_\uloc)}\right.\\
&\qquad\qquad\qquad\qquad\qquad\qquad\qquad\left. + \norm{b_\ep}_{L^2(B_3(k'))}^2 + \norm{b_\ep}_{L^2(B_3(k'))} \norm{a}_{L^\infty(0,T_0;L^3_\uloc)} \right).
}
Therefore, the contribution of $\pi_{\ep,{\rm far}}$ to the local energy equation satisfies
\EQS{\label{eq-4.22-BT-SIMA2021}
\int_0^t&\int_{B_2(k)} \pi_{\ep,{\rm far}}(v_\ep\cdot\nb\phi(x-k))\, dxds\\
&\le C(\ep) T_0 \norm{v_\ep}_{L^\infty L^2_\uloc} \esssup_{0<s<t} \sum_{k'\in\ZZ^3;|k'-k|>4} \frac1{|k-k'|^4} \int_{B_3(k')} \bke{|v_\ep|^2 + |b_\ep|^2} dy\\
&\quad + C(\ep)\de T_0 \norm{v_\ep}_{L^\infty L^2(B_2(k))} \esssup_{0<s<t} \sum_{k'\in\ZZ^3;|k'-k|>4} \frac1{|k-k'|^4} \bke{\norm{v_\ep}_{L^2(B_3(k'))} + \norm{b_\ep}_{L^2(B_3(k'))}}.
}

Taking the essential supremum in $t$, raising both sides of \eqref{eq-4.22-BT-SIMA2021} to the power $q/2$, and summing over $k\in\ZZ^3$, we apply H\"older's and Young's inequalities to control the far-field pressure term.
This establishes the estimate \eqref{eq-4.17-BT-SIMA2021}, completing the proof of Lemma \ref{lem-4.1-BT-SIMA2021}.
\end{proof}

\begin{lemma}\label{lem-4.2-BT-SIMA2021}
Assume that $v_0,b_0\in E^2_q$, for some $2\le q<\infty$, are divergence free.
There exists a small universal constant $c_0$ so that for all $\de\in(0,c_0]$ and for all divergence free vector fields $u,a:\R^3\times[0,T_0]\to\R^3$ satisfying 
\[
\esssup_{0<t\le T_0} \norm{(u,a)(t)}_{L^3_\uloc\times L^3_\uloc} < \de
\]
for some $T_0>0$ and if, additionally, a given local energy solution $(v,b)$ to the perturbed MHD equations, \eqref{eq-4.31-BT-SIMA2021}, satisfies
\EQN{
\norm{\esssup_{0\le t\le T_0} \int_{B_1(x_0)} \bke{|v|^2 + |b|^2} dx + \int_0^{T_0} \int_{B_1(x_0)} \bke{|\nb v|^2 + |\nb b|^2} dxdt}_{\ell^{\frac{q}2}(x_0\in\ZZ^3)} < \infty,
}
then there are positive universal constants $C_1$ and $\la_0<1$ such that
\EQN{
\norm{\esssup_{0\le t\le \la} \int_{B_1(x_0)} \frac{|v|^2 + |b|^2}2 dx + \int_0^{\la} \int_{B_1(x_0)} \bke{|\nb v|^2 + |\nb b|^2} dxdt}_{\ell^{\frac{q}2}(x_0\in\ZZ^3)} < C_1 A_{0,q},
}
where
\[
A_{0,q} = \norm{\int_{B_1(x_0)} \bke{|v_0|^2 + |b_0|^2} dx}_{\ell^{\frac{q}2}(x_0\in\ZZ^3)},\quad
\la = \min\bke{T_0,\,\la_0,\,\frac{\la_0}{A_{0,q}^2}}.
\]
Consequently,
\EQN{
\norm{\int_0^{\la} \int_{B_1(x_0)} |v|^{\frac{10}3} + |b|^{\frac{10}3} + \abs{\pi - c_{x_0}(t)}^{\frac53} dxdt}_{\ell^{\frac{3q}{10}}(x_0\in\ZZ^3)} \le C A_{0,q}^{\frac53}.
}
\end{lemma}

\begin{proof}
The proof of the lemma is an adaption of the proof of \cite[Lemma 4.2]{BT-SIMA2021} for the Navier--Stokes equations to the MHD equations.

Once the perturbation terms in the local energy inequality for $(v,b)$ are estimated, the proof proceeds identically to that of Lemma \ref{lem3.1-BT-SIMA2021} with $R=1$ and $\la_R = \la$.

The linear terms in the perturbed local energy inequality can be estimated as follows:
\EQN{
\int_0^\la& \int \bke{u\cdot\nb v + v\cdot\nb u - a\cdot\nb b - b\cdot\nb a}\cdot\bke{\phi(x-\ka)v} dxdt\\
&= \int_0^\la \int \left[u\cdot\nb v\cdot(\phi(x-\ka)v) - v\otimes u:\nb(\phi(x-\ka)v)\right.\\
&\qquad\qquad\ \left. - a\cdot\nb b\cdot(\phi(x-\ka)v) + b\otimes a:\nb(\phi(x-\ka)v)\right] dxdt\\
&\le C \int_0^\la\int_{B_2(\ka)} \bkt{|u|(|v|^2+|v||\nb b|) + |a|(|v||b|+|v||\nb b| + |b||\nb v|)} dxdt\\
&\le C \norm{u}_{L^\infty L^3_\uloc} \int_0^\la \norm{v}_{L^6(B_2(\ka))} \bke{\norm{v}_{L^2(B_2(\ka))} + \norm{\nb v}_{L^2(B_2(\ka))} } dt\\
&\quad + C \norm{a}_{L^\infty L^3_\uloc} \int_0^\la \norm{v}_{L^6(B_2(\ka))} \bke{\norm{b}_{L^2(B_2(\ka))} + \norm{\nb b}_{L^2(B_2(\ka))} } dt\\
&\quad + C \norm{a}_{L^\infty L^3_\uloc} \int_0^\la \norm{b}_{L^6(B_2(\ka))} \norm{\nb v}_{L^2(B_2(\ka))} dt\\
&\le C\la\de\esssup_{0<t<\la} \sum_{\ka'\sim\ka} \int_{B_1(\ka')} \bke{|v(x,t)|^2 + |b(x,t)|^2} dx\\
&\quad + C\de\esssup_{0<t<\la} \sum_{\ka'\sim\ka} \int_0^\la \int_{B_1(\ka')} \bke{|\nb v(x,t)|^2 + |\nb b(x,t)|^2} dxdt,
}
and
\EQN{
\int_0^\la& \int \bke{u\cdot\nb b + v\cdot\nb a - a\cdot\nb v - b\cdot\nb u}\cdot\bke{\phi(x-\ka)b} dxdt\\
&= \int_0^\la \int \left[u\cdot\nb b\cdot(\phi(x-\ka)b) - v\otimes a:\nb(\phi(x-\ka)b)\right.\\
&\qquad\qquad\ \left. - a\cdot\nb v\cdot(\phi(x-\ka)b) + b\otimes u:\nb(\phi(x-\ka)b)\right] dxdt\\
&\le C \int_0^\la\int_{B_2(\ka)} \bkt{|u|(|b|^2+|b||\nb b|) + |a|(|v||b|+|v||\nb b| + |b||\nb v|)} dxdt\\
&\le C \norm{u}_{L^\infty L^3_\uloc} \int_0^\la \norm{b}_{L^6(B_2(\ka))} \bke{\norm{b}_{L^2(B_2(\ka))} + \norm{\nb b}_{L^2(B_2(\ka))} } dt\\
&\quad + C \norm{a}_{L^\infty L^3_\uloc} \int_0^\la \norm{v}_{L^6(B_2(\ka))} \bke{\norm{b}_{L^2(B_2(\ka))} + \norm{\nb b}_{L^2(B_2(\ka))} } dt\\
&\quad + C \norm{a}_{L^\infty L^3_\uloc} \int_0^\la \norm{b}_{L^6(B_2(\ka))} \norm{\nb v}_{L^2(B_2(\ka))} dt\\
&\le C\la\de\esssup_{0<t<\la} \sum_{\ka'\sim\ka} \int_{B_1(\ka')} \bke{|v(x,t)|^2 + |b(x,t)|^2} dx\\
&\quad + C\de\esssup_{0<t<\la} \sum_{\ka'\sim\ka} \int_0^\la \int_{B_1(\ka')} \bke{|\nb v(x,t)|^2 + |\nb b(x,t)|^2} dxdt.
}

The pressure can be decomposed into local and far-field contributions.
The local part contains new terms that are handled exactly as in the previous estimates, once the Calderon--Zygmund inequality is applied.
The far-field pressure splits as $\pi_{\rm far} = \pi_{{\rm far}, (v, b)} + \pi_{{\rm far}, (u, a)}$, where $\pi_{{\rm far}, (v, b)}$ matches the far-field term treated in the proof of Lemma \ref{lem3.1-BT-SIMA2021}, and $\pi_{{\rm far}, (u, a)}$ is the remaining contribution.
Since the estimate for $\pi_{{\rm far}, (v, b)}$ is already established in the proof of Lemma \ref{lem3.1-BT-SIMA2021}, we focus on bounding $\pi_{{\rm far}, (u, a)}$ in $B_2(\ka)\times(0,T)$.
Specifically, we have
\EQN{
\abs{ \pi_{{\rm far}, (u, a)}(x,t) } 
&\le C\int \frac1{|\ka-y|^4} \bke{|v(y,t)||u(y,t)| + |b(y,t)||a(y,t)|} \bke{1-\chi_4(y-\ka)} dy\\
&\le C\de \wt{K} * e_{\la}^{1/2}(\ka),
}
where $\wt{K}$ is defined in \eqref{eq-def-wt-K}, and
\[
e_\la(\ka) = \esssup_{0\le t\le\la} \int_{B_1(\ka)} \bke{|v(x,t)|^2 + |b(x,t)|^2} dx + \int_0^\la \int_{B_1(\ka)} \bke{|\nb v(x,t)|^2 + |\nb b(x,t)|^2} dxdt.
\]
This yields the estimate:
\EQN{
\int_0^\la\int& \pi_{{\rm far},(u,a)}(x,s) v(x,s)\cdot\nb\phi(x-\ka)\, dxds\\
&\le C\int_0^\la\int_{B_2(\ka)} \de^{1/2} (\wt{K} * e_{\la}^{1/2})\de^{1/2} |v|\, dxds\\
&\le C\de\int_0^\la\int_{B_2(\ka)} (\wt{K} * e_{\la}^{1/2})^2\, dxds + \de\int_0^\la\int_{B_2(\ka)} |v|^2\, dxds\\
&\le C\de\la\bke{(\wt{K} * e_\la^{1/2})(\ka)}^2 + \de\la\esssup_{0<t<\la} \sum_{|\ka-\ka'|<4}\int_{B_1(\ka')} |v|^2\, dx.
}

Combining the above estimates with the argument in the proof of Lemma \ref{lem3.1-BT-SIMA2021} (see \eqref{eq-3.12-BT-SIMA2021}), we obtain
\EQS{
e_\la(\ka) &\le \int \bke{|v_0|^2 + |b_0|^2} \phi(x-\ka)\, dx + C\la\sum_{\ka'\in\ZZ^3; |\ka'-\ka|\le2} e_\la(\ka')\\
&\quad +C\la^{1/4}\sum_{\ka'\in\ZZ^3; |\ka'-\ka|\le10} \bke{e_\la(\ka')}^{3/2} + C\la^{1/4}\bke{(\wt{K} * e_\la)(\ka)}^{3/2}\\
&\quad + C\de\sum_{|\ka-\ka'|<10} e_\la(\ka') + C\de\la^{1/4}\bke{(\wt{K} * e_\la^{1/2})(\ka)}^2,
}
where we are using $\la\le\la_0\le1$.
The first two lines above coincide exactly with the estimates in the proof of Lemma \ref{lem3.1-BT-SIMA2021}, so we focus on the two additional terms in the last line.
To control the final term, we apply the $\ell^{q/2}$ norm:
\[
\norm{\bke{(\wt{K} * e_\la^{1/2})(\ka)}^2}_{\ell^{q/q}} = \norm{(\wt{K} * e_\la^{1/2})(\ka)}_{\ell^{q/q}}^2 \le C\norm{\wt{K}}_{\ell^1} \norm{e_\la}_{\ell^{q/2}}.
\]
We now choose $c_0$ (which bounds $\de$) sufficiently small so that, after taking the $\ell^{q/2}$ norm of both sides of the inequality, the $\de$-weighted terms on the right can be absorbed into the left-hand side.
With this absorption, the remaining terms are exactly as in the proof of Lemma \ref{lem3.1-BT-SIMA2021}, and the conclusion follows by the same argument.
\end{proof}

\begin{lemma}\label{lem-4.3-BT-SIMA2021}
Let $\ep,\de>0$ be given and assume $\de\le c_0$, where $c_0$ is given in Lemma \ref{lem-4.3-BT-SIMA2021}.
Assume that $v_0,b_0\in L^2$ are divergence free and that $u,a:\R^3\times[0,T_0]\to\R^3$ are divergence free that satisfy
\[
\esssup_{0<t\le T_0} \bke{\norm{u(t)}_{L^3_\uloc} + \norm{b(t)}_{L^3_\uloc}} < \de.
\]
Then there exist $T\in(0, T_0]$ and a weak solution $(v_\ep, b_\ep)$ and pressure $\pi_\ep$ to \eqref{eq-4.29-BT-SIMA2021} on $\R^3\times[0,T]$.
Furthermore, we have that $(v_\ep,b_\ep)\in L^\infty(0,T;E^2_q\times E^2_q)$ and satisfies
\EQN{
\norm{\esssup_{0\le t\le T_0} \int_{B_1(x_0)} \frac{|v_\ep|^2 + |b_\ep|^2}2\, dx + \int_0^{T_0} \int_{B_1(x_0)}\bke{|\nb v_\ep|^2 + |\nb b_\ep|^2} dxdt}_{\ell^{\frac{q}2}(x_0\in\ZZ^3)}^2
\le 2C\norm{(v_0,b_0)}_{E^2_q\times E^2_q}
}
for some positive constant $C$ independent of $\ep, \de, (u,a)$ and $(v_0,b_0)$.
Here, $T=\min\bke{T_0,\la_0,\la_0A_{0,q}^{-2}}$ depends on $\norm{(v_0,b_0)}_{E^2_q\times E^2_q}$ but not on $\norm{(v_0,b_0)}_{L^2\times L^2}$, $(v_\ep,b_\ep)$, $\ep,\de$, or $(u,a)$.
\end{lemma}

\begin{proof}
The proof of the lemma is an adaption of the proof of \cite[Lemma 4.3]{BT-SIMA2021} for the Navier--Stokes equations to the MHD equations.

Let $(v_\ep, b_\ep,\pi_\ep)$ be a smooth solution of \eqref{eq-4.29-BT-SIMA2021} on $\R^3\times[0,T_0]$ with initial data$(v_0,b_0)\in L^2$.
The energy equality for the regularized perturbed problem reads:
\EQN{
&\norm{v_\ep(t)}_{L^2}^2+ \norm{b_\ep(t)}_{L^2}^2 + 2\int_0^t\int_0 (|\nb v_\ep|^2 + |\nb b_\ep|^2)\, dxds\\
&= \norm{v_0}_{L^2}^2 + \norm{b_0}_{L^2}^2 + 2\int_0^t \int \bkt{(v_\ep\cdot\nb v_\ep-b_\ep\cdot\nb b_\ep)\cdot(\eta_\ep * u) + (v_\ep\cdot\nb b_\ep-b_\ep\cdot\nb v_\ep)\cdot(\eta_\ep * a)} dxds
}
By the estimate from \cite[p.~217]{MR1938147}, the right-hand side is uniformly bounded in $\ep$, implying that $\norm{(v_\ep,b_\ep)(t)}_{L^2\times L^2}\le M_1$ for all $t\in[0,T_0]$ for some constant $M_1$ independent of $\ep$.
Furthermore, if $(v_\ep,b_\ep)\in{\bf LE}_q(0,T_0)$, then by Lemma \ref{lem-4.2-BT-SIMA2021}, the local energy estimates extend up to time $T=\min\bke{T_0,\la_0,\la_0A_{0,q}^{-2}}$, yielding $\norm{(v_\ep,b_\ep)}_{{\bf LE}_q(0,T_0)} < M_2$ for some constant $M_2$ independent of $\ep$.

Now, let $T_{\ep,\de}$ be the time-scale provided in Lemma \ref{lem-4.1-BT-SIMA2021} corresponding to initial data of size $M_1$ in $L^2$.
Then, Lemma \ref{lem-4.1-BT-SIMA2021} ensures that the solution $(v_\ep, b_\ep)$ exists on $\R^3\times[0,T_{\ep,\de}]$ and belongs to ${\bf LE}_q(0,T_{\ep,\de})$.
Since Lemma \ref{lem-4.2-BT-SIMA2021} also applies to the regularized system, we further conclude that
$\norm{(v_\ep,b_\ep)}_{{\bf LE}_q(0,T_{\ep,\de})}<M_2$ and hence
$\esssup_{0<t\le T_{\ep,\de}} \norm{(v,b)(t)}_{E^2_q\times E^2_q}\le M_2$.
In addition, the energy estimate gives $\esssup_{0<t\le T_{\ep,\de}} \norm{(v,b)(t)}_{L^2\times L^2}\le M_1$.
This allows us to restart the solution at any time $t_*\in[T_{\ep,\de}/2, 3T_{\ep,\de}/4]$, and apply Lemma \ref{lem-4.2-BT-SIMA2021} again with the same bounds.
By uniqueness, the extended solution coincides with the original one, and hence we obtain a solution on $[0,3T_{\ep,\de}/2]$ that remains in ${\bf LE}_q(0,3T_{\ep,\de}/2)$.
Repeating this argument and iterating the solution step-by-step, we reach the full time interval $[0,T]$.
Throughout the iteration, the ${\bf LE}_q$ and $L^2$ norms remain bounded uniformly by $M_2$ and $M_1$, respectively.
Therefore, for each $\ep>0$, the solution $(v_\ep,b_\ep)$ to the regularized system exists on $[0,T]$ and satisfies $(v_\ep,b_\ep)\in{\bf LE}_q(0,T)$ with bounds independent of $\ep$.
\end{proof}

\begin{lemma}\label{lem-4.4-BT-SIMA2021}
Let $c_0$ and $\la_0$ be the constants in Lemma \ref{lem-4.2-BT-SIMA2021}.
Assume that $v_0,b_0\in E^2_q$ are divergence free and that $u,a:\R^3\times[0,T_0]\to\R^3$ are divergence free and satisfy
\[
\esssup_{0<t\le T_0} \bke{\norm{u(t)}_{L^3_\uloc} + \norm{a(t)}_{L^3_\uloc}} < \de\le c_0\ \text{ and }\ 
\esssup_{0<t\le T_0} \bke{\norm{u(t)}_{L^4_\uloc} + \norm{a(t)}_{L^4_\uloc}} < \infty.
\]
Let $T=\min\bke{T_0,\,\la_0,\,\la_0A_{0,q}^{-2}}$.
Then there exist a local energy solution $(v,b)$ and $\pi$ to the perturbed MHD equations, \eqref{eq-4.31-BT-SIMA2021}, satisfying 
\[
\norm{(v,b)}_{{\bf LE}_q(0,T)} \le C\norm{(v_0,b_0)}_{E^2_q\times E^2_q}
\]
for some constant $C>0$.
\end{lemma}

\begin{proof}
The proof of the lemma is an adaption of the proof of \cite[Lemma 4.4]{BT-SIMA2021} for the Navier--Stokes equations to the MHD equations.

Fix $(v_0,b_0)\in E^2_q\times E^2_q$.
For each $\ep>0$, approximate the data by divergence-free vector fields $(v_0^{(\ep)}, b_0^{(\ep)})\in L^2\times L^2$ satisfying $\norm{v_0-v_0^{(\ep)}}_{E^2_q} + \norm{b_0-b_0^{(\ep)}}_{E^2_q}<\ep$.
Such approximations can be constructed using the Bogovskii map (see \cite{Tsai-book}).
Let $(v_\ep,b_\ep)$ denote the solutions constructed in Lemma \ref{lem-4.4-BT-SIMA2021} corresponding to the initial data $(v_0^{(\ep)}, b_0^{(\ep)})$.
By the uniform estimates from Lemma \ref{lem-4.4-BT-SIMA2021}, we obtain bounds on $\pd_tv_\ep$ and $\pd_tb_\ep$ in the dual of $L^3(0,T; W_0^{1,3}(B_M(0)))$, which allow us to extract a subsequence $(v_n, b_n):=(v_{\ep_n}, b_{\ep_n})$ of $(v_\ep,b_\ep)$ and $\pi_n := \pi_{\ep_n}$ of $\pi_\ep$, such that, as $n\to\infty$,
\EQN{
(v_n, b_n) \overset{*}{\rightharpoonup} (v,b)&\ \text{ in }L^\infty(0,T; L^2_\loc\times L^2_\loc),\\
(v_n, b_n) \rightharpoonup (v,b)&\ \text{ in }L^2(0,T; H^1_\loc\times H^1_\loc),\\
(v_n, b_n), (\eta_{\ep_n}*v_n, \eta_{\ep_n}*b_n) \rightarrow (v,b)&\ \text{ in }L^3(0,T; L^3_\loc\times L^3_\loc),\\
(\eta_{\ep_n}*u, \eta_{\ep_n}*a) \rightarrow (u,a)&\ \text{ in }L^3(0,T; L^3_\loc\times L^3_\loc),\\
\pi_n^{(k)} \rightharpoonup \pi^{(k)}&\ \text{ in }L^{3/2}(0,T; L^{3/2}(B_k(0))),
}
where $\pi^{(k)}(x,t) = \pi(x,t) - c_k(t)$ for $x\in B_k(0)$ and $t\in(0,T_0]$ for some $c_k\in L^{3/2}(0,T_0)$, and $\pi_n^{(k)}$ is the local pressure expansion for $\pi_n$ on ball $B_k(0)$.
The limit $(v,b,\pi)$ is a local energy solution to the perturbed MHD equations with initial data $(v_0,b_0)$.
We claim that $(v,b,\pi)$ satisfies the \emph{perturbed} local energy inequality: for all nonnegative $\phi\in C^\infty_c(\R^3\times[0,T))$,
\EQS{\label{eq-4.33-BT-SIMA2021}
2&\iint \bke{|\nb v|^2 + |\nb b|^2}\phi\, dxdt\\
&\le \int \bke{|v_0|^2 + |b_0|^2}\phi\, dx + \iint \bke{|v|^2 + |b|^2}(\pd_t\phi + \De\phi)\, dxdt + \iint \bke{|v|^2 + |b|^2 + 2\pi}(v\cdot\nb\phi)\, dxdt\\
&\quad + \iint \bke{|v|^2 + |b|^2}(u\cdot\nb\phi)\, dxdt + 2\iint (v\cdot\nb v - b\cdot\nb b)\cdot(u\phi)\, dxdt \\
&\quad + 2\iint (v\cdot u + b\cdot a)(v\cdot\nb\phi)\, dxdt + 2\iint (v\cdot\nb b -b\cdot\nb v)\cdot(a\phi)\, dxdt  \\
&\quad - 2\iint (v\cdot a + b\cdot u)(b\cdot\nb\phi)\, dxdt - 2\iint (v\cdot b)((b+a)\cdot\nb\phi)\, dxdt.
}
The first two lines are inherited via standard compactness arguments.
We now focus on the convergence of the remaining terms, especially those not involving $\nb\phi$, which are of higher order.
We have
\EQS{
&\abs{\iint (v_n\cdot\nb v_n)\cdot(\eta_{\ep_n} * u)\phi - (v\cdot\nb v)\cdot u\phi\, dxdt}\\
&\qquad\le \abs{\iint ((v_n-v)\cdot\nb v_n)\cdot(u\phi)\, dxdt} + \abs{\iint (v\cdot\nb (v_n-v))\cdot(u\phi)\, dxdt}\\
&\qquad\quad + \abs{\iint (v_n\cdot\nb v_n)\cdot(\eta_{\ep_n} * u - u)\phi\, dxdt}\\
&\qquad=: I_{1,n} + I_{2,n} + I_{3,n},
}
\EQS{
&\abs{\iint (b_n\cdot\nb b_n)\cdot(\eta_{\ep_n} * u)\phi - (b\cdot\nb b)\cdot u\phi\, dxdt}\\
&\qquad\le \abs{\iint ((b_n-b)\cdot\nb b_n)\cdot(u\phi)\, dxdt} + \abs{\iint (b\cdot\nb (b_n-b))\cdot(u\phi)\, dxdt}\\
&\qquad\quad + \abs{\iint (b_n\cdot\nb b_n)\cdot(\eta_{\ep_n} * u - u)\phi\, dxdt}\\
&\qquad=: I_{4,n} + I_{5,n} + I_{6,n},
}
\EQS{
&\abs{\iint (v_n\cdot\nb b_n)\cdot(\eta_{\ep_n} * a)\phi - (v\cdot\nb b)\cdot a\phi\, dxdt}\\
&\qquad\le \abs{\iint ((v_n-v)\cdot\nb b_n)\cdot(a\phi)\, dxdt} + \abs{\iint (v\cdot\nb (b_n-b))\cdot(a\phi)\, dxdt}\\
&\qquad\quad + \abs{\iint (v_n\cdot\nb b_n)\cdot(\eta_{\ep_n} * a - a)\phi\, dxdt}\\
&\qquad=: I_{7,n} + I_{8,n} + I_{9,n},
}
and
\EQS{
&\abs{\iint (b_n\cdot\nb v_n)\cdot(\eta_{\ep_n} * a)\phi - (b\cdot\nb v)\cdot a\phi\, dxdt}\\
&\qquad\le \abs{\iint ((b_n-b)\cdot\nb v_n)\cdot(a\phi)\, dxdt} + \abs{\iint (b\cdot\nb (v_n-v))\cdot(a\phi)\, dxdt}\\
&\qquad\quad + \abs{\iint (b_n\cdot\nb v_n)\cdot(\eta_{\ep_n} * a - a)\phi\, dxdt}\\
&\qquad=: I_{10,n} + I_{11,n} + I_{12,n}.
}
Our goals is to show that the above twelve quantities vanishes as $n\to\infty$.
Let $B$ be a ball containing $\supp\,\phi$.
Then, using H\"older's inequality and log-convexity of $L^p$ norms, we have
\[
I_{1,n}
\lec \norm{u}_{L^\infty L^4_\uloc} \norm{v_n-v}_{L^2(0,T;L^2(B))}^{1/4} \norm{v_n}_{L^2(0,T;H^1(B))}^{7/4}\to0,
\]
\[
I_{4,n}
\lec \norm{u}_{L^\infty L^4_\uloc} \norm{b_n-b}_{L^2(0,T;L^2(B))}^{1/4} \norm{b_n}_{L^2(0,T;H^1(B))}^{7/4}\to0,
\]
\[
I_{7,n}
\lec \norm{a}_{L^\infty L^4_\uloc} \norm{v_n-v}_{L^2(0,T;L^2(B))}^{1/4} \norm{b_n}_{L^2(0,T;H^1(B))}^{7/4}\to0,
\]
and
\[
I_{10,n}
\lec \norm{a}_{L^\infty L^4_\uloc} \norm{b_n-b}_{L^2(0,T;L^2(B))}^{1/4} \norm{v_n}_{L^2(0,T;H^1(B))}^{7/4}\to0,
\]
as $n\to\infty$ by strong convergence of $(v_n,b_n)$ to $(v,b)$ in $L^2(0,T;L^2(B)\times L^2(B))$.
Next, weak convergence of $(v_n,b_n)$ to $(v,b)$ in $L^2(0,T;H^1(B)\times H^1(B))$ ensures that $I_{2,n}$, $I_{5,n}$, $I_{8,n}$, $I_{11,n}\to0$ as $n\to\infty$, since the products $v_iu_j\phi, b_iu_j\phi, v_ia_j\phi$, and $b_ia_j\phi$ all belong to $L^2(B\times(0,T))$. 
Finally, for the mollifier terms, we use strong convergence of the mollified quantities in $L^\infty(0,T;L^3(B))$ and uniform bounds on $v_n$, $b_n$ in $L^2(0,T;H^1(B))$, to deduce $I_{3,n}$, $I_{6,n}$, $I_{9,n}$, and $I_{12,n}\to0$ as $n\to\infty$.
Hence, all error terms vanish in the limit, and \eqref{eq-4.33-BT-SIMA2021} holds. 
Moreover, following the argument in \cite[(3.28)-(3.29)]{KT-CMP2020}, we derive the time-slice version of the perturbed local energy inequality: for any nonnegative $\psi\in C^\infty_c(\R^3)$ and any $t\in(0,T)$,
\EQS{\label{eq-4.35-BT-SIMA2021}
\int& \bke{|v(t)|^2 + |b(t)|^2}\psi\, dx + 2\int_0^t\int \bke{|\nb v|^2 + |\nb b|^2}\psi\, dxdt\\
&\le \int \bke{|v_0|^2 + |b_0|^2}\psi\, dx + \int_0^t\int \bke{|v|^2 + |b|^2} \De\psi\, dxdt + \iint \bke{|v|^2 + |b|^2 + 2\pi}(v\cdot\nb\psi)\, dxdt\\
&\quad + \int_0^t\int \bke{|v|^2 + |b|^2}(u\cdot\nb\psi)\, dxdt + 2\int_0^t\int (v\cdot\nb v - b\cdot\nb b)\cdot(u\psi)\, dxdt \\
&\quad + 2\iint (v\cdot u + b\cdot a)(v\cdot\nb\psi)\, dxdt + 2\int_0^t\int (v\cdot\nb b -b\cdot\nb v)\cdot(a\psi)\, dxdt  \\
&\quad - 2\int_0^t\int (v\cdot a + b\cdot u)(b\cdot\nb\psi)\, dxdt - 2\int_0^t\int (v\cdot b)((b+a)\cdot\nb\psi)\, dxdt.
}

We now establish the bound for $\norm{(v,b)}_{{\bf LE}_q(0,T)}$.
Let $\phi$ the cutoff function used in the proof of Lemma \ref{lem3.1-BT-SIMA2021}, and define $\psi(x) := \phi(x-k)$ for each $k\in\ZZ^3$.
Fix a large integer $K>0$, and restrict to $|k|\le K$.
Applying \eqref{eq-4.35-BT-SIMA2021} with this choice of $\psi$, the right-hand side can be approximated by the corresponding terms for the sequence $\{(v_n,b_n)\}$, since all such quantities converge in the limit.
In particular, for all $|k|\le K$, we can ensure that the difference between the terms involving $(v,b)$ and those for $(v_n,b_n)$ is less than $2^{-K}K^{-3}$ uniformly, provided $n\ge N_K$ for some sufficiently large $N_K$.
Taking these approximate terms, applying standard estimates (which can be derived similar to the proof of \eqref{eq-3.6-BT-SIMA2021})
\EQS{\label{eq-3.7-BT-SIMA2021}
\norm{\int_0^{\la_R R^2} \int_{B_R(x_0R)} |v|^3 + |b|^3 + \abs{\pi - c_{R x_0,R}(t)}^{3/2} dxdt}_{\ell^{\frac{q}3}(x_0\in\ZZ^3)} \le C \la_R^{\frac1{10}} R^{\frac12} \bke{A_{0,q}(R)}^{\frac32},\quad R>0.
}
Taking the essential supremum in time followed by the $\ell^{q/2}$-sum over $|k|\le K$, we obtain a uniform bound 
\EQN{
&\bkt{\sum_{|k|\le K} \bke{\esssup_{0<t<T} \int_{B_1(k)} \bke{|v(x,t)|^2 + |b(x,t)|^2} dx + \int_0^T\int_{B_1(k)} \bke{|\nb v|^2 + |\nb b|^2} dxdt}^{q/2} }^{1/q}\\
&\qquad\qquad\qquad\qquad\qquad\qquad\qquad\qquad\qquad\qquad\qquad\qquad\qquad\le C\bke{\norm{(v_0,b_0)}_{E^2_q\times E^2_q} + \frac{C}{2^K}}.
}
Since this bound holds for all $K\in\NN$, it follows that $(v,b)\in {\bf LE}_q(0,T)$, with the norm estimate $\norm{(v,b)}_{{\bf LE}_q(0,T)}\le C\norm{(v_0,b_0)}_{E^2_q\times E^2_q}$.
\end{proof}

\begin{lemma}\label{lem-5.1-BT-SIMA2021}
Let $2\le q<\infty$.
Assume $v_0,b_0\in L^2$ are divergence free, and assume $u,a:\R^3\times[0,T_0]\to\R^3$ are divergence free and satisfy
\[
\esssup_{0<t\le T_0} \bke{\norm{u(t)}_{L^3_\uloc} + \norm{a(t)}_{L^3_\uloc}} < \de\le c_0\ \text{ and }\ 
\esssup_{0<t\le T_0} \bke{\norm{u(t)}_{L^4_\uloc} + \norm{a(t)}_{L^4_\uloc}} < \infty.
\]
where $c_0$ is from Lemma \ref{lem-4.2-BT-SIMA2021}.
For any $T\in(0,T_0]$, if $\de\le\de_0(T)\le c_0$ is sufficiently small, then there exists a local energy solution $(v,b)$ to the perturbed MHD equations, \eqref{eq-4.31-BT-SIMA2021}, so that $(v,b)\in{\bf LE}_q(0,T)$.
In particular, this is true when $(u,a) \equiv (0,0)$.
\end{lemma}

\begin{proof}
The proof of the lemma is an adaption of the proof of \cite[Lemma 5.1]{BT-SIMA2021} for the Navier--Stokes equations to the MHD equations.

We begin with the special case $(u,a) = (0,0)$ to highlight the key ideas.
Suppose the initial data $(v_0, b_0)\in L^2\times L^2$, and define $a_k = \int_{B_1(k)} \bke{|v_0|^2 + |b_0|^2}dx$ for $k\in\ZZ^3$.
Then, 
\[
\sum_ka_k^{q/2}\le(\max_k a_k)^{q/2-1}\sum_ka_k\le\bke{\sum_ka_k}^{q/2},
\]
which shows that $(v_0,b_0)\in E^2_q\times E^2_q$ with $\norm{(v_0,b_0)}_{E^2_q\times E^2_q}\le M_2:= C\norm{(v_0,b_0)}_{L^2\times L^2}$.
Let $(v,b)$ be the solution of the perturbed MHD equations with $(u,a) = (0,0)$ (so that it is a solution of \eqref{MHD}) constructed via Lemma \ref{lem-4.4-BT-SIMA2021} with initial data $(v_0,b_0)$.
Then $(v,b)\in {\bf LE}_q(0,T_0)$, where $T_0=T_0(M_2)$ is the existence time depending on the size of the initial data in $E^2_q\times E^2_q$.
For almost every $t\in(0,T)$, we have that $(v,b)(t)\in E^3\times E^3$ and that $\norm{(v,b)(t)}_{L^2\times L^2}\le \norm{(v_0,b_0)}_{L^2\times L^2}$.
The inclusion $(v,b)(t)\in E^3$ follows from Lemma \ref{lem-solution-in-E4q} and the embedding $E^4_q\subset E^3$. 
Hence, $\norm{(v,b)(t)}_{E^2_q\times E^2_q}\le M_2$ for almost every $t\in(0,T_0)$.
In particular, these bounds hold at some time $t_0\in(T_0/2,T_0)$.
We now restart the MHD equations at time $t_0$, treating $(v,b)(t_0)$ as new initial data in $E^3\times E^3$. By Lemma \ref{lem-4.4-BT-SIMA2021}, there exists a local energy solution $(v_1,b_1)$ in ${\bf LE}_q(t_0,t_0+T_0)$.
By uniqueness of local energy solution with $E^3\times E^3$ (Corollary \ref{cor1.8-BT-CPDE2020}), we have $(v_1,b_1) = (v,b)$ on some short interval $[t_0,t_0+\De t_1]$.
This allows us to glue $(v_1,b_1)$ to $(v,b)$, yielding a local energy solution (still denoted $(v,b)$) on $[0,3T_0/2]$ that lies in ${\bf LE}_q(t_0,t_0+T_0)\cap{\bf LE}_q(0,T_0)$.
Hence, $(v,b)\in {\bf LE}_q(0,3T_0/2)$.
Repeating this argument, we obtain a solution $(v,b)\in{\bf LE}_q(0,T)$ for any $T>0$, using the uniform-in-time control of $\norm{(v,b)}_{E^2_q\times E^2_q}$.

Now consider the general case $(u,a)\neq(0,0)$.
Let $(v_0,b_0)\in L^2\times L^2\subset E^2_q\times E^2_q$, and let $(v,b)$ be the local energy solution to the perturbed MHD equations, \eqref{eq-4.31-BT-SIMA2021}, given by Lemma \ref{lem-4.4-BT-SIMA2021}.
Assuming $\de:= \norm{(u,a)}_{L^\infty(L^3_\uloc\times L^3\uloc)}\ll c_0$, we have $(v,b)\in{\bf LE}_q(0,T_0)$ for some $T_0 = T_0(\norm{(v_0,b_0)}_{E^2_q\times E^2_q})$.
We now derive an energy estimate.
Using the following bounds:
\[
\int_{\R^3} |u||v|(|\nb v|+|v|)
\le \sum_k \int_{B_1(k)} |u||v|(|\nb v|+|v|)
\lec \norm{u}_{L^3_\uloc} \sum_k \int_{B_1(k)} (|\nb v|^2 + |v|^2),
\]
\[
\int_{\R^3} |u||b|(|\nb b|+|b|)
\le \sum_k \int_{B_1(k)} |u||b|(|\nb b|+|b|)
\lec \norm{u}_{L^3_\uloc} \sum_k \int_{B_1(k)} (|\nb b|^2 + |b|^2),
\]
\[
\int_{\R^3} |a||v|(|\nb b|+|b|)
\le \sum_k \int_{B_1(k)} |a||v|(|\nb b|+|b|)
\lec \norm{a}_{L^3_\uloc} \sum_k \int_{B_1(k)} (|\nb b|^2 + |b|^2 + |v|^2),
\]
and
\[
\int_{\R^3} |a||b|(|\nb v|+|v|)
\le \sum_k \int_{B_1(k)} |a||b|(|\nb v|+|v|)
\lec \norm{a}_{L^3_\uloc} \sum_k \int_{B_1(k)} (|\nb v|^2 + |v|^2 + |b|^2),
\]
we obtain the energy inequality:
\EQS{
&\norm{(v,b)(t)}_{L^2\times L^2}^2 + 2\norm{(\nb v,\nb b)}_{L^2(0,t;L^2\times L^2)}^2\\
&\qquad\le \norm{(v_0,b_0)}_{L^2\times L^2}^2\\
&\qquad\quad + C\norm{(u,a)}_{L^\infty(L^3_\uloc\times L^3_\uloc)} \bke{\norm{(\nb v,\nb b)}_{L^2(0,t;L^2\times L^2)}^2 + t\sup_{0<s<t} \norm{(v,b)(s)}_{L^2\times L^2}^2 },
}
for $0<t<T$.
If $T\le T_0$, the result follows. Otherwise, we choose $\de$ sufficiently small (depending on $T$) to absorb the right-hand side, yielding:
\[
\sup_{0<t<T} \norm{(v,b)(t)}_{L^2\times L^2}^2 + 2\norm{(\nb v,\nb b)}_{L^2(0,T; L^2\times L^2)}^2 \le 2\norm{(v_0,b_0)}_{L^2\times L^2}^2.
\]
This in turn implies
\[
\sup_{0<t<T} \norm{(v,b)(t)}_{E^2_q\times E^2_q}
\le C\norm{(v_0,b_0)}_{L^2\times L^2}.
\]
Thus, we obtain uniform-in-time control of $\norm{(v,b)(t)}_{E^2_q\times E^2_q}$ and the argument from the $(u,a) = (0,0)$ case applies to yield the desired result.
\end{proof}

\begin{lemma}\label{lem-A.1-BT-SIMA2021}
Suppose that $v_0,b_0\in E^4_\infty$ are divergence free.
Assume also that 
$\de:=\norm{(v_0,b_0)}_{E^4_\infty\times E^4_\infty}<\ep_*$ for a universal constant $\ep_*$.
Then there exists a second universal constant $\tau_0>0$ and $(v,b)$ and $\pi$ comprising a local energy solution to \eqref{MHD} in $\R^3\times(0,\tau_0)$ with initial data $(v_0,b_0)$ so that $(v,b)$ and $\pi$ are smooth in space and time, $(v,b)\in C([0,\tau_0]; E^4_\infty\times E^4_\infty)$, and 
\[
\sup_{0\le t\le \tau_0} \norm{(v,b)(t)}_{L^4_\uloc\times L^4_\uloc} < C\de.
\]
Furthermore, if $(u,a)\in\mathcal{N}_{\rm MHD}(v_0,b_0)$, then $(u,a) = (v,b)$ on $\R^3\times[0,\tau_0]$.
\end{lemma}

\begin{proof}
The proof of the lemma is an adaption of the proof of \cite[Lemma A.1]{BT-SIMA2021} for the Navier--Stokes equations to the MHD equations.

Since $(v_0,b_0)\in E^4_\infty\times E^4_\infty$, it follows that $(v_0,b_0)\in E^2\times E^2$ as $E^4_\infty\subset E^2$.
Viewing the MHD equations as a coupled system of inhomogeneous Stokes systems, we apply the linear theory from \cite[\S5]{KS-book2007} and follow the argument of Theorem 1.5 therein to construct a global-in-time local energy solution $(v,b)$ evolving from $(v_0,b_0)$.
We may assume $\norm{(v_0,b_0)}_{L^3_\uloc\times L^3_\uloc} < \ep_3$, where $\ep_3$ is given in Theorem \ref{thm-1.7-BT-CPDE2020}.
Then, for all $x_0\in\R^3$ and $r\le 1$,
\[
\frac1r\int_{B_r(x_0)} (|v_0|^2 + |b_0|^2)\, dx 
\le C\bke{\int_{B_r(x_0)} (|v_0|^3 + |b_0|^3)\, dx }
\le C\ep_3.
\]
Thus, Theorem \ref{thm-1.7-BT-CPDE2020} ensures uniqueness of the local energy solution $(v,b)\in\mathcal{N}_{\rm MHD}(v_0,b_0)$ up to some time $\tau_0$.

Now consider the mild solution constructed in Theorem \ref{thrm.subcritical-mhd} with $r=4$, $q=\infty$. Since $E^4\subset\mathcal{L}^4_\uloc$, the closure of $BUC(\R^3)$ in the $L^4_\uloc$ norm, there exists a time $T>0$ and a unique mild solution $(u,a)\in C([0,T); L^4_\uloc)$.
In the construction of this mild solution (see Section \ref{sec-subcritical}), we may redefine the space $\mathcal{E}_T$ with the norm $\norm{(v,b)}_{\mathcal{E}_T} = \sup_{0<t<T} \norm{(v,b)(t)}_{E^r_q} + \sup_{0<t<T} t^{3/(2r)}\norm{(v,b)(t)}_{L^\infty}$ to ensure that the mild solution $(u(t),a(t))\in L^\infty\times L^\infty$ for all $t>0$.
Then, by adapting the regularity argument from \cite[\S4]{GIM-QuadMat1999}, we conclude that $(u,a)$ is smooth in both space and time for all $t>0$.
By choosing $\ep_*$ sufficiently small, the existence time $T$ for the mild solution in Theorem \ref{thrm.subcritical-mhd} exceeds $\tau_0$.

We now verify that $(u,a)$ defines a local energy solution.
By embeddings, the same convergence properties at $t=0$ hold with $L^4$ and $L^4_\uloc$ replaced by $L^2$ and $L^2_\uloc$, respectively.
This implies that if $w\in L^2(\R^3)$ is compactly supported, then
\[
\lim_{t\to0}\int (u(x,t)-u_0(x))w(x)\, dx = \lim_{t\to0}\int (a(x,t)-a_0(x))w(x)\, dx = 0.
\]
Moreover, the maps 
\[
t\mapsto \int u(x,t)\cdot w(x)\, dx\quad \text{ and }\quad
t\mapsto \int a(x,t)\cdot w(x)\, dx
\]
are continuous for $t>0$ due to the smoothness of $(u,a)$.

To complete the verification, we adapt the pressure construction from \cite[Theorem 1.4]{BT-JMFM2022}, as carried out in \cite[\S6]{BT-JMFM2022} for the Navier--Stokes equations.
This yields a pressure $\pi$ such that $(u,a,\pi)$ satisfies the MHD equations in the distributional sense.
The local expansion of $\pi$ also guarantees that $\pi\in L^{3/2}_\loc(\R^3\times(0,T))$.
The local energy inequality follows from the space-time smoothness of $(u,a)$, and item 2 in the definition of local energy solution is satisfied since $(u,a)\in L^\infty L^3_\uloc$.
Hence, $(u,a)\in\mathcal{N}_{\rm MHD}(v_0,b_0)$, and uniqueness implies $(v,b) = (u,a)$ on $\R^3\times(0,\tau_0)$.
Therefore,
\[
\norm{(v,b)(t)}_{L^3_\uloc\times L^3_\uloc} \le C\norm{(v,b)(t)}_{L^4_\uloc\times L^4_\uloc}<C\de
\]
for all $t\in(0,\tau_0)$.

The mild solution constructed in Theorem \ref{thrm.subcritical-mhd} satisfies $(v,b)\in C([0,\tau_0];L^4_\uloc\times L^4_\uloc)$.
Since $L^4_\uloc = E^4_\infty\subset E^2_\infty$, we may apply Lemma \ref{lem-solution-in-E4q} with $q=\infty$ and data $(v_0,b_0)\in E^4_\infty\times E^4_\infty\subset E^2_\infty\times E^2_\infty$ to conclude that for almost every $t>0$, $(v,b)(t)\in E^4_\infty\times E^4_\infty$.
This further implies $(v,b)\in C([0,\tau_0];E^4_\infty\times E^4_\infty)$.
\end{proof}

\begin{proof}[Proof of Theorem \ref{th4.8-mhd} for $q\ge2$]
The proof of Theorem \ref{th4.8-mhd} for $q\ge2$ is an adaption of the proof of \cite[Theorem 1.5]{BT-SIMA2021} for the Navier--Stokes equations to the MHD equations.

Assume that $v_0,b_0\in E^2_q$ are divergence free.
By Lemma \ref{lem-4.4-BT-SIMA2021} with $(u,a) = (0,0)$, there exists a local energy solution $(v,b)$ to the MHD equations on $\R^3\times(0,T_0)$ such that $(v,b)\in{\bf LE}_q(0,T_0)$.
Moreover, by Lemma \ref{lem-solution-in-E4q}, we have 
\[
(v,b)(t)\in E^4_q\times E^4_q\ \text{ for a.e. }t\in[T_0/2, T_0].
\]

Choose a time $t_0>T_0/2$ so that $(v,b)(t_0)\in E^4_q\times E^4_q$.
We aim to construct a local energy solution in ${\bf LE}_q(t_0,t_0+\tau_0)$ with initial data $(v,b)(t_0)\in E^4_q\times E^4_q$, where $\tau_0$ is the fixed time-scale in Lemma \ref{lem-A.1-BT-SIMA2021}.
Using the Bogovskii map (see \cite{Tsai-book} for the details), for any $\de>0$, we can decompose the initial data as
\[
(v,b)(t_0) = (u_0,a_0) + (w_0,d_0),\qquad 
\nb\cdot w_0 = \nb\cdot d_0=0,
\]
where
\[
\norm{(u_0,a_0)}_{E^4_q\times E^4_q} < \de,\qquad
w_0, d_0\in L^2(\R^3).
\]

By Lemma \ref{lem-A.1-BT-SIMA2021}, choosing $\de$ sufficiently small ensures the existence of a local energy solution $(u,a)$ with pressure $\pi$, defined on $\R^3\times(t_0,t_0+\tau_0)$, evolving from the initial data $(u_0,a_0)$, which is smooth in both space and time and satisfies
\[
\sup_{t_0\le t\le t_0+\tau_0}  \norm{(u,a)(t)}_{L^4_\uloc\times L^4_\uloc} 
< C\de.
\]
Furthermore, by the uniqueness result in Lemma \ref{lem-A.1-BT-SIMA2021}, this solution $(u,a)$ coincides with the one given by Lemma \ref{lem-4.4-BT-SIMA2021}, and hence $(u,a)\in{\bf LE}_q(t_0,t_0+\tau_0)$.

Next, we apply Lemma \ref{lem-5.1-BT-SIMA2021} with the perturbation factor $(u,a)$, again choosing $\de$ sufficiently small to ensure that the time-scale it yields is at least $\tau_0$.
This gives a local energy solution $(w,d)\in{\bf LE}_q(t_0,t_0+\tau_0)$ to the the perturbed MHD equation with initial data $(w_0,d_0)$ and associated pressure $\pi$.
Define $(v_1,b_1) := (u,a) + (w,d)$.
This gives a local energy solution on $\R^3\times(t_0,t_0+\tau_0)$.
To very the local energy inequality for $(v_1,b_1)$, we use approximations $(w^{(n)},d^{(n)})\to(w,d)$, as in the proof of Lemma \ref{lem-4.4-BT-SIMA2021}, and apply the inequality to $(u,a) + (w^{(n)},d^{(n)})$.

Since $(v,b)$ and $(v_1,b_1)$ coincide at $t_0$, and since $(v,b)(t_0)\in E^3\times E^3$ since $E^4_\infty\subset E^3$, Corollary \ref{cor1.8-BT-CPDE2020} implies that $(v,b)(x,t) = (v_1,b_1)(x,t)$ on $\R^3\times(t_0,t_0+\ga)$ for some $\ga>0$.
Thus, we may glue $(v_1,b_1)$ to $(v,b)$ to extend the solution to ${\bf LE}_q(0,t_0+\tau_0)$.
Repeating this procedure $n$ times yields a solution $(v,b)\in{\bf LE}_q(0,t_0+n\tau_0)$.
Taking the limit $n\to\infty$, we obtain the global-in-time local energy solution asserted in Theorem \ref{th4.8-mhd} for $q\ge2$.
\end{proof}

\subsubsection{The case $1\le q<2$}\label{sec-1<=q<2}

Now, we consider the case $1\le q<2$ and look at the localized and regularized MHD equations, \eqref{eq-4.10-BLT-MathAnn2024}.
The following lemma corresponds to \cite[Lemma 4.6]{BLT-MathAnn2024} for the Navier--Stokes equations.

\begin{lemma}\label{lem-4.6-BLT-MathAnn2024}
Let $q\ge1$.
For each $0<\ep<1$ and divergence free $v_0, b_0$ with $\norm{v_0}_{E^2_q}\le B$, $\norm{b_0}_{E^2_q}\le B$, if $0<T<\min\bke{1,\,c\ep^3 B^{-2}}$, we can find a unique solution $(v,b) = (v^\ep,b^\ep)$ to the integral form of \eqref{eq-4.10-BLT-MathAnn2024}
\EQS{\label{eq-4.11-BLT-MathAnn2024}
v(t) &= e^{t\De}v_0 - \int_0^t e^{(t-\tau)\De}\mathbb{P}\nb\cdot\bke{\mathcal{J}_\ep(v)\otimes v\Phi_\ep - \mathcal{J}_\ep(b)\otimes b\Phi_\ep}(\tau)\, d\tau,\\
b(t) &= e^{t\De}b_0 - \int_0^t e^{(t-\tau)\De}\mathbb{P}\nb\cdot\bke{\mathcal{J}_\ep(v)\otimes b\Phi_\ep - \mathcal{J}_\ep(b)\otimes v\Phi_\ep}(\tau)\, d\tau,
}
satisfying
\[
\norm{(v,b)}_{{\bf LE}_q(0,T)} \le 2C_0B,
\]
where $c>0$ and $C_0>1$ are absolute constants.
\end{lemma}

\begin{proof}
The proof of the lemma is an adaption of the proof of \cite[Lemma 4.6]{BLT-MathAnn2024} for the Navier--Stokes equations to the MHD equations.

Let $\Psi(v,b)$ denote the mapping defined by the right-hand side of \eqref{eq-4.11-BLT-MathAnn2024} for $(v,b)\in{\bf LE}_q(0,T)$.
By \cite[Lemma 2.9]{BLT-MathAnn2024} and the assumption $T\le1$, we obtain the estimate
\EQN{
\norm{\Psi(v,b)}_{{\bf LE}_q(0,T)} 
&\lec \norm{v_0}_{E^2_q}  + \norm{\mathcal{J}_\ep(v)\otimes v\Phi_\ep}_{E_{T,q}^{2,2}} + \norm{\mathcal{J}_\ep(b)\otimes b\Phi_\ep}_{E_{T,q}^{2,2}} \\
&\quad + \norm{b_0}_{E^2_q} + \norm{\mathcal{J}_\ep(v)\otimes b\Phi_\ep}_{E_{T,q}^{2,2}} + \norm{\mathcal{J}_\ep(b)\otimes v\Phi_\ep}_{E_{T,q}^{2,2}}\\
&\lec \norm{v_0}_{E^2_q}  + \norm{\mathcal{J}_\ep(v)}_{L^\infty(0,T;L^\infty(\R^3))} \norm{v}_{E_{T,q}^{2,2}} + \norm{\mathcal{J}_\ep(b)}_{L^\infty(0,T;L^\infty(\R^3))} \norm{b}_{E_{T,q}^{2,2}}\\
&\quad + \norm{b_0}_{E^2_q}  + \norm{\mathcal{J}_\ep(v)}_{L^\infty(0,T;L^\infty(\R^3))} \norm{b}_{E_{T,q}^{2,2}} + \norm{\mathcal{J}_\ep(b)}_{L^\infty(0,T;L^\infty(\R^3))} \norm{v}_{E_{T,q}^{2,2}}\\
&\lec \norm{v_0}_{E^2_q} + \norm{b_0}_{E^2_q} + \ep^{-\frac32} \sqrt{T} \bke{\norm{v}_{E_{T,q}^{\infty,2}}^2 + \norm{b}_{E_{T,q}^{\infty,2}}^2 }.
}
Therefore, for some constants $C_0, C_1>0$,
\[
\norm{\Psi(v,b)}_{{\bf LE}_q(0,T)} 
\le C_0\norm{(v_0,b_0)}_{E^2_q\times E^2_q} + C_1\ep^{-\frac32} \sqrt{T} \norm{(v,b)}_{{\bf LE}_q(0,T)}^2,
\]
To estimate the difference, let $(v,b), (u,a)\in {\bf LE}_q(0,T)$. Then
\EQN{
&\norm{\Psi(v,b) - \Psi(u,a)}_{{\bf LE}_q(0,T)}\\
&\qquad \le C_1\ep^{-\frac32}\sqrt{T}\bke{\norm{(v,b)}_{{\bf LE}_q(0,T)} + \norm{(u,a)}_{{\bf LE}_q(0,T)}} \norm{(v,b) - (u,a)}_{{\bf LE}_q(0,T)}.
}
Applying the Picard contraction principle, we see that if the time $T$ satisfies $T<\frac{\ep^3}{64(C_0C_1B)^2} = c\ep^3B^{-2}$, then $\Psi$ has a unique fixed point $(v,b)\in {\bf LE}_q(0,T)$ with $\norm{(v,b)}_{{\bf LE}_q(0,T)}\le 2C_0B$, solving the integral system \eqref{eq-4.11-BLT-MathAnn2024}.
\end{proof}

\begin{lemma}\label{lem-4.7-BLT-MathAnn2024}
Let $v_0,b_0\in E^2_q$, $q\ge1$, be divergence free.
For each $\ep\in(0,1)$, we can find $(v^\ep,b^\ep)$ in ${\bf LE}_q(0,T)$ and $\pi^\ep$ in $L^\infty(0,T;L^2(\R^3))$ for some positive $T=T(\ep,\norm{(v_0,b_0)}_{E^2_q\times E^2_q})$ which solve the localized and regularized MHD equations, \eqref{eq-4.10-BLT-MathAnn2024}, in the sense of distributions and $(v^\ep,b^\ep)(t) \to (v_0,b_0)$ in $L^2(E)\times L^2(E)$ as $t\to0^+$ for any compact subset $E$ of $\R^3$.
\end{lemma}

\begin{proof}
The proof of the lemma is an adaption of the proof of \cite[Lemma 4.7]{BLT-MathAnn2024} for the Navier--Stokes equations to the MHD equations.
We provide the corresponding details by following the same logic used in the proof of \cite[Lemma 3.4]{KT-CMP2020} for the Navier--Stokes equations, adapting it from the $L^2_\uloc$ framework to the $E^2_q$ setting, and from the Navier--Stokes equations to the MHD equations.

By Lemma \ref{lem-4.6-BLT-MathAnn2024}, there is a mild solution $(v^\ep,b^\ep)\in{\bf LE}_q(0,T)$ to \eqref{eq-4.11-BLT-MathAnn2024} for $T=T(\ep,\norm{(v_0,b_0)}_{E^2_q\times E^2_q})$.
Apparently,
\EQN{
\norm{v^\ep-e^{t\De}v_0}_{E^{\infty,2}_{t',q}}
 &= \norm{\int_0^t e^{(t-\tau)\De}\mathbb{P}\nb\cdot\bke{\mathcal{J}_\ep(v)\otimes v\Phi_\ep - \mathcal{J}_\ep(b)\otimes b\Phi_\ep}(\tau)\, d\tau}_{E^{\infty,2}_{t',q}}\\
 &\lec \norm{\mathcal{J}_\ep(v)\otimes v\Phi_\ep }_{E^{2,2}_{t',q}} + \norm{\mathcal{J}_\ep(b)\otimes b\Phi_\ep }_{E^{2,2}_{t',q}}\\
 &\lec \ep^{-3/2}\sqrt{t'}\bke{\norm{v}_{E^{\infty,2}_{t',q}}^2 + \norm{b}_{E^{\infty,2}_{t',q}}^2},
}
and 
\EQN{
\norm{b^\ep-e^{t\De}b_0}_{E^{\infty,2}_{t',q}}
 &= \norm{\int_0^t e^{(t-\tau)\De}\mathbb{P}\nb\cdot\bke{\mathcal{J}_\ep(v)\otimes b\Phi_\ep - \mathcal{J}_\ep(b)\otimes v\Phi_\ep}(\tau)\, d\tau}_{E^{\infty,2}_{t',q}}\\
 &\lec \norm{\mathcal{J}_\ep(v)\otimes b\Phi_\ep }_{E^{2,2}_{t',q}} + \norm{\mathcal{J}_\ep(b)\otimes v\Phi_\ep }_{E^{2,2}_{t',q}}\\
 &\lec \ep^{-3/2}\sqrt{t'}\norm{vb}_{E^{\infty,2}_{t',q}}
 \lec \ep^{-3/2}\sqrt{t'}\bke{\norm{v}_{E^{\infty,2}_{t',q}}^2 + \norm{b}_{E^{\infty,2}_{t',q}}^2},
}
where we have used \cite[Lemma 2.9]{BLT-MathAnn2024} and assumed $t'\le T\le 1$.
Also, for any compact subset $E$ of $\R^3$, we have $\norm{e^{t\De}v_0 - v_0}_{L^2(E)}$, $\norm{e^{t\De}b_0 - b_0}_{L^2(E)}\to0$ as $t$ goes to $0$ by Legesgue's convergence theorem.
Then, it follows that $\lim_{t\to0^+}\norm{v^\ep(t)-v_0}_{L^2(E)}=0$ and $\lim_{t\to0^+}\norm{b^\ep(t)-b_0}_{L^2(E)}=0$ for any compact subset $E$ of $\R^3$.

Note that both $e^{t\De}v_0$ and $e^{t\De}b_0$, with $v_0,b_0\in E^2_q$, solve the homogeneous heat equation in the distribution sense. 
Also, using $\nb\cdot v_0 = \nb\cdot b_0 = 0$, we can easily see that $\nb\cdot e^{t\De}v_0 = \nb\cdot e^{t\De}b_0 = 0$.

On the other hand, $\mathcal{J}_\ep(v^\ep), \mathcal{J}_\ep(b^\ep)\in L^\infty(\R^3\times[0,T])$ and $(v^\ep,b^\ep)\in{\bf LE}_q(0,T)$ imply
\[
\mathcal{J}_\ep(v^\ep)\otimes v^\ep\Phi_\ep,\quad
 \mathcal{J}_\ep(b^\ep)\otimes b^\ep\Phi_\ep,\quad 
  \mathcal{J}_\ep(v^\ep)\otimes b^\ep\Phi_\ep,\quad  
  \mathcal{J}_\ep(b^\ep)\otimes v^\ep\Phi_\ep\in L^\infty(0,T;L^2(\R^3)).
\]
Hence by the classical theory, $u^\ep = v^\ep - e^{t\De}v_0$ and $\pi^\ep$ defined by 
\[
\pi^\ep = (-\De)^{-1}\pd_i\pd_j\bkt{\bke{\mathcal{J}_\ep(v^\ep_i)v^\ep_j - \mathcal{J}_\ep(b^\ep_i)b^\ep_j}\Phi_\ep}
\in L^\infty(0,T;L^2(\R^3)),
\]
solves the Stokes system with the source term $\nb\cdot\bkt{\bke{\mathcal{J}_\ep(v^\ep)\otimes v^\ep - \mathcal{J}_\ep(b^\ep)\otimes b^\ep}\Phi_\ep}$ in the distribution sense.
Moreover, $a^\ep = b^\ep - e^{t\De}b_0$ solves the forced heat equation with the forcing  $\nb\cdot\bkt{\bke{\mathcal{J}_\ep(v^\ep)\otimes b^\ep - \mathcal{J}_\ep(b^\ep)\otimes v^\ep}\Phi_\ep}$ in the distribution sense.
By adding the homogeneous heat equation for $e^{t\De}v_0$ with $\nb\cdot e^{t\De}v_0 = 0$ and the Stokes system for $u^\ep$ and $\pi^\ep$, $v^\ep = e^{t\De}v_0 + u^\ep$ satisfies
\[
\pd_tv^\ep - \De v^\ep + \bke{\mathcal{J}_\ep(v^\ep)\cdot\nb}(v^\ep\Phi_\ep) - \bke{\mathcal{J}_\ep(b^\ep)\cdot\nb}(b^\ep\Phi_\ep) + \nb\pi^\ep = 0
\]
in the sense of distribution.
Moreover, by adding the homogeneous heat equation for $e^{t\De}b_0$ with $\nb\cdot e^{t\De}v_0 = 0$ and the forced heat equation for $a^\ep$, $b^\ep = e^{t\De}b_0 + a^\ep$ satisfies
\[
\pd_tb^\ep - \De b^\ep + \bke{\mathcal{J}_\ep(v^\ep)\cdot\nb}(b^\ep\Phi_\ep) - \bke{\mathcal{J}_\ep(b^\ep)\cdot\nb}(v^\ep\Phi_\ep)= 0
\]
in the sense of distribution. 
\end{proof}

We next show global existence for the localized and regularized MHD equations, \eqref{eq-4.10-BLT-MathAnn2024}.

\begin{lemma}\label{lem-4.8-BLT-MathAnn2024}
Assume $v_0,b_0\in E^2_q$, $1\le q<2$, are divergence free, and fix $\ep\in(0,1)$.
Then, there exists a global solution $(v^\ep,b^\ep,\pi^\ep)$ to the localized and regularized MHD equations, \eqref{eq-4.10-BLT-MathAnn2024}, such that $(v^\ep,b^\ep)\in {\bf LE}_q(0,T)$ for any $T<\infty$, and $(v^\ep,b^\ep,\pi^\ep)$ satisfies the a priori bounds \eqref{eq-3.5-BT-SIMA2021} and \eqref{eq-3.6-BT-SIMA2021-q<2} for all $R=n\in\NN$ up to time $T_n = \la_0n^2\min\bke{1,\, n^2(c_3\norm{(v_0,b_0)}_{E^2_q\times E^2_q}^2)^{-2}}$ for some $c_3>0$, where $\la_0$ is given in Lemma \ref{lem3.1-BT-SIMA2021}.
\end{lemma}

\begin{proof}
The proof of the lemma is an adaption of the proof of \cite[Lemma 4.8]{BLT-MathAnn2024} for the Navier--Stokes equations to the MHD equations.

We set the radius $R=n\in\NN$. 
By \eqref{eq-3.17-BT-SIMA2021}, we have the bound $A_{0,q}(n)\le c_3 (\norm{v_0}_{E^2_q}^2 + \norm{b_0}_{E^2_q}^2)$ for all $n\in\NN$.
Define
\[
T_n = \la_0n^2\min\bket{1, n^2\bke{c_3(\norm{v_0}_{E^2_q}^2 + \norm{b_0}_{E^2_q}^2)}^{-2}}
\le \la_n n^2,
\]
where the constants $\la_0$ and $\la_n$ are as in Lemma \ref{lem3.1-BT-SIMA2021}.
Note that $T_n$ is increasing and $T_n\to\infty$ as $n\to\infty$.
Now, by the same argument used in the proof of Lemma \ref{lem3.1-BT-SIMA2021},
if a solution $(v^\ep,b^\ep,\pi^\ep)$ of \eqref{eq-4.10-BLT-MathAnn2024} satisfies $(v^\ep,b^\ep)\in{\bf LE}_q(0,T)$, then it satisfies the a priori bounds \eqref{eq-3.5-BT-SIMA2021} and \eqref{eq-3.6-BT-SIMA2021-q<2} on the interval $[0,\min(T,T_n)]$ with radius $R=n$.

Since the system \eqref{eq-4.10-BLT-MathAnn2024} is a coupled system of inhomogeneous Stokes systems with localized and regularized forcing, standard theory guarantees the existence of unique global solution $(v^\ep,b^\ep,\pi^\ep)$.
By uniqueness, this solution agrees with the ${\bf LE}_q$-solution constructed in Lemma \ref{lem-4.7-BLT-MathAnn2024}, and thus $(v^\ep,b^\ep)\in{\bf LE}_q(0,\tau)$ for some $\tau=\tau(\ep,\norm{(v_0,b_0)}_{E^2_q\times E^2_q})>0$.
Fix $n\in\NN$.
Applying \eqref{eq-3.5-BT-SIMA2021} with $R=n$, we obtain $\norm{(v^\ep,b^\ep)(\tau)}_{E^2_q\times E^2_q}\le C(n) \norm{(v_0,b_0)}_{E^2_q\times E^2_q}$.
Then, by Lemma \ref{lem-4.7-BLT-MathAnn2024}, there exists an ${\bf LE}_q$-solution on $(\tau,\tau+\tau_1)$ for some $\tau_1 = \tau_1(\ep,C(n)\norm{(v_0,b_0)}_{E^2_q\times E^2_q})>0$.
By uniqueness, this solution coincides with $(v^\ep,b^\ep)$, and we conclude that $(v^\ep,b^\ep)\in{\bf LE}_q(0,\tau+\tau_1)$, with the a priori bound \eqref{eq-3.5-BT-SIMA2021} valid up to time $\tau+\tau_1$.
This argument can be iterated: by repeatedly extending the solution, we obtain $(v^\ep,b^\ep)\in {\bf LE}_q(0,\tau+k\tau_1)$ for $k\in\NN$, until $\tau+k\tau_1\ge T_n$.
Thus, for each $n\in\NN$, we obtain $(v^\ep,b^\ep)\in {\bf LE}_q(0,T_n)$, with the a priori bound \eqref{eq-3.5-BT-SIMA2021} holding for $R=n$ up to time $T_n$.
Since $T_n\to\infty$, the lemma follows.
\end{proof}

\begin{proof}[Proof of Theorem \ref{th4.8-mhd} for $1\le q<2$]
The proof of Theorem \ref{th4.8-mhd} for $1\le q<2$ is an adaption of the proof of \cite[Theorem 1.4]{BLT-MathAnn2024} for the Navier--Stokes equations to the MHD equations.
We provide the necessary details by following the same stragegy as in the proof of \cite[Theorem 1.5]{BT-CPDE2020} for the Navier--Stokes equations, adapting the argument from the $L^2_\uloc$ framework to the $E^2_q$ setting, and from the Navier--Stokes to the MHD equations.

For $k\in\NN$, let $(v^k,b^k,\pi^k)$ be the solution of the localized, regularized MHD equations, \eqref{eq-4.10-BLT-MathAnn2024}, with $\ep=1/k$, given in Lemma \ref{lem-4.8-BLT-MathAnn2024}.
They share the same a priori bound \eqref{eq-3.5-BT-SIMA2021} for $R=n$ up to time $T_n$, thus
\[
\sup_{k\in\NN} \norm{(v^k,b^k)}_{{\bf LE}_q(0,T_n)}<\infty,\qquad
\forall n\in\NN.
\]
Using this a priori bound, we now construct the desired global solutions as the limit of $v^k$ defined in $(0,T_k)$, $T_k\to\infty$ by induction.
Lemma \ref{lem-4.8-BLT-MathAnn2024} implies that $(v^\ep,b^\ep)$ are uniformly bounded in the class from inequalities 
\EQ{\label{eq-4.1-KS-AMST2007}
\sup_{0<t<T_1}\int_{B_1} (|v^k|^2 + |b^k|^2)\, dx
+ \int_0^{T_1}\int_{B_1} (|\nb v^k|^2 + |\nb b^k|^2)\, dxdt
\le cA
}
\EQ{\label{eq-4.2-KS-AMST2007}
\int_0^{T_1}\int_{B_1} (|v^k|^{10/3} + |b^k|^{10/3})\, dxdt \le cA^{5/3},
}
\EQ{\label{eq-4.3-KS-AMST2007}
\int_0^{T_1}\int_{B_1} |\pi^k(x,t) - c_{0,2}^k(t)|^{3/2}\, dxdt \le C(T_1,A),
}
where $c_{0,2}^k(t)$ is the function of $t$ in \eqref{eq-3.6-BT-SIMA2021} with $x_0=0$ and $R=2$,
and 
\EQ{\label{eq-4.4-KS-AMST2007}
\norm{\pd_tv^k}_{\chi_1} + \norm{\pd_tb^k}_{\chi_1} \le C(T_1,A),
}
where $\chi_1$ is the space dual to $L^3(0,T_1; W^{1,3}_0(B_1))$.
Hence there exists a sequence $(v^{1,k}, b^{1,k})$ (where the corresponding $\ep$ are denoted by $\ep_{1,k}$) that converges to a solution $(v_1,b_1)$ of \eqref{MHD} on $B_1\times(0,T_1)$ in the following sense:
\EQN{
(v^{1,k}, b^{1,k}) \overset{*}{\rightharpoonup} (v_1,b_1)&\ \text{ in }L^\infty(0,T_1; (L^2\times L^2)(B_1)),\\
(v^{1,k}, b^{1,k}) \rightharpoonup (v_1,b_1)&\ \text{ in }L^2(0,T_1; (H^1\times H^1)(B_1)),\\
(v^{1,k}, b^{1,k}) \rightarrow (v_1,b_1)&\ \text{ in }L^3(0,T_1; (L^3\times L^3)(B_1)),\\
(\mathcal{J}_{\ep_{1,k}}v^{1,k}, \mathcal{J}_{\ep_{1,k}}b^{1,k}) \rightarrow (v_1,b_1)&\ \text{ in }L^3(0,T_1; (L^3\times L^3)(B_1)).
}
By Lemma \ref{lem-4.8-BLT-MathAnn2024} all $v^{1,k}$ are also uniformly bounded on $B_n\times[0,T_n]$ for $n\in\NN$, $n\ge2$ and, recursively, we can extract subsequences $\{(v^{n,k}, b^{n,k})\}_{k\in\NN}$ from $\{(v^{n-1,k}, b^{n-1,k})\}_{k\in\NN}$ which converge to a solution $(v_n,b_n)$ of \eqref{MHD} on $B_n\times(0,T_n)$ as $k\to\infty$ in the following sense:
\EQN{
(v^{n,k}, b^{n,k}) \overset{*}{\rightharpoonup} (v_n,b_n)&\ \text{ in }L^\infty(0,T_n; (L^2\times L^2)(B_n)),\\
(v^{n,k}, b^{n,k}) \rightharpoonup (v_n,b_n)&\ \text{ in }L^2(0,T_n; (H^1\times H^1)(B_n)),\\
(v^{n,k}, b^{n,k}) \rightarrow (v_n,b_n)&\ \text{ in }L^3(0,T_n; (L^3\times L^3)(B_n)),\\
(\mathcal{J}_{\ep_{n,k}}v^{n,k}, \mathcal{J}_{\ep_{n,k}}b^{n,k}) \rightarrow (v_n,b_n)&\ \text{ in }L^3(0,T_n; (L^3\times L^3)(B_n)).
}
Let $(\td v_n, \td b_n)$ be the extension by $0$ of $(v_n,b_n)$ to $\R^3\times(0,\infty)$.
Note that, at each step, $(\td v_n, \td b_n)$ agrees with $(\td v_{n-1}, \td b_{n-1})$ on $B_{n-1}\times(0,T_{n-1})$.
Let $(v,b) = \lim_{n\to\infty} (\td v_n, \td b_n)$.
Then $(v,b) = (v_n, b_n)$ on $B_n\times(0,T_n)$ for every $n\in\NN$.

Let $(v^k,b^k) = (v^{k,k}, b^{k,k})$ on $B_k\times(0,T_k)$ and equal $0$ elsewhere.
Let $\ep_k$ denote the corresponding regularization parameter.
Then, for every fixed $n$ and as $k\to\infty$,
\EQS{\label{eq-3.6-BT-CPDE2020}
(v^k, b^k) \overset{*}{\rightharpoonup} (v,b)&\ \text{ in }L^\infty(0,T_n; (L^2\times L^2)(B_n)),\\
(v^k, b^k) \rightharpoonup (v,b)&\ \text{ in }L^2(0,T_n; (H^1\times H^1)(B_n)),\\
(v^k, b^k) \rightarrow (v,b)&\ \text{ in }L^3(0,T_n; (L^3\times L^3)(B_n)),\\
(\mathcal{J}_{\ep_k}v^k, \mathcal{J}_{\ep_k}b^k) \rightarrow (v,b)&\ \text{ in }L^3(0,T_n; (L^3\times L^3)(B_n)).
}
Based on the uniform bounds for the approximates, we have that $(v,b)$ satisfies \eqref{eq-thm1.4BT4}.

To resolve the pressure, let
\EQN{
\pi^k(x,t) &= -\frac13\bkt{\mathcal{J}_{\ep_k}(v^k)\cdot v^k(x,t)\Phi_{\ep_k}(x) - \mathcal{J}_{\ep_k}(b^k)\cdot b^k(x,t)\Phi_{\ep_k}(x)}\\
&\quad + \text{p.v.} \int_{B_2} K_{ij}(x-y)\bkt{\mathcal{J}_{\ep_k}(v^k_i) v^k_j(y,t) - \mathcal{J}_{\ep_k}(b^k_i) b^k_j(y,t)}\Phi_{\ep_k}(y)\, dy\\
&\quad + \text{p.v.} \int_{B_2^c} (K_{ij}(x-y)-K_{ij}(-y)) \bkt{\mathcal{J}_{\ep_k}(v^k_i) v^k_j(y,t) - \mathcal{J}_{\ep_k}(b^k_i) b^k_j(y,t)}\Phi_{\ep_k}(y)\, dy,
}
which, together with $(v^k,b^k) = (v^{\ep_k}, b^{\ep_k})$, solves \eqref{eq-4.10-BLT-MathAnn2024} with $\ep=\ep_k$ in the distributional sense.

From the convergence properties of $(v^k, b^k)$, it follows that $\pi^k\to\pi$ in $L^{3/2}(0,T_n;L^{3/2}(B_n))$ for all $n$ where $\pi(x,t) = \lim_{n\to\infty}\bar{\pi}^n(x,t)$ in which $\bar{\pi}^n(x,t)$ is defined for $|x|<2^n$ by
\[
\bar{\pi}^n(x,t) = -\frac13\bke{|v(x,t)|^2 - |b(x,t)|^2} + \text{p.v.}\int_{B_2} K_{ij}(x-y)\bke{v_iv_j-b_ib_j}(y,t)\, dy + \bar{\pi}^n_3 + \bar{\pi}^n_4,
\]
with
\EQN{
\bar{\pi}^n_3(x,t) &= \text{p.v.}\int_{B_{2^{n+1}}\setminus B_2} \bke{K_{ij}(x-y)-K_{ij}(-y)}\bke{v_iv_j-b_ib_j}(y,t)\, dy,\\
\bar{\pi}^n_4(x,t) &= \int_{B_{2^{n+1}}^c} \bke{K_{ij}(x-y)-K_{ij}(-y)}\bke{v_iv_j-b_ib_j}(y,t)\, dy.
}
We have $\bar{\pi}^n_3, \bar{\pi}^n_4\in L^{3/2}((0,T)\times B_{2^n})$ and
\[
\bar{\pi}^n_3 + \bar{\pi}^n_4 = \bar{\pi}^{n+1}_3 + \bar{\pi}^{n+1}_4\ \text{ in }L^{3/2}((0,T)\times B_{2^n}).
\]
Thus $\bar{\pi}^n$ is independent of $n$ for $n>\log_2|x|$.

We now establish the above local pressure expression for all scales.
Note that the formula is valid for $\pi^k$ at all scales, that is, for any $T>0$, fixed $R>0$ and $x_0\in\R^3$, we have the following equality in $L^{3/2}(B_{2R}(x_0)\times(0,T))$,
\EQN{
\hat{\pi}_{x_0,R}^k(x,t) &:= \pi^k(x,t) - c_{x_0,R}^k(t)\\
&= -\De^{-1} \div\div\bkt{\bke{(\mathcal{J}_kv^k\otimes v^k - \mathcal{J}_kb^k\otimes b^k)\Phi_k}\chi_{4R}(x-x_0)}\\
&\quad - \int_{\R^3} \bke{K(x-y) - K(x_0-y)} \bke{(\mathcal{J}_kv^k\otimes v^k - \mathcal{J}_kb^k\otimes b^k)\Phi_k }(y,t) \bke{1-\chi_{4R}(y-x_0)} dy,
}
where $\mathcal{J}_k = \mathcal{J}_{\ep_k}$ and $\Phi_k = \Phi_{\ep_k}$.
Similarly, let 
\EQN{
\hat{\pi}_{x_0,R}(x,t)
&= -\De^{-1} \div\div\bkt{(v\otimes v - b\otimes b)\chi_{4R}(x-x_0)}\\
&\quad - \int_{\R^3} \bke{K(x-y) - K(x_0-y)} (v\otimes v - b\otimes b)(y,t) \bke{1-\chi_{4R}(y-x_0)} dy.
}
Fix $T>0$, $x_0\in\R^3$ and $R>0$.
Choose $n$ large enough that $B_{8R}(x_0)\times(0,T)\subset Q_n=B_n\times(0,T_n)$.
We claim that $\hat{\pi}_{x_0,R}^k(x,t)$ converges to $\hat{p}_{x_0,R}(x,t)$ in $L^{3/2}(B_{2R}(x_0)\times(0,T))$.
If this is the case, by taking the limit of the weak form of \eqref{eq-4.10-BLT-MathAnn2024}, we can show that $(v,b,\hat{\pi}_{x_0,R})$ also satisfies \eqref{MHD} in $B_{2R}(x_0)\times(0,T)$.
Hence $\nb\pi - \nb\hat{\pi}_{x_0,R} = 0$, and we may define
\[
c_{x_0,R}(t) = \pi(x,t) - \hat{\pi}_{x_0,R}(x,t)
\]
which is hence a function of $t$ in $L^{3/2}(0,T)$ that is independent of $x$.
This gives the desired local pressure expansion in $B_{2R}(x_0)\times(0,T)$.

To verify the claim we work term by term.
Note that the estimate in \cite[(3.26)]{KT-CMP2020} shows that
\[
\norm{v_iv_j-(\mathcal{J}_kv^k_i)v^k_j\Phi_k}_{L^{3/2}(B_M\times[0,T_n])}, \qquad 
\norm{b_ib_j-(\mathcal{J}_kb^k_i)b^k_j\Phi_k}_{L^{3/2}(B_M\times[0,T_n])}\to0,
\]
as $k\to\infty$ for every $M>0$ and $n\in\NN$.
This implies
\EQN{
-\De^{-1} \div\div&\bkt{\bke{(\mathcal{J}_kv^k\otimes v^k - \mathcal{J}_kb^k\otimes b^k)\Phi_k}\chi_{4R}(x-x_0)}\\
&\qquad \rightarrow -\De^{-1} \div\div\bkt{(v\otimes v - b\otimes b)\chi_{4R}(x-x_0)}
}
in $L^{3/2}(B_{2R}(x_0)\times(0,T_n))$,
and
\EQN{
- \int_{|x|<M}& \bke{K(x-y) - K(x_0-y)} \bke{(\mathcal{J}_kv^k\otimes v^k - \mathcal{J}_kb^k\otimes b^k)\Phi_k }(y,t) \bke{1-\chi_{4R}(y-x_0)} dy\\
&\rightarrow - \int_{|x|<M} \bke{K(x-y) - K(x_0-y)} (v\otimes v - b\otimes b)(y,t) \bke{1-\chi_{4R}(y-x_0)} dy
}
in $L^{3/2}(B_{2R}(x_0)\times(0,T_n))$ for every $M>8R$.
For the far-field part, still assuming $M>8R$, we have
{\small
\EQN{
&\norm{\int_{|x|\ge M} \bke{K(x-y) - K(x_0-y)} \bke{(\mathcal{J}_kv^k\otimes v^k - \mathcal{J}_kb^k\otimes b^k)\Phi_k - (v\otimes v - b\otimes b)}(y,t) dy}_{L^{3/2}(B_{2R}(x_0)\times(0,T_n))}\\
&\qquad \le C(R,n,\norm{(v_0,b_0)}_{L^2_\uloc\times L^2_\uloc}) M^{-1}
\le C(R,n,\norm{(v_0,b_0)}_{E^2_q\times E^2_q}) M^{-1},
}}
where we've used the embedding $E^2_q\subset L^2_\uloc$.
This can be made arbitrarily small by taking $M$ large and noting $R$ and $n$ are fixed.
Consequently, and since the other parts of the pressure converge, we conclude that 
\EQ{\label{eq-4.6-KS-AMST2007}
\hat{\pi}_{x_0,R}^k(x,t) \rightarrow \hat{\pi}_{x_0,R}(x,t)\ \text{ in }L^{3/2}(B_{2R}(x_0)\times(0,T_n)),
}
which leads to the desired local pressure expansion.
Since $n$ was arbitrary, this gives the pressure formula for arbitrarily large times.

At this point we have established items 1.-3. from the definition of local energy solutions. We now check remaining items.

Fix $T_0$ and choose $n$ so that $T_n\ge T_0$.
Then \eqref{eq-3.6-BT-CPDE2020} holds for all $n$ with $T_n$ replaced by $T_0$.
Furthermore, the estimates \eqref{eq-4.1-KS-AMST2007}--\eqref{eq-4.4-KS-AMST2007} and \eqref{eq-4.6-KS-AMST2007} are valid in $B_n\times[0,T_0]$ up to a re-definition of $A$.
Moreover,  we have
\EQ{\label{eq-4.7-KS-AMST2007}
\norm{\pd_tv}_{\chi_n} + \norm{\pd_tb}_{\chi_n} \le C(n,T_0,A),
}
and 
\EQ{\label{eq-4.9-KS-AMST2007}
\bigg\|\esssup_{0\leq t \leq T_0} \int_{B_1(k) } \bke{|v|^2 + |b|^2} dx+ \int_0^{T_0}\int_{B_1(k) } \bke{|\nabla v|^2 + |\nabla b|^2} dx\,dt \bigg\| _{\ell^{q/2}(k\in \ZZ^3)}
\leq  2A.
}
It follows from \eqref{eq-4.7-KS-AMST2007} and \eqref{eq-4.9-KS-AMST2007} that for every $n$,
\EQ{\label{eq-4.10-KS-AMST2007}
t\mapsto \int_{B_n} v(x,t)\cdot w(x)\, dx\qquad
t\mapsto \int_{B_n} b(x,t)\cdot w(x)\, dx
}
are continuous in $t\in[0,T_0]$ for every $w\in L^2(B_2)$.
Since $T_0$ was arbitrary, we can extend this to all times.
The local energy inequality follows from the local energy equality for $(v^k,b^k)$ and $\pi^k$, and \eqref{eq-3.6-BT-CPDE2020}, \eqref{eq-4.6-KS-AMST2007} in $B_n\times[0,T_0]$, \eqref{eq-4.7-KS-AMST2007}, and $\hat{\pi}_n(x,t) = \pi(x,t) - c_n(t)$ for some $c_n\in L^{3/2}(0,T_0)$.
Convergence to initial data in $L^2_\loc$ follows from \eqref{eq-4.10-KS-AMST2007} and the local energy inequality. 
This confirms that items 4.-6. from the definition of local energy solutions are satisfied and finishes the proof of Theorem \ref{th4.8-mhd} for $1\le q<2$.
\end{proof}

\section*{Acknowledgments}
I warmly thank Zachary Bradshaw and Tai-Peng Tsai for helpful comments.
The research was partially support by the AMS-Simons Travel Grant and the Simons Foundation Math + X Investigator Award \#376319 (Michael I. Weinstein).
The author gratefully acknowledges the unwavering financial and emotional support of his wife, Anyi Bao.

\bigskip

\end{document}